\documentclass[preprint]{elsarticle}

\usepackage[french,english]{babel}
\usepackage[T1]{fontenc}
\usepackage{pstricks,amsmath,amsfonts,amssymb,amsthm,bbm,latexsym,graphicx,color}
\usepackage{amsmath,amssymb,exscale,epsfig}

\newtheorem{theoreme}{Theorem}

\newtheorem{definition}{Definition}
\newtheorem{proposition}{Proposition}
\newtheorem{corollaire}{Corollary}
\newtheorem{remarque}{Remark}

\newcommand{\T}{\ensuremath{{\mathcal{T}}}}

\newcommand{\B}{\ensuremath{{\mathcal{B}}}}

\newcommand{\Mdi}{\ensuremath{\mathcal{M}_\infty(\Rd)}}

\newcommand{\R}{\ensuremath{{\mathbb{R}}}}

\newcommand{\Rd}{\ensuremath{{\mathbb{R}^2}}}

\newcommand{\Md}{\ensuremath{\mathcal{M}(\Rd)}}
\newcommand{\C}{\ensuremath{\mathcal{C}}}

\newcommand{\del}{\ensuremath{\mathcal{D}el}}
\newcommand{\vor}{\ensuremath{\mathcal{V}or}}
\newcommand{\per}{\ensuremath{\mathcal{P}er}}
\newcommand{\vol}{\ensuremath{\mathcal{V}ol}}

\newcommand{\NN}{\ensuremath{{\mathbb{N}}}}

\newcommand{\vp}{\ensuremath{\varepsilon}}
\newcommand{\rem}{\ensuremath{\mathcal{R}}}

\newcommand{\g}{\ensuremath{\gamma}}

\newcommand{\1}{\ensuremath{\mbox{\rm 1\kern-0.23em I}}}
\renewcommand{\L}{\ensuremath{\Lambda}}
\renewcommand{\l}{\ensuremath{\lambda}}

\newcommand{\Bd}{\ensuremath{{\B(\Rd)}}}

\begin{document}

\begin{frontmatter}


\title{Practical simulation and estimation for Gibbs Delaunay-Voronoi tessellations with geometric hardcore interaction.}

\author{D. Dereudre\fnref{fn1}} 
\ead{david.dereudre@univ-valenciennes.fr}

\author{F. Lavancier\fnref{fn2}} 
\ead{frederic.lavancier@univ-nantes.fr}

\fntext[fn1]{ Universit\'e Lille Nord de France,
 F\'ed\'eration CNRS 2956,
 UVHC, LAMAV, F-59313 Valenciennes Cedex 09, France.
}

\fntext[fn2]{Universit\'e de Nantes,
 Laboratoire de Math\'ematiques Jean Leray,
 Unit\'e Mixte de Recherche CNRS 6629,
 UFR Sciences et Techniques,
 2 rue de la Houssini\`ere - BP 92208 -
 F-44322 Nantes Cedex, France.}

\begin{abstract} 

General models of Gibbs Delaunay-Voronoi tessellations, which can be viewed as extensions of Ord's process, are considered. The interaction may occur on each cell of the tessellation and between neighbour cells. The tessellation may also be subjected to a geometric hardcore interaction, forcing the cells not to be too large, too small, or too flat. This setting, natural for applications, introduces some theoretical difficulties since the interaction is not necessarily hereditary. Mathematical results available for studying these models are reviewed and further outcomes are provided. They concern the existence, the simulation and the estimation of such tessellations. Based on these results, tools to handle these objects in practice are presented: how to simulate them, estimate their parameters and validate the fitted model. Some examples of simulated tessellations are studied in details.
\end{abstract}


\end{frontmatter}

\section{Introduction}

In the domains of physics and biology, some large-scale random geometric structures can be mathematically modeled using Poisson-Voronoi or Poisson-Delaunay tessellations. In cosmology for instance, since \cite{MS}, modeling the large-scale galaxy distribution generally relies on Voronoi tessellations (see \cite{IW1} and \cite{IW2}). In biology, Voronoi tessellations are often used to model the cellular configuration of a tissue (since the seminal work of \cite{H}). This tool is also relevant to model the geometrical structure of proteins (cf. \cite{P} for a state of the art) or microstructures like foams. Mathematical properties of Poisson-Voronoi and Poisson-Delaunay tessellations have been widely studied (see  \cite{M} for instance).

Unfortunately, these models have the disadvantage of yielding strong independence properties due to the Poissonian nature of the underlying point process. In different biological studies, the necessity to introduce an interaction between the cells of the tessellation  to achieve greater realism has indeed been emphasized. In \cite{FRAEJ} for instance, the interaction between neighbouring epithelial cells is dealt with using a Hamiltonian energy. This Hamiltonian is a function of the area of each cell of the Voronoi tessellation, but it involves also a pair-interaction that depends on the length of the common edge of two cells. The same kind of interaction  (but between two types of cells) is also considered in \cite{EF}. Moreover, some geometric hardcore interactions are sometimes demanded. As an example, Lautensack and Sych (\cite{LS}) modeled foams by a tessellation built from a Matern model with hardcore interaction. The resulting tessellation is then constrained to reach a desired regularity. The study of the regularity of the tessellation is also at the heart of the article of Eglen and Willshaw (\cite{EW}): their work shows the relevance of forcing the geometry of cells in order to model retinal neurons. 

It is thus natural to consider Gibbsian modifications of the Poisson-Voronoi or Poisson-Delaunay tessellation, involving a smooth interaction but also a hardcore interaction (in a general sense, see Definition \ref{hardcore}), in order to produce more realistic models of interacting random structures.

A first mathematical model has been proposed by Ord (see the discussion in  \cite{Ripley}).  In this model, the interaction relies on each cell of the Voronoi tessellation. In particular, a classical hardcore interaction forces the cells not to be too small. This model can be viewed as a nearest neighbour Gibbs point process and was studied in  \cite{BM1}. Its existence on the infinite support $\R^d$ is implied by the results in \cite{BBD} and \cite{BBD1}. A Birth-Death simulation algorithm for simulating such nearest neighbour Gibbs processes is presented in \cite{BBD2}. However, tessellations involving geometric hardcore interactions do not generally belong to a classical theoretical framework as the previous one. They are in general not hereditary in the sense that,  when removing a point from an allowed tessellation, the resulting tessellation may become forbidden (see Section \ref{estimation}, or \cite{DL}, about this property).  Consider for instance a generalization of the Ord process where the cells are forced not to be too large: this natural model is not hereditary.  The existence of a Gibbs Delaunay-Voronoi tessellation on $\R^d$ associated to a large class of possible non-hereditary interactions has been proved   recently in \cite{D} and \cite{DDG}. For these processes, no simulation algorithm has been presented so far.

From a statistical point of view, the issue is the estimation of the interaction.  Assuming a parametric form, this may be achieved through the maximum likelihood or the pseudo-likelihood procedure. The maximum likelihood estimator suffers from a lack of theoretical justifications, except for restrictive examples of interacting point processes (see \cite{MW} for a review), which do not concern tessellation models. On the other hand, some theoretical results are available for the pseudo-likelihood estimator.  Consistency and asymptotic normality are proved in \cite{BCD} in a general framework including some Gibbs tessellation models, but without any hardcore interaction. A generalization to interactions involving a possible non-hereditary hardcore part is considered in \cite{DL}. In this article, the consistency of the estimation of both the hardcore part and the smooth part of the interaction is proved, in a setting concerning  a large class of tessellation models. 

In the present article, we  rewrite these theoretical results, sometimes established in an abstract setting, to the framework of Gibbs Delaunay-Voronoi tessellations. Moreover, some theoretical complements are given. In particular, a Birth-Death-Move algorithm is presented to simulate Gibbs tessellations with non-hereditary interactions, and a convergence result is proved. We also extend the recent concept of residuals introduced in \cite{BTMH} to the non-hereditary setting.  Nevertheless, the aim of this article is mainly to clarify how to handle (non-hereditary) Gibbs Delaunay-Voronoi  tessellations in practice: what kinds of models are available? How should we simulate these tessellations? How should we fit them to a data set and validate the fitted model? 

In the first part, the formal definition of Gibbs Delaunay-Voronoi tessellations is given. We restrict ourselves to tessellations on the plane $\R^2$ for simplicity. Three example models are then considered: a non stationary crystallized triangulation model, a stationary interacting Delaunay model and a Voronoi tessellation model. We think that these models could be relevant for the biological applications cited before. Moreover, they are used throughout the article to illustrate the proposed methods. In Section \ref{simulations}, we explain how to simulate (non-hereditary) Gibbs Delaunay-Voronoi tessellations thanks to a Birth-Death-Move Metropolis-Hastings algorithm. Some simulations of the three above examples are presented. In Section \ref{estimation}, we consider the estimation issue. As explained there, the pseudo-likelihood approach is preferred to the maximum likelihood procedure for practical reasons. As a matter of fact,  the maximum likelihood estimator is prohibitively time-consuming in our setting. However, if possible, maximum likelihood could be used in a second step to refine the pseudo-likelihood estimation. A procedure is presented to estimate both the hardcore parameters and the interaction, as considered in \cite{DL}. Finally, the concept of residuals as recently introduced in \cite{BTMH} is generalized, which gives a method to validate the fitted model.
In the appendix, we present some theoretical justifications. They concern the existence of Delaunay-Voronoi tessellations, the convergence of the simulation algorithm, and the consistency of the estimation procedure.

\section{The Gibbs Delaunay-Voronoi tessellations model.}
\subsection{The Poisson Delaunay-Voronoi tessellations.}

In paragraph \ref{point}, we recall the basic definition of point configurations.  In \ref{vordel}, some regularity assumptions are given to ensure that the Delaunay-Voronoi tessellations are well-defined. Randomness is introduced in paragraph \ref{poissonvordel}, via Poisson point processes, to define the well-known Poisson Delaunay-Voronoi tessellations which are models of random tessellations without interaction between the cells. 
The interaction is introduced in Section \ref{theorie}.
  
\subsubsection{Point configurations.}\label{point}

Let us denote by $\Rd$ the $2$-dimensional Euclidean real space. $\B(\Rd)$ is the set of bounded Borel sets in $\Rd$. The state space $\Md$ is the set of regular locally finite point configurations $\g$  in $\Rd$ defined by 
\begin{equation}\label{statespace}
 \Md = \left\{\g \subset \Rd \text{ such that } 
\begin{array}{l}
\text{-a) for all } \L \text{ in } \B(\Rd), \text{ Card}(\g\cap\L)<+\infty \\
\text{-b) four points of } \g \text{ are not on a same circle} \\
\text{-c) for every half plane }H \text{ in } \Rd, \text{ Card } (\g \cap H)>0

\end{array}
\right\},
\end{equation}
where $\text{Card} (\g\cap \Delta)$ denotes the number of points from $\g$ in the set $\Delta$. Let $\g$ be in $\Md$ and $\L$ a Borel set in $\Rd$, we denote by $\g_\L$ the restriction of $\g$ on $\L$ which is just the set $\g\cap\L$. For  a point $x$ in $\Rd$, we denote by $\g+x$ the configuration $\g \cup \{x\}$ and if $x$ belongs to $\g$, $\g-x$ denotes the set $\g \backslash\{x\}$.

\subsubsection{Delaunay-Voronoi tessellations.}\label{vordel}

Let us recall the definition of Delaunay-Voronoi tessellations, which are given for example in \cite{M} page 15. For a point configuration $\g$ in $\Md$, a set of three points $T=\{x,y,z\}$  belonging to $\g$ is a Delaunay triangle in $\g$ if the open circumscribed ball $\B(T)$ of $T$ does not contain any point of $\g$. The Delaunay tessellation $\del(\g)$ is defined by the set of all Delaunay triangles $T$ in $\g$. By points b) and c) in (\ref{statespace}), $\del(\g)$ is a partition of the space $\Rd$.

Concerning the Voronoi tessellation coming from $\g$, for every $x$ in $\g$, we define the Voronoi cell $C(x,\g)$ by
$$ C(x,\g)=\Big\{ z\in\Rd \text{ such that }\quad \forall y\in\g\backslash\{x\} \quad |z-x|\le |z-y| \Big\}.$$
From points a) and c) in (\ref{statespace}), we remark that $C(x,\g)$ is a bounded closed convex set in $\Rd$.  The Voronoi tessellation $\vor(\g)$ is defined by the set of all $ C(x,\g)$ for $x$ in $\g$. $\vor(\g)$ is also a partition of the plane $\Rd$. 

There are some relations between these two tessellations. Indeed, $T=\{x,y,z\}$ in $\g$ is a Delaunay triangle if and only if $ C(x,\g) \cap C(y,\g) \cap C(z,\g) \neq \emptyset$. 

\subsubsection{Poisson Delaunay-Voronoi tessellations.}\label{poissonvordel}

In this paragraph we define the Poisson Delaunay-Voronoi tessellations as in \cite{M} or \cite{SKM}. Let us recall that the space of point configurations $\Md$ is endowed with the $\sigma$-algebra $\sigma(\Md)$ generated by the sets $\{\g \in \Md, N_{\Lambda}(\g)=n \}$, $n\in \NN$, $\Lambda\in \B(\Rd)$, where $N_{\Lambda}(\g)$ denotes the number of points of $\g$ in $\Lambda$. The most prominent probability measures on $\Md$ are the Poisson processes. Let us denote them by $\pi^\nu$, where $\nu$ is a locally finite measure on $\Rd$ and stands for the intensity measure (see \cite{M} page 83). When $\nu$ is equal to $z\l$ ($z>0$, $\l$ the Lebesgue measure) we simply write $\pi^z$ which represents the classical stationary Poisson Point Process with intensity $z$. Let us remark that $\pi^\nu$ is not necessary stationary but obviously $\pi^z$ is. 

For every $\L$ in $\B(\Rd)$, $\pi^\nu_\L$ (respectively $\pi^z_\L$) denotes the Poisson process $\pi^\nu$ (respectively $\pi^z$) restricted on $\L$.

The law of $\del(\g)$ (respectively $\vor(\g)$) under the process $\pi^\nu$ is called the Poisson Delaunay (respectively Voronoi) tessellation with intensity $\nu$. These Poisson Delaunay-Voronoi tessellations are well-studied in \cite{M}, Section 4.

\subsection{Random Delaunay-Voronoi tessellations with interaction.}\label{theorie}

This section is devoted to the presentation of interacting random Delaunay-Voronoi tessellations. The interaction is introduced, via Gibbs modifications of the Poisson Delaunay-Voronoi tessellations, by specifying the conditional densities. This is the classical strategy used in physics and biology (see for example \cite{R}). 

For every $\L$ in $\B(\Rd)$, we consider the conditional density $f_\L$ with respect to the Poisson process $\pi^\nu_\L$ defined by    
\begin{equation}\label{density}
f_\L(\g_\L,\g_{\L^c}) = \frac{1}{Z_\L(\g_{\L^c})} e^{-E_\L(\g_\L,\g_{\L^c})},
\end{equation}
where $\g_\L$ is a point configuration inside $\L$, $\g_{\L^c}$ is a point configuration outside $\L$ and $E_\L(\g_\L,\g_{\L^c})$ is the energy of $\g_\L$ given the outside configuration $\g_{\L^c}$. $E_\L$ is a functional from $\vor(\g)$ or $\del(\g)$ to $\R\cup\{+\infty\}$ which we will precise later. $Z_\L(\g_{\L^c}):= \int e^{-E_\L(\g_\L,\g_{\L^c})} \pi^\nu_\L(d\g_\L)$ is the normalization constant in order to have a probability density under $\pi^\nu_\L$.

Let us remark that the conditional densities favor the point configurations with low energy and conversely penalize the point configurations with high energy. If the energy $E_\L(\g_\L,\g_{\L^c})$ is equal to infinity then, with probability one, the configuration is even forbidden and does not appear.
One classical example of such a situation is when the points of $\g$ are prohibited from being closer than a distance $R$ apart,  id est  $E_\L(\g_\L,\g_{\L^c})=+\infty$ if there exist $x$ in $\g_{\L}$ and $y$ in $\g$ such that $|x-y|\le R$. This constraint is usually designated as a hardcore interaction. In this paper, we generally call hardcore interaction any situation where $E_\L(\g_\L,\g_{\L^c})=+\infty$.
 \begin{definition}\label{hardcore}
An energy (or an interaction) is said to contain a hardcore part if  $E_\L(\g_\L,\g_{\L^c})=+\infty$ for some $\g$ and some $\L$.
\end{definition}

We denote by $\Mdi$ the set of allowed configurations which is defined by
\begin{equation} \label{Mdi}
\Mdi= \Big\{ \g\in\Md \text{ such that for all } \L \text{ in } \B(\Rd), \quad E_\L(\g_\L,\g_{\L^c})<+\infty \Big\}.
\end{equation} 
Now let us define the model of random Delaunay-Voronoi tessellations with interaction.
\begin{definition}\label{defgibbs}
A probability measure $P$ on $\Mdi$ is a Gibbs Delaunay-Voronoi tessellation for the energies $(E_\L)_{\L\in\B(\Rd)}$ and the intensity measure $\nu$ if for every $\L$ in $\B(\Rd)$ and for $P$-almost every outside configuration $\g_{\L^c}$, the law of $P$ given $\g_{\L^c}$ is absolutely continuous with respect to $\pi^\nu$ with the density $ f_\L(.,\g_{\L^c})$.
\end{definition}
This definition of Gibbs measures is the classical one that can be found for example in \cite{Pr}.\\
Let us point out several problems about the existence of these Gibbs Delaunay-Voronoi tessellations. First of all there are some conditions on the energies $(E_\L)_{\L\in\B(\Rd)}$ to ensure that the conditional densities $(f_\L)_{\L\in\B(\Rd)}$ are well-defined and compatible. Moreover, even if it is the case, it is not obvious that Gibbs Delaunay-Voronoi tessellations exist and are unique (the non unicity of Gibbs processes is called phase transition in statistical mechanics). In Section 5.1, we give some conditions to ensure the existence of stationary Gibbs Delaunay-Voronoi tessellations for energy functions given by (\ref{del}) and (\ref{vor}) below. As far as we know, uniqueness or phase transition for such Gibbs measures have not been proved. For similar Gibbs models with multi-type particles, a phase transition result is given in \cite{BBD3}.  

Since we are interested by models of random Delaunay-Voronoi tessellations, the energy functions have to depend on the local geometry of the Delaunay or Voronoi tessellations. We need some notations. Let us first  define the neighbour
relations $\sim_{\del}$, $\sim_{\vor}$ between the cells in $\del(\g)$ or $\vor(\g)$. 
\begin{align*}
&\text{ For all }\quad T, T' \in \del(\g), \qquad  T\sim_{\del} T' \quad\text { if }\quad  \text{ Card} (T \cap T') \ge 2.\\
&\text{ For all }\quad C, C' \in \vor(\g), \qquad C\sim_{\vor} C' \quad\text { if }\quad C\cap C'\neq \emptyset.
\end{align*}
In fact, $T\sim_{\del} T'$ if $T$ and $T'$ have a common edge and $C\sim_{\vor} C'$ if $C$ and $C'$ have a common edge at their boundary.

Now let us define the cells in $\del(\g)$ or $\vor(\g)$ which are inside or outside a given bounded set $\L$ in $\B(\Rd)$. A triangle $T\in \del(\g)$ (respectively a cell $C\in \vor(\g)$) is outside $\L$ if for every configuration $\g'_\L$ in $\L$,  $T$ (respectively $C$) is in $\del(\g_{\L^c} \cup \g'_\L)$ (respectively $\vor(\g_{\L^c} \cup \g'_\L)$). In other words, $T$ (or $C$) is outside  $\L$ if $T$ (or $C$) remains in $\del(\g)$ (or $\vor(\g)$) for any modification of the configuration $\g$ inside $\L$. $T$ (or $C$) is inside $\L$ if it is not outside
$\L$. We denote by $\del_\L(\g)$ (respectively $\vor_\L(\g)$) the cells $T$ in $\del(\g)$ (respectively  $C$ in $\vor(\g)$)
which are inside $\L$.\\

{\it a) A general form for the energy of a Delaunay tessellation}\\
We define the energy $E_\L$ of the Delaunay tessellation by
\begin{equation}\label{del}
E_\L(\g_\L,\g_{\L^c}) = \sum_{T\in \del_\L(\g)} V_1(T) + \sum_{
\begin{subarray}{l}
\{T,T'\} \subset  \del(\g) \\
T\sim_{\del}T' \\
T \text { or } T' \text{ in } \del_\L(\g)
\end{subarray}
}  V_2(T,T'),
\end{equation}
where $V_1$ is a function from the space of triangles $\T$ to $\R\cup\{+\infty\}$ and $V_2$ is a symmetric function from $\T^2$ to $\R\cup\{+\infty\}$. In Section 2.3, we give precise examples of functions $V_1$ and $V_2$.  \\

{\it b) A general form for the energy of a Voronoi tessellation}\\
Similarly, we define the energy $E_\L$ of the Voronoi tessellation by
\begin{equation}\label{vor}
E_\L(\g_\L,\g_{\L^c}) = \sum_{C\in \vor_\L(\g)} V_1(C) + \sum_{
\begin{subarray}{l}
\{C,C'\} \subset  \vor(\g) \\
C\sim_{\vor}C' \\
C \text { or } C' \text{ in } \vor_\L(\g)
\end{subarray}
}  V_2(C,C'),
\end{equation}
where $V_1$ is a function from the space of bounded convex sets $\C$ to $\R\cup\{+\infty\}$ and $V_2$ is a symmetric function from $\C^2$ to $\R\cup\{+\infty\}$. In Section 2.3,  a precise example of functions $V_1$ and $V_2$ is provided.

\subsection{Three explicit reference models}\label{troisexemples}

We present in this section three explicit examples of Gibbs Delaunay-Voronoi models, that will be our reference models until the end of the paper. All the functions $V_1$ and $V_2$ given in this section satisfy the assumptions in Section 5.1 and so the associated Gibbs Delaunay-Voronoi tessellations exist. \\

{\it Model 1: a non stationary crystallized triangulations model.}\\
In this model, we propose to define a non stationary random triangulation in which the angles of triangles are forced to be larger than a fixed real $\alpha$ in $[0,\frac{\pi}{3}[$. If $\alpha$ is chosen close to $\frac{\pi}{3}$, the model produces rigid random triangulations. Moreover, it is possible to have a non stationary density of points.\\
We assume  that the intensity measure $\nu$ is absolutely continuous with respect to the Lebesgue measure $\l$ and the energy $E_\L$ is defined by (\ref{del}) with
$$ V_1(T)= \left\{ \begin{array}{cl}
 +\infty & \text{ if } \alpha(T)\le \alpha, \\
 0 & \text{ otherwise, }
 \end{array} \right.\qquad \text{and} \qquad  V_2=0,$$
where $\alpha(T)$ is the minimal angle inside $T$. In fact, the energy $E_\L$ is equal to plus infinity if there exists a triangle inside $\L$ which is too flat. \\

{\it Model 2: a stationary interacting Delaunay model.}

In this model, we propose to study an example of stationary Delaunay triangulation with interaction. It is a simple model in order to present different practical and theoretical aspects in this work (modeling, simulation, estimation). This model has not the ambition to be directly usable in physics or biology.

First of all, we fix the intensity measure $\nu$ to be equal to $z\lambda$. Via a geometric hardcore interaction, we force the edges not to be too small, the triangles not to be too large and via a smooth interaction the large perimeters of triangles are favored or penalized (depending on the sign of $\theta$). More precisely, let $0<\vp<\alpha$ and $\theta$ be in $\R$, the energy $E_\L$ is defined by (\ref{del}) with 
$$ V_1(T)=\left\{ \begin{array}{cl}
 +\infty & \text{ if } l(T)\le \vp, \\
 +\infty & \text{ if } R(T)\ge \alpha,\\
 \theta \per(T) & \text{ otherwise, }
 \end{array} \right.\qquad \text{and} \qquad V_2=0,$$
where $l(T)$ is the minimal length of sides of $T$, $R(T)$ is the radius of the circumscribed ball of $T$ and $\per(T)$ is the perimeter of $T$. \\

{\it Model 3: a Voronoi tessellation model.}

In this third example, the interaction is defined as far as possible to fit with the biological applications evoked in the introduction (geometric regularities of cells, interaction between cells), although other interactions could be chosen. We suppose that the model is stationary so we fix the intensity $\nu$ to be equal to $z\lambda$. The geometry of cells is controlled by a hardcore interaction $V_1$ which forces the cells not to be too small,  too large or  too flat. Moreover, we add a smooth interaction involving a competition between the volumes of neighbours cells.\\ 
Let $ 0<\vp<\alpha$,  $B>1/(2\sqrt 3)$ and $\theta$ be in $\R$. The energy $E_\L$  is defined by (\ref{vor}) with 
$$ V_1: C \mapsto V_1(C)=\left\{ \begin{array}{cl}
 +\infty & \text{ if } h_{\min} (C) \le \vp, \\
 +\infty & \text{ if } h_{\max} (C) \ge \alpha,\\
 +\infty & \text{ if } h_{\max}^2 (C) \ge B\vol(C),\\
 0 & \text{ otherwise, }
 \end{array} \right.$$
 and
$$ V_2: (C,C')\mapsto V_2(C,C')= \theta \left(\frac{\max(\vol(C),\vol(C'))}{\min(\vol(C),\vol(C'))}-1\right)^{\frac{1}{2}},$$ 
where $h_{\min}(C)$ is the minimal distance between the center $x$ of the cell $C$ and the edges of the boundary of $C$. Similarly, $h_{\max}(C)$ is the maximal distance between $x$ and the edges of $C$ (see Figure \ref{cell}).

  \begin{figure}[htbp]
  \begin{center}
  \begin{picture}(200,150) 
  \put(0,5){\includegraphics[scale=0.3]{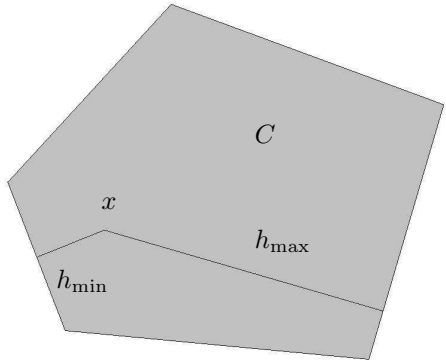}}
  \put(42,65){$x$}
  \put(25,35){$h_{\min}$}
  \put(100,50){$h_{\max}$}
  \put(100,90){$C$}
  \end{picture}

  \caption{{\small Example of Voronoi cell $C$ with center $x$ and distances $h_{\min}$, $h_{\max}$.}}\label{cell}
  \end{center}
  \end{figure}

The choice of the power $\frac{1}{2}$ is arbitrary and may be changed. Nevertheless, it seems to lead to more realistic simulations. The parameter $B$ controls the form of the cell: the smaller B, the more regular the cell. For instance, for a hexagonal cell, $B=1/(2\sqrt 3)\approx 0.29$. Let us remark that if $\theta$ is positive the interaction $V_2$ forces the neighbour cells to have the same volume. Conversely, if $\theta$ is negative it forces the neighbour cells to have different volumes. The sign of $\theta$ is crucial in this model.


\section{Simulations}\label{simulations}

\subsection{Gibbs tessellations on a finite window}
In this section, we deal with the simulation of our models. First of all, let us remark that Gibbs Delaunay-Voronoi tessellations are processes on $\Rd$ and so one has to restrict or approximate them inside a fixed window $\L$. The most natural choice would be to simulate the restriction on $\L$ but it is almost impossible to do it. Therefore the common method is to consider the finite volume Gibbs approximations on $\L$, which is the probability measure absolutely continuous with respect to the Poisson process $\pi_\L^\nu$ with the density $f_\L$ given in (\ref{density}). The outside point configuration $\g_{\L^c}$ is then fixed and chosen arbitrarily. Results in statistical mechanics show, in general, that the thermodynamic limits of these finite volume Gibbs measures, when the volume $\L$ goes to $\Rd$, are Gibbs measures (see \cite{Pr}). Therefore, finite volume Gibbs measures are good approximations of our models.

There are essentially three possibilities to fix the outside configurations. The first possibility is to consider the empty outside configuration but it is not usable in our context because it produces non bounded Delaunay-Voronoi cells and so non computable energies. The second possibility is to fix an outside point configuration which has to be specified explicitly. Again, this is not practical because the strong hardcore interaction, which appears in our three reference models, makes it difficult to find such a configuration. The third possibility is the periodic outside configuration which is built by periodization in the full plane $\Rd$ of the random configuration inside $\L$. We choose this approach because it seems relevant to deal with the hardcore problems coming from the boundary effects.   

Now, let us give the precise construction of the periodic finite volume Gibbs measure. The simulation window $\L$ is chosen as the square $[0,1]^2$. Any other window in $\Rd$ can be considered in the same way. A simple rescaling procedure enables us to reduce to this case. For every point configuration $\g$ in $[0,1]^2$, we denote by $\bar \g$ the periodic configuration on $\R^2$ defined by 
\begin{equation}\label{periodise}\bar \g= \bigcup_{i\in\mathbb Z^2} \tau_i(\g),\end{equation}
where $\tau_i$ is the translation in $\Rd$ with respect to the vector $i$. 

In the sums over the cells in (\ref{del})  or (\ref{vor}), we must ensure that each collection of periodic cells has a unique contribution in the computation of the periodic energies. A solution consists in selecting only the cells whom barycenters are in $[0,1]^2$. Therefore, for any Voronoi cells $C$ (respectively Delaunay triangle $T$), we denote by $<C>$ (respectively $<T>$) the barycenter of the cell $C$ (respectively triangle $T$). Similarly, for any couple of Voronoi cells $(C,C')$ (respectively triangles $(T,T')$) we denote by $<C,C'>$ (respectively $<T,T'>$) the barycenter of the set $C\cup C'$ (respectively $T\cup T'$) .

The periodic energy $\bar E(\g)$ associated to the energy of a Delaunay tessellation (\ref{del}) is defined by 
\begin{equation}\label{delper}
\bar E(\g) = \sum_{\begin{subarray}{c}
T\in \del(\bar\g)\\
<T> \in [0,1]^2
\end{subarray}
} V_1(T) + \sum_{
\begin{subarray}{c}
\{T,T'\} \subset  \del(\bar\g) \\
C\sim_{\del}C' \\
<T,T'> \in [0,1]^2
\end{subarray}
}  V_2(T,T').
\end{equation}

Similarly the periodic energy $\bar E(\g)$ associated to the energy of a Voronoi tessellation (\ref{vor}) is defined by 
\begin{equation}\label{vorper}
\bar E(\g) = \sum_{\begin{subarray}{c}
C\in \vor(\bar\g)\\
<C> \in [0,1]^2
\end{subarray}
} V_1(C) + \sum_{
\begin{subarray}{c}
\{C,C'\} \subset  \vor(\bar\g) \\
C\sim_{\vor}C' \\
<C,C'> \in [0,1]^2
\end{subarray}
}  V_2(C,C').
\end{equation}

Now let us define the Gibbs process on $[0,1]^2$ with periodic outside configuration. 

\begin{definition}\label{Pbar}
The periodic Gibbs Delaunay-Voronoi tessellation $\bar P$ is the point process on $[0,1]^2$ which is absolutely continuous with respect to the Poisson point process on $[0,1]^2$ (denoted by $\pi^\nu_0$), where the density $\bar f$ is defined, for every $\g$ in $[0,1]^2$, by
$$ \bar f(\g) = \frac{1}{\bar Z} e^{-\bar E(\g)} \quad \text{and} \quad \bar Z= \int e^{-\bar E(\g')}\pi^\nu_0 (d\g').$$
\end{definition}

\begin{remarque}
In the case where the intensity measure and the energy functions are stationary, we know there exists at least one stationary Gibbs measure (see Theorems 1 and 2 in the appendix). However, non stationary Gibbs measures may exist too. This phenomenon is called 
the breakdown of symmetry (see \cite{Pr} 4.1, for instance). It is not proved theoretically that symmetry breakdown can occur for the models studied in this paper. Sometimes, this phenomenon can be observed via simulations. Nevertheless let us remark that this will not be possible with the periodic Gibbs models simulated here. Indeed, if the intensity measure and the energy functions are stationary, as in models 2 and 3, then the associated periodic Gibbs models are stationary too and the symmetry breakdown is not observable. Free or configurational boundary conditions should be investigated.   
\end{remarque}

\subsection{Algorithm of simulation}\label{algorithm}

As explained above, to deal with the boundary effects, we choose to simulate periodic tessellations. For sake of brevity, we confuse $\bar \g$ and $\g$ in this section, similarly we will use $f$ instead of $\bar f$. The window of simulation is $\L=[0,1]^2$ and we omit the indexation by $\L$ in the sequel. Moreover, the notation $\Mdi$ is abusively extended in this section to the periodic tessellations with finite energies.

There exist different algorithms to simulate finite volume Gibbs processes. Some perfect simulation algorithms have been developed  (see \cite{GT}, \cite{PW}) but they seem not to be really implementable in our context due to the strong rigidity of our hardcore models. So we make the choice to simulate them via the classical Birth-Death-Move Metropolis-Hastings algorithm, which we recall below (see also \cite{MW}).

For $x\in [0,1]^2$, $\mathcal N (x,\sigma^2)$ denotes the Gaussian distribution on $\R^2$ centered at $x$ with covariance matrix $diag(\sigma^2,\sigma^2)$, where $\sigma>0$. This law is the proposal density for moving a point. Note that if the moved point falls outside $[0,1]^2$ (in step 5. below), it is replaced inside $[0,1]^2$ by the periodic property. We assume that the intensity measure $\nu$ is absolutely continuous with respect to the Lebesgue measure, and we denote by $g$ its density. In particular, in the stationary case, i.e. when $\nu=z\lambda$ with $z>0$, then $g$ is identically equal to $z$.

\begin{enumerate}
\item For $\g_0\in\Mdi$, let $n=card(\g_0)$.

\item Draw independently $a$ and $b$ uniformly on $[0,1]$.

\item If $a<1/3$ then generate $x$ uniformly on $[0,1]^2$ and set 
\begin{equation*}
\g_1=\begin{cases} \g_0+x & \textrm{if }\ b< \frac{f(\g_0+x)g(x)}{(n+1)f(\g_0)},\\
\g_0 &\textrm{otherwise.}
\end{cases}
\end{equation*}

\item If $a>2/3$ then generate $x$ uniformly on $\g_0$ and set 
\begin{equation*}
\g_1=\begin{cases} \g_0-x & \textrm{if }\ b< \frac{n f(\g_0-x)}{f(\g_0)g(x)},\\
\g_0 &\textrm{otherwise.}
\end{cases}
\end{equation*}

\item If $1/3<a<2/3$ then generate $x$ uniformly on $\g_0$, generate $y\sim\mathcal N(x,\sigma^2)$ and set
\begin{equation*}
\g_1=\begin{cases} \g_0-x+y & \textrm{if }\ b< \frac{f(\g_0-x+y)}{f(\g_0)},\\
\g_0 &\textrm{otherwise.}
\end{cases}
\end{equation*}

\item Iterate from 1. where $\g_0\leftarrow\g_1$.
\end{enumerate}

This algorithm can be refined: The probability of move, birth or death proposals may differ and may depend on $x$, similarly the law for choosing a point before killing it, adding it or moving it may be chosen properly (e.g. according to the intensity law $\nu$).  The idea of this procedure is that, starting from an allowed configuration $\g_0$, the iterations converge to the realization of an invariant measure which is the Gibbs process we want to simulate. For classical Gibbs point processes, this convergence is proved for example in Section 7.3 of \cite{MW}. In our setting, the convergence is not obvious. It  mainly relies on the following connectivity property (see also Definition \ref{connected} in the appendix): from any allowed configuration $\g$, it is possible to reach another allowed configuration $\g'$ thanks to an iterative birth-death-move procedure as above. Since a hardcore Gibbs tessellation may be very rigid, this property does not always hold (consider for instance the Delaunay tessellations where all the triangles are imposed to be almost equilateral). Yet, in most situations, the connectivity exists. Let us note that the moving step is crucial here, because it allows the connectivity of rigid tessellations that a simple birth-death procedure would not.

We show in the appendix that, under some assumptions, the algorithm converges. These assumptions are fulfilled for Model 2 when $\; 2\varepsilon < \alpha < \frac{1}{2}$. For the Voronoi tessellation presented in Model 3, we claim that the convergence holds for a large reasonable set of parameters $(\epsilon, \alpha, \theta)$. However, in this case theoretical justifications become tedious and are not achieved in this paper.

\subsection{Practical implementation}

In the above algorithm, the choice of the initial configuration is crucial. We must start from an allowed configuration $\g_0$ in $[0,1]^2$, i.e.  $\bar\g_0\in\Mdi$. For this reason, we cannot start from the empty configuration. In our simulations, we chose to start from the point configuration whom  Delaunay tessellation is a regular lattice of triangles (see its plot on top left of Figure \ref{model1}). This starting configuration is allowed by all our hardcore models, provided the distance between points is properly chosen. 

The computation of the ratio in steps 3.-5. of the algorithm is time-consuming since it supposes the computation of two tessellations plus a calculus on their cells to obtain their respective density. But it is possible to simplify this computation by focusing on a smaller window. Indeed, consider for instance step 3, i.e. the birth case (the same approach remains true for the death or the move step). When one adds a point $x$ in a configuration, the new tessellation differs from the previous one only in a neighbourhood of $x$. Thus, the ratio of the densities reduces to a difference of energies in this neighbourhood. The size of this neighbourhood is determined by the size of cells around $x$. For instance in Model 2, the diameter of cells is forced to be smaller than  $2\alpha$, so it suffices to focus on a ball with radius $2\alpha$ around $x$ to compute the difference of energies.

When there is a hardcore interaction, the convergence of the algorithm may be slow. Indeed, when one adds a point, the new tessellation can be forbidden. Hence, in presence of  a strong hardcore interaction, most of the new tessellations proposed by the birth or death step in the algorithm will be refused. For this reason, we check the progress of the algorithm by monitoring control. Every one thousand iterations, we count the total number of points in the configuration as well as  the number of  accepted birth steps (among one thousand steps), the number of accepted death steps and the number of accepted move steps. The plot of theses numbers all along the iterative algorithm helps us to  check the stabilization of the iterative process (though this is not a proof of the convergence to the invariant measure).

\subsection{Some examples}
We present some simulations of the models introduced in Section \ref{troisexemples}. We use the Birth-Death-Move algorithm presented before with $\sigma=0.015$.

Figure \ref{model1} shows a simulated tessellation from Model 1 where $\alpha=\pi/6$. Let us denote by  $g$ the density of the non-stationary intensity measure $\nu$. It is 1-periodic with respect to each component, i.e. for all $(x,y)\in\R^2$, $g(x+1,y)=g(x,y)$ and $g(x,y+1)=g(x,y)$, and for $(x,y)\in[0,1]^2$, we assume $$g(x,y)=100[(x-0.5)^2+(y-0,5)^2]^{-0.75}.$$
The triangles of such a tessellation are forced not to be too flat and to be more dense around the point $(0.5,0.5)$. The monitoring control seems to justify the convergence of the algorithm after $5.10^4$ iterations.

  \begin{figure}[htbp]
  \begin{center}
  \hspace{0cm} \includegraphics[angle=0,scale=.29]{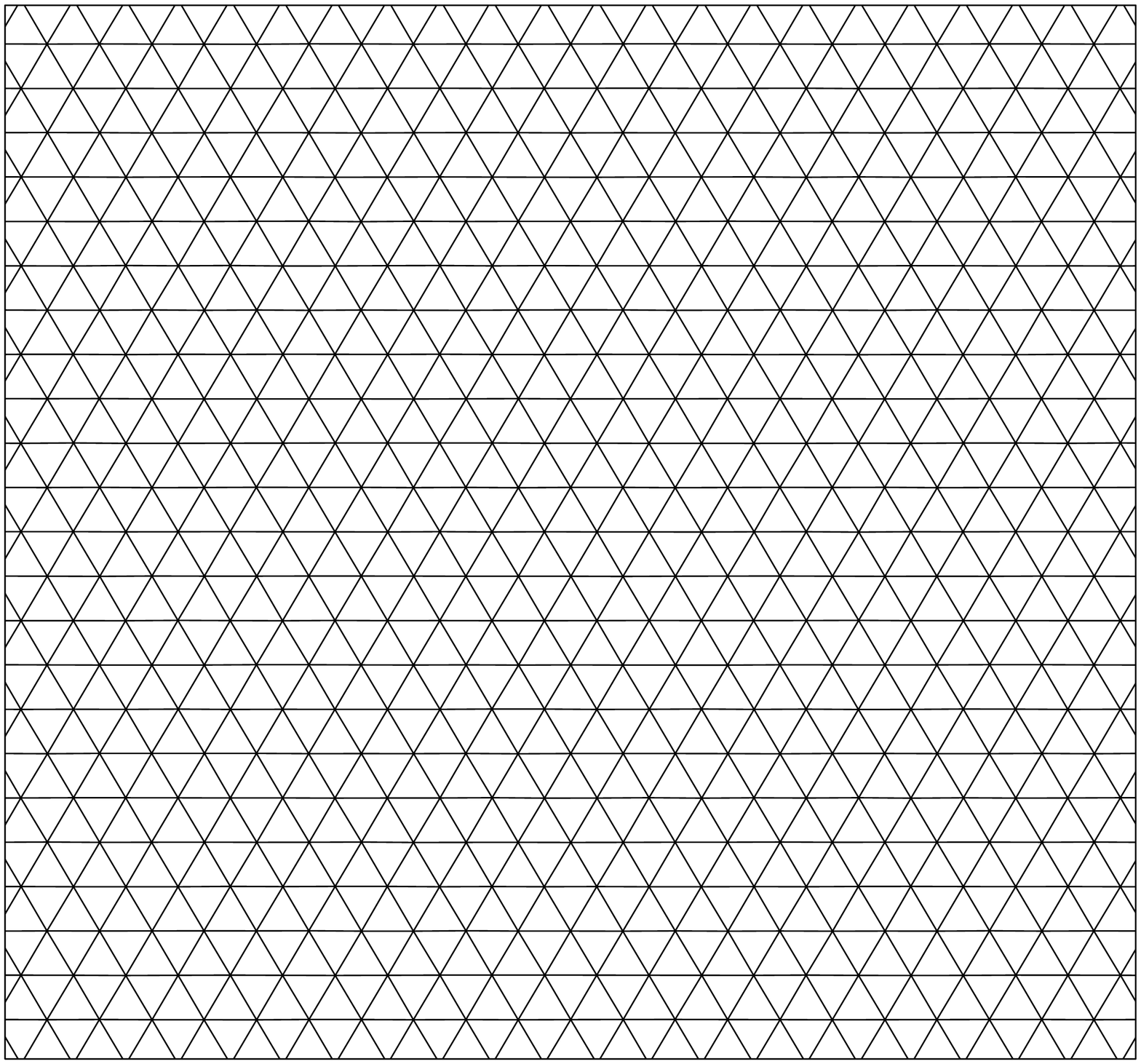}
  \hspace{0cm} \includegraphics[angle=0,scale=.28]{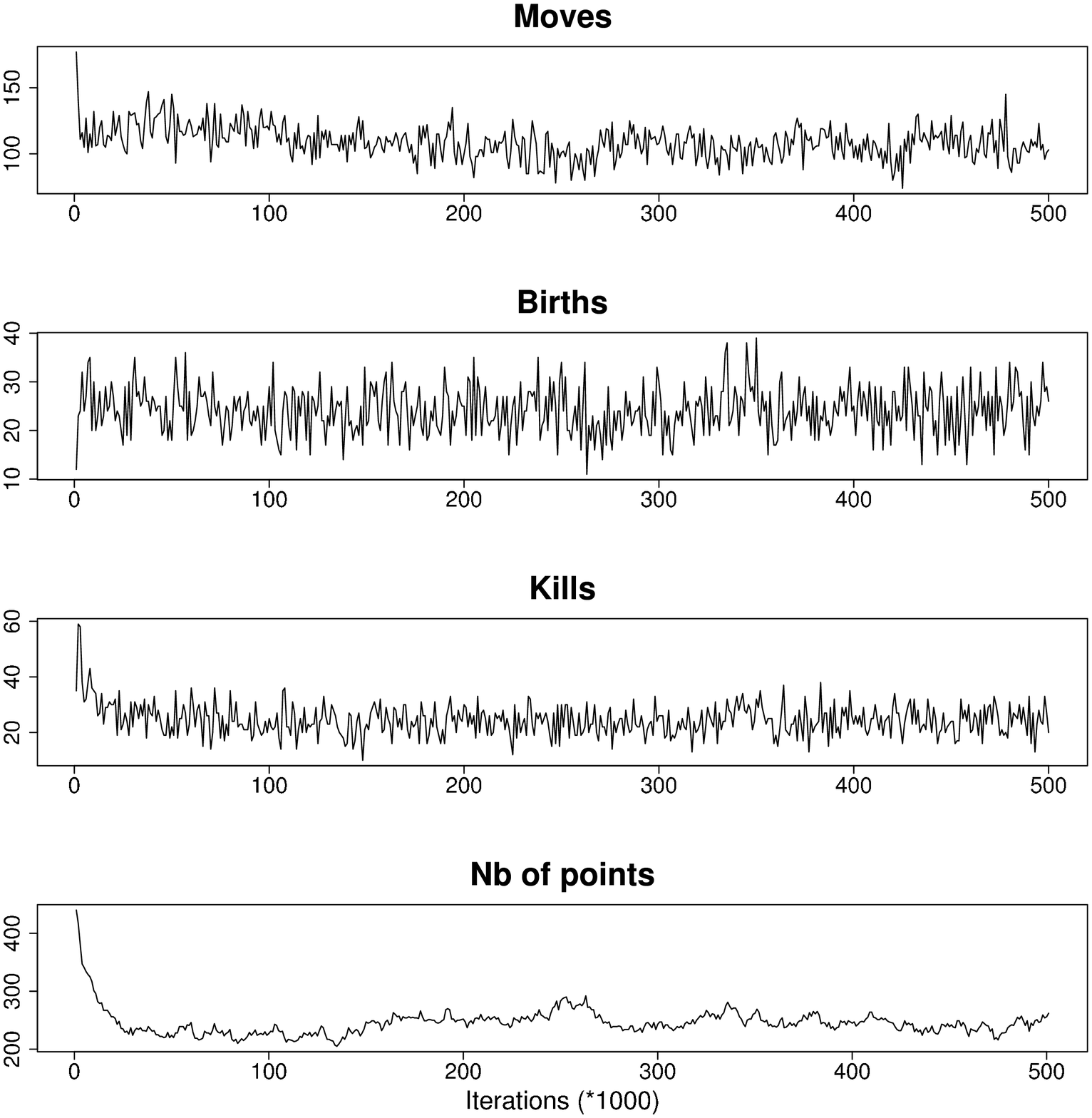}
  \hspace{-0.8cm} \includegraphics[angle=0,scale=.35]{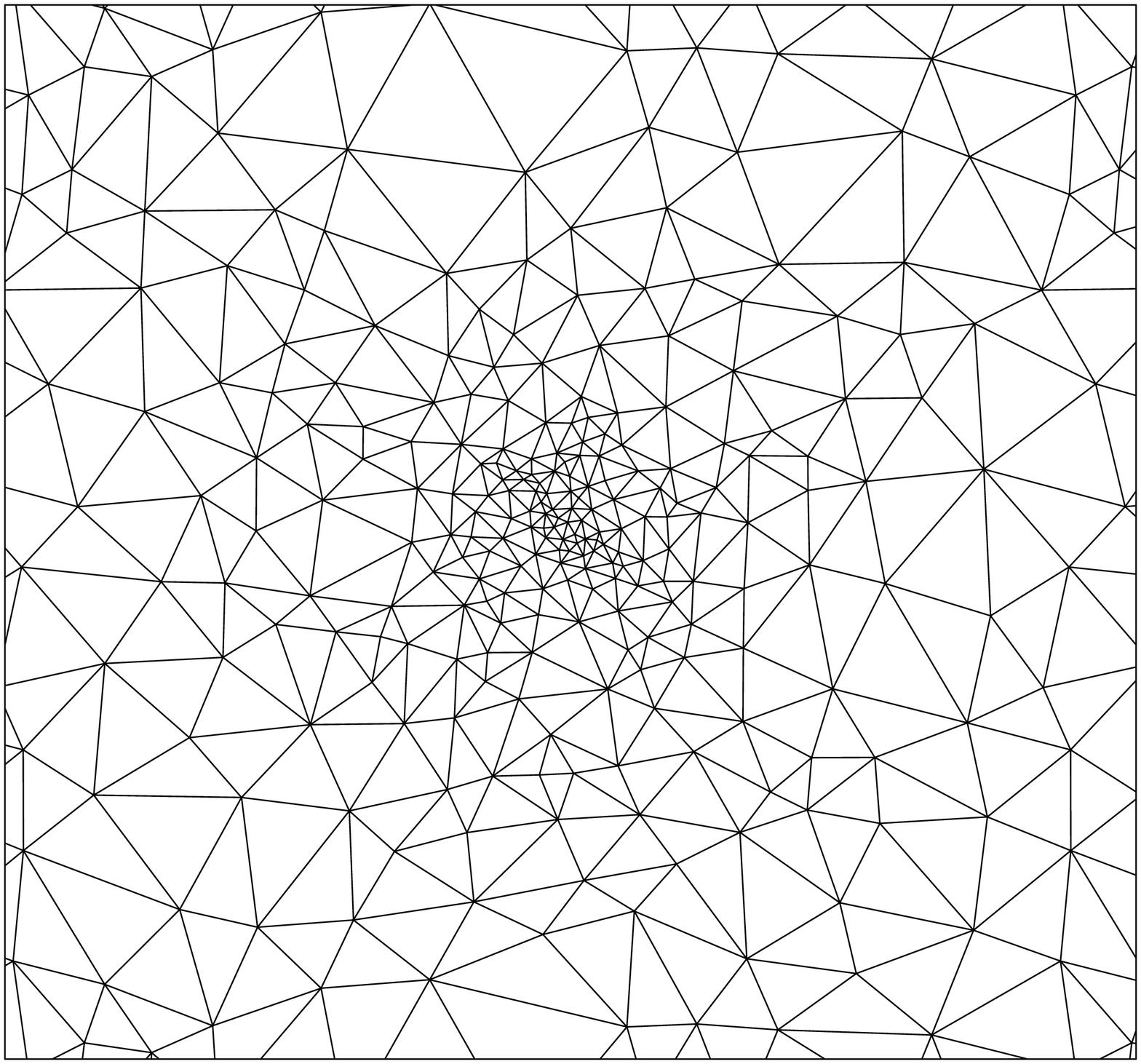}
  \caption{{\small Simulation of Model 1 with $\alpha=\pi/6$ and $g(x,y)=100 [(x-0.5)^2+(y-0.5)^2]^{-0.75}$ ($5.10^5$ iterations). Top left: initial point configuration; Top right: monitoring control (from top to bottom: number of moved points, number of birth points, number of killed points and total number of points, pointed out every 1000 iterations); Bottom: final simulated tessellation.}}\label{model1}
  \end{center}
  \end{figure}

In Figure \ref{model2}, two tessellations from Model 2 have been simulated with $\alpha=0.08$, $z=1000$, and $\theta=\pm 5$. We did not introduce in these simulations the hardcore parameter $\epsilon$. When $\theta>0$ (on bottom), the tessellation is more likely to exhibit a small number of vertices since the total sum of perimeters is then low, which minimizes the energy.  For $\theta<0$ (on top), this is the contrary: the energy is minimal when the total sum of perimeters is high, inducing a lot of vertices in the tessellation. In these two cases, the size of the triangles is controlled by the hardcore parameter $\alpha$, which might be unnecessary when   $\theta<0$ but is certainly a big constraint when $\theta>0$.

Model 3, involving Voronoi tessellation, has been simulated for $\alpha=0.05$, $B=0.625$, $z=100$, and $\theta=-0.8,\ -0.5,\ 0.5,\ †0.8$. The hardcore parameters $\alpha$ and $B$ force the cell not to be too large or  too flat. We did not impose a minimal length of sides through the hardcore parameter $\epsilon$. The parameter $\theta$ quantifies the dependence between the size of neighbour cells: when $\theta<0$ they are most likely to have different sizes, whereas for $\theta>0$ they tend to exhibit the same volume. These two opposite behaviors are clearly observable in the two extreme cases $\theta=-0.8$ and $\theta=0.8$ in Figure \ref{model3}. When $|\theta|=0.5$, this difference is more difficult to distinguish. A challenging task will be to properly estimate the parameters $\alpha$, $B$, $z$, and $\theta$ in both the apparently closed situation $\theta=-0.5$ and $\theta=0.5$. This problem is addressed in the next section.

In Figure \ref{model3-monitoring}, one can check the convergence of the algorithm for the four above simulations. Note that in the very rigid case $\theta=0.8$, the iteration process needs a lot of time (tens of thousands iterations) before starting a birth-death step. This shows the importance of the moving step.

More simulations of Model 3 are presented in Figure \ref{model3bis} with their monitoring control in Figure \ref{model3bis-monitoring}. They show the impact of $B$ on the geometry of the cells. It plays a bigger role when $\theta<0$, since in this case we note the cells may be very flat without this hardcore parameter ($B=+\infty$). Let us remark that some clusters of small cells appear when $\theta<0$. This is more visible when $B$ is not too constraining.

  \begin{figure}[htbp]
    \setlength{\tabcolsep}{0.1cm} \centerline{
      \begin{tabular}[]{cc}
  \includegraphics[angle=0,scale=.29]{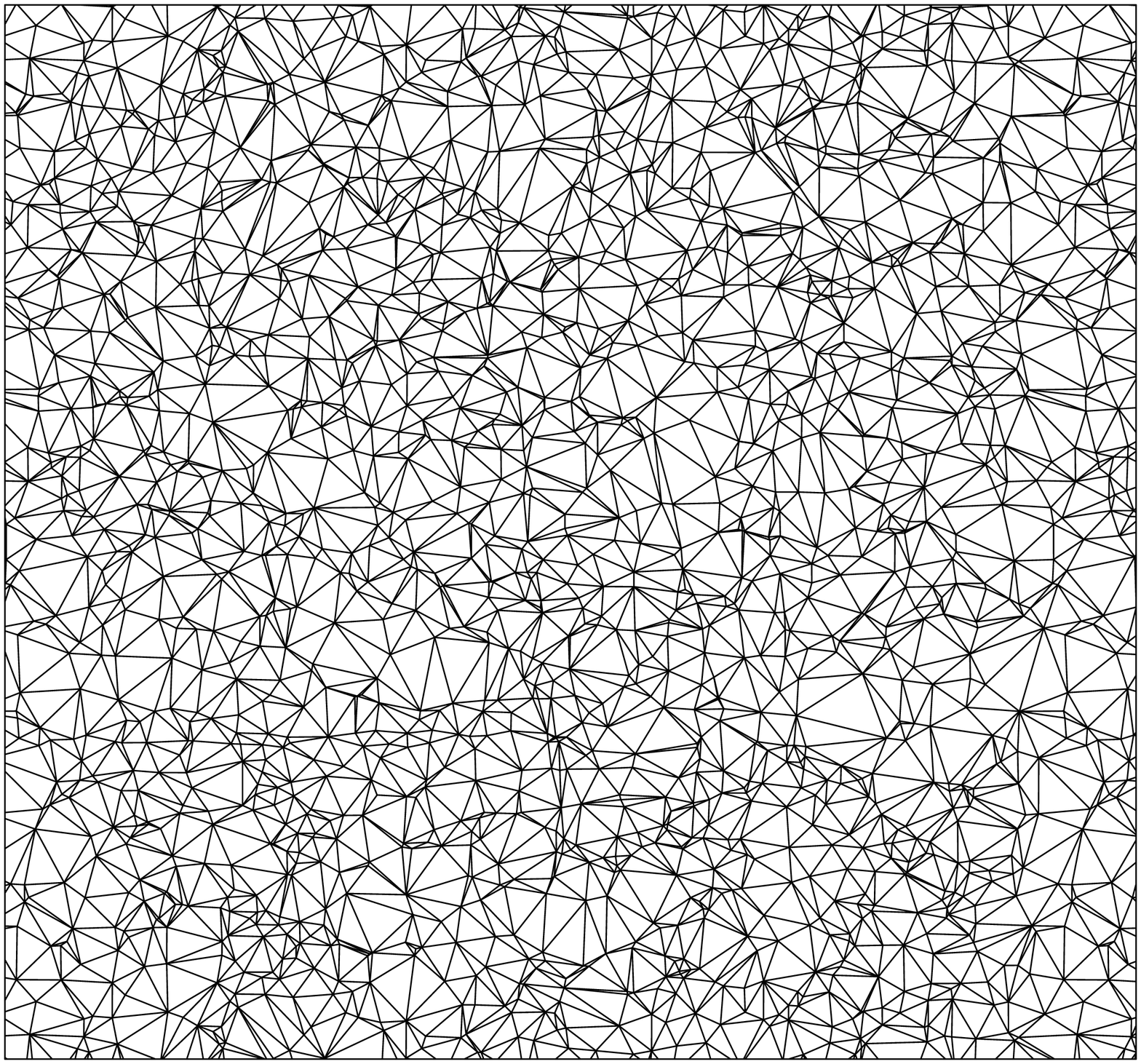} &
  \includegraphics[angle=0,scale=.28]{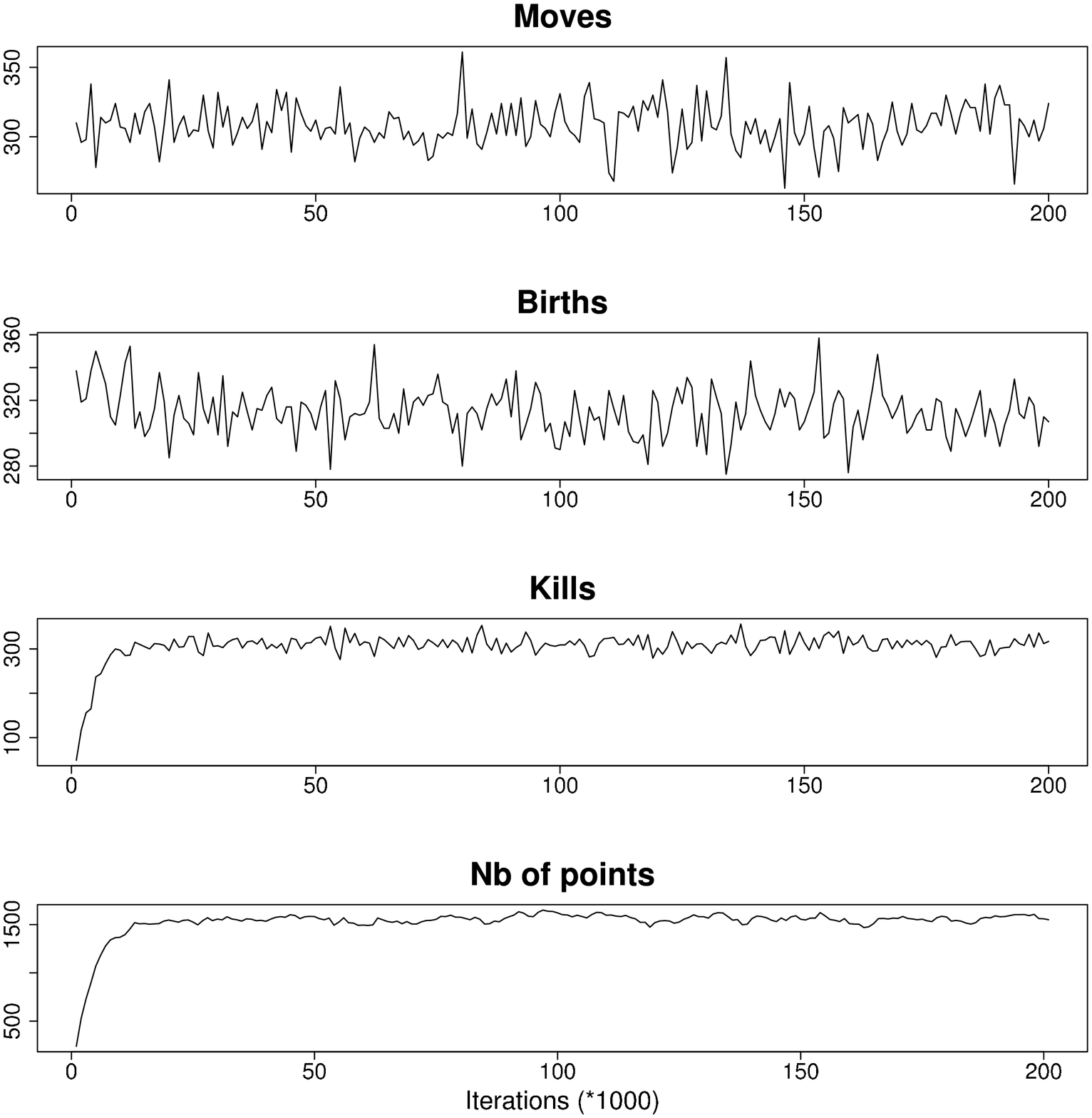} \\
  \footnotesize {$\theta=-5$} & \footnotesize {} 
     	 \\
  \includegraphics[angle=0,scale=.29]{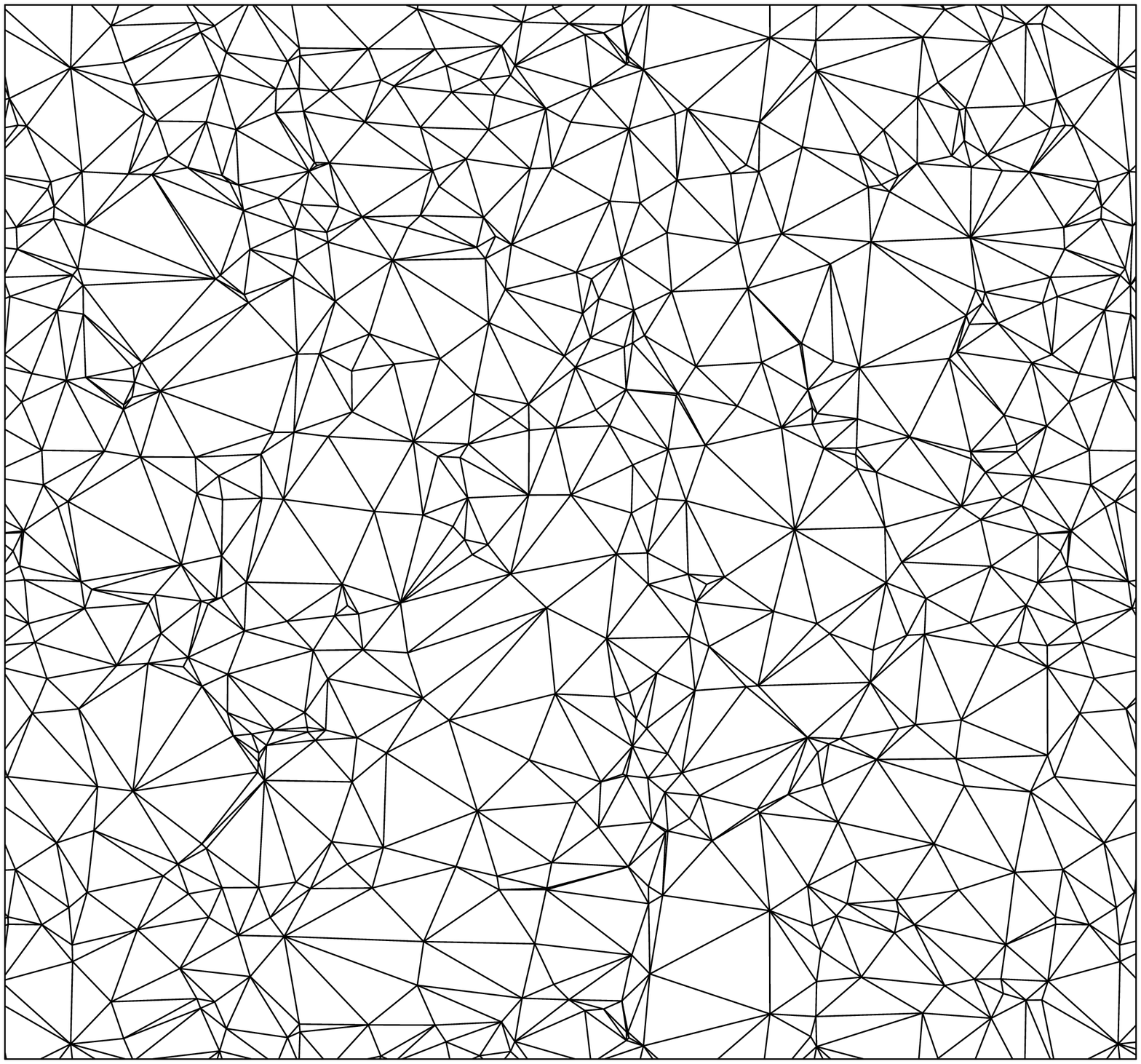}&
   \includegraphics[angle=0,scale=.28]{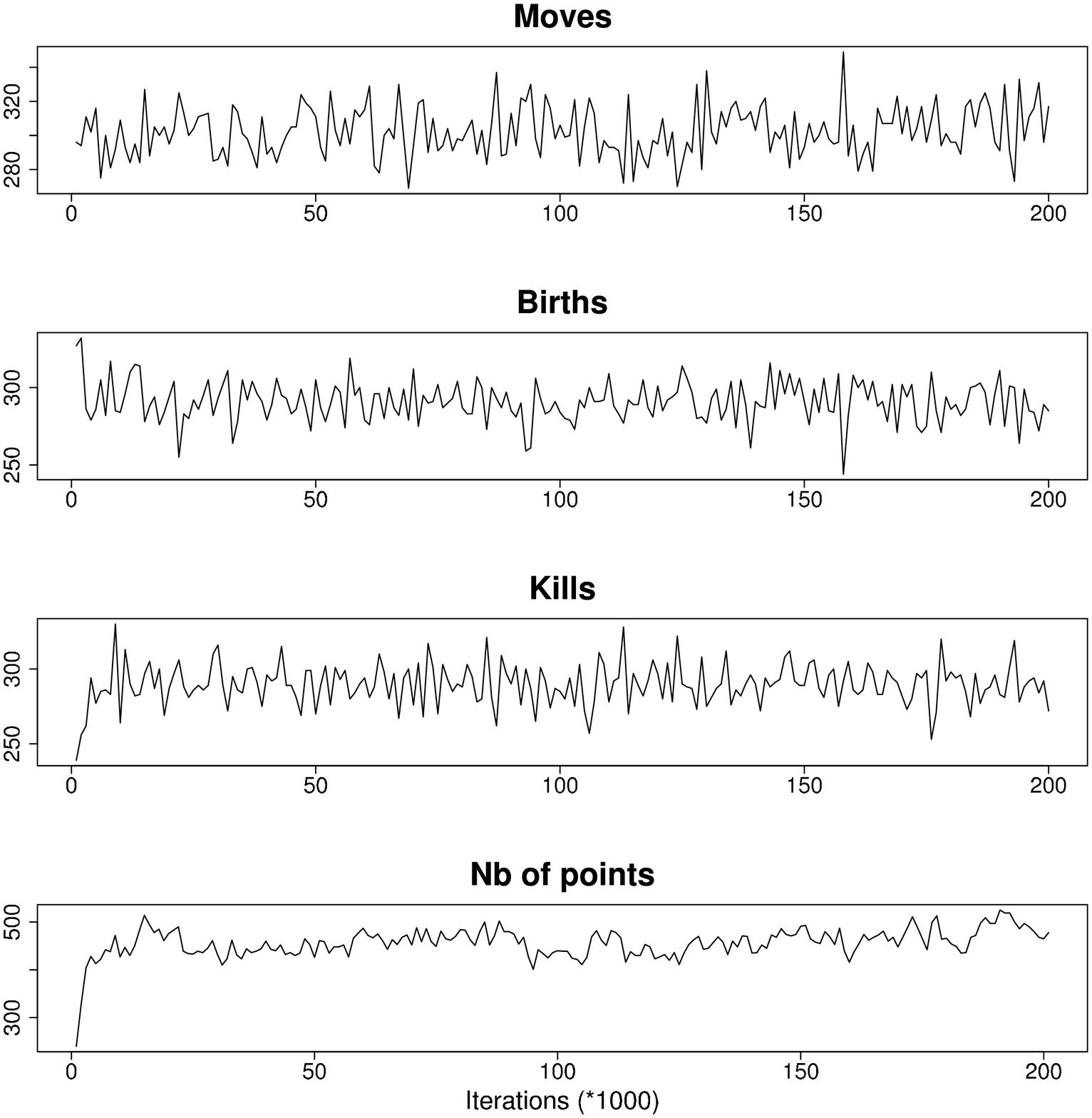}\\
          \footnotesize {$\theta=5$} &
          \footnotesize {}
      \end{tabular}
  }
   \caption{{\small Simulation of Model 2 with $\alpha=0.08$, $z=1000$ and $\theta=-5$ (top), $\theta=5$ (bottom) after $2.10^5$ iterations. Top left: final simulated tessellation when $\theta=-5$; Top right: monitoring control when $\theta=-5$ (with the same plots as for Figure \ref{model1}); Bottom left: final simulated tessellation when $\theta=5$; Bottom right: monitoring control when $\theta=5$ (with the same plots as before).}}\label{model2}

  \end{figure}

  \begin{figure}[htbp]
    \setlength{\tabcolsep}{0.1cm} \centerline{
      \begin{tabular}[]{cc}
  \includegraphics[angle=0,scale=.29]{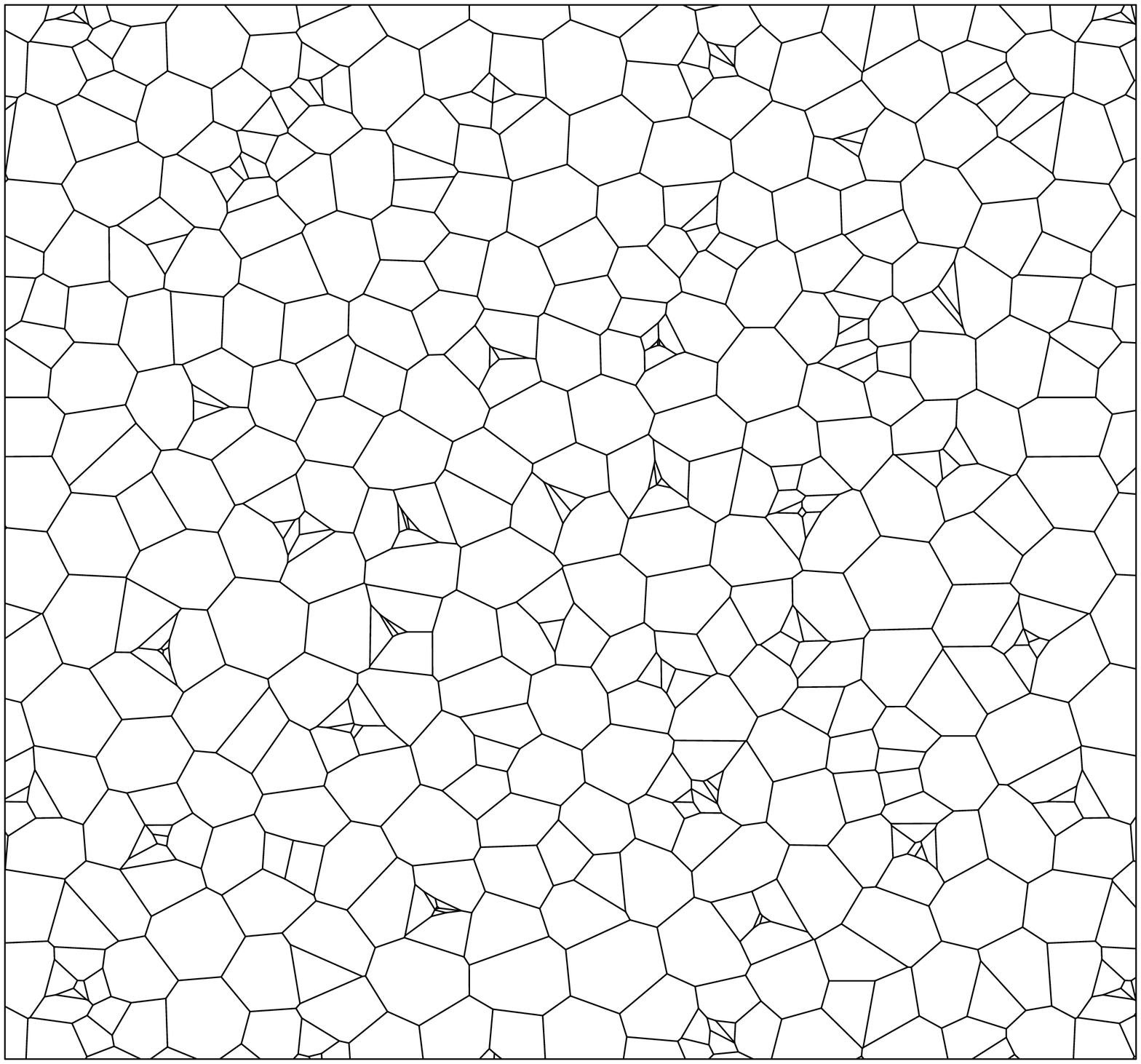} &
  \includegraphics[angle=0,scale=.29]{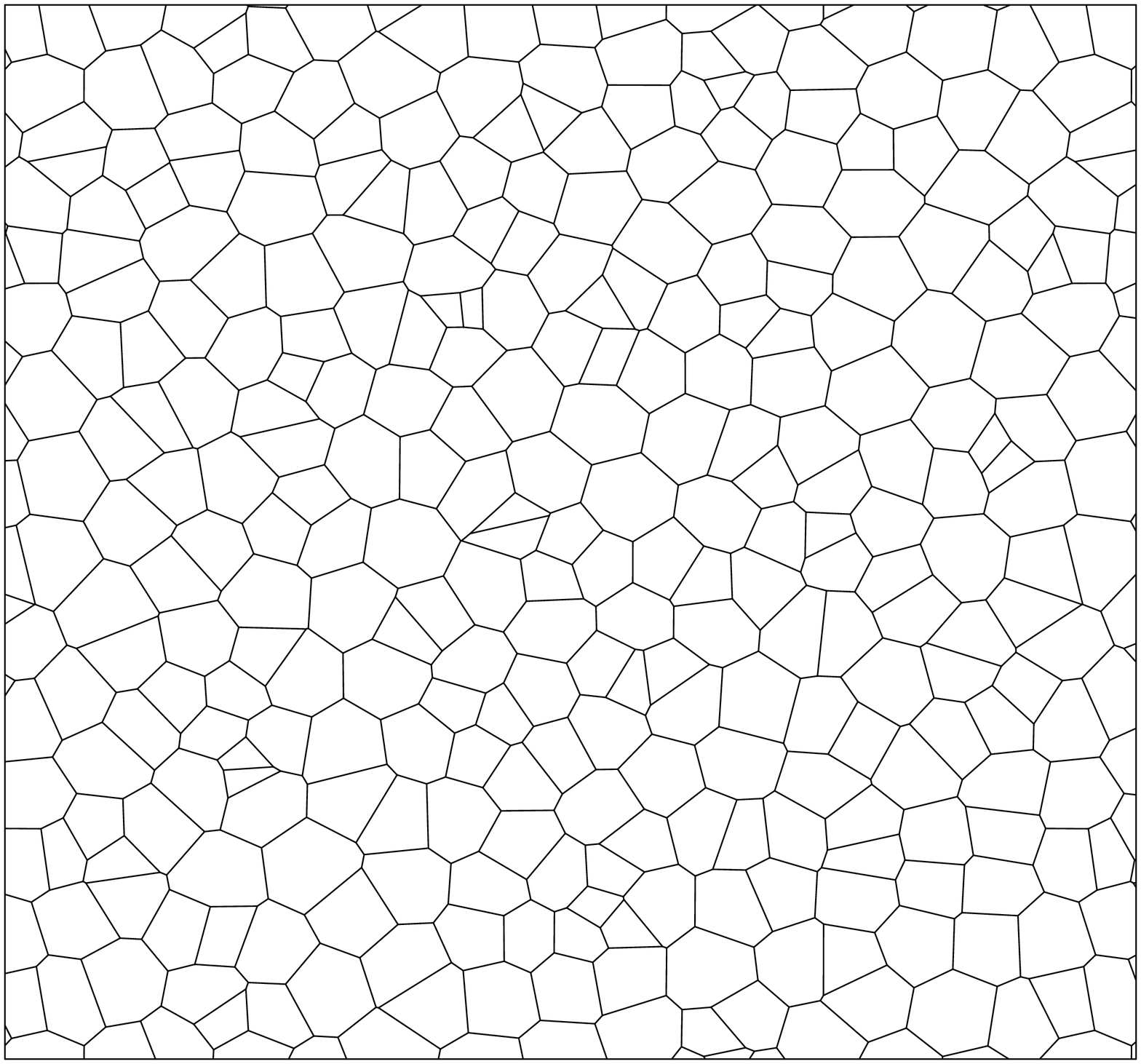} \\
  \footnotesize {$\theta=-0.8$} &
  \footnotesize {$\theta=-0.5$} \\
  \includegraphics[angle=0,scale=.29]{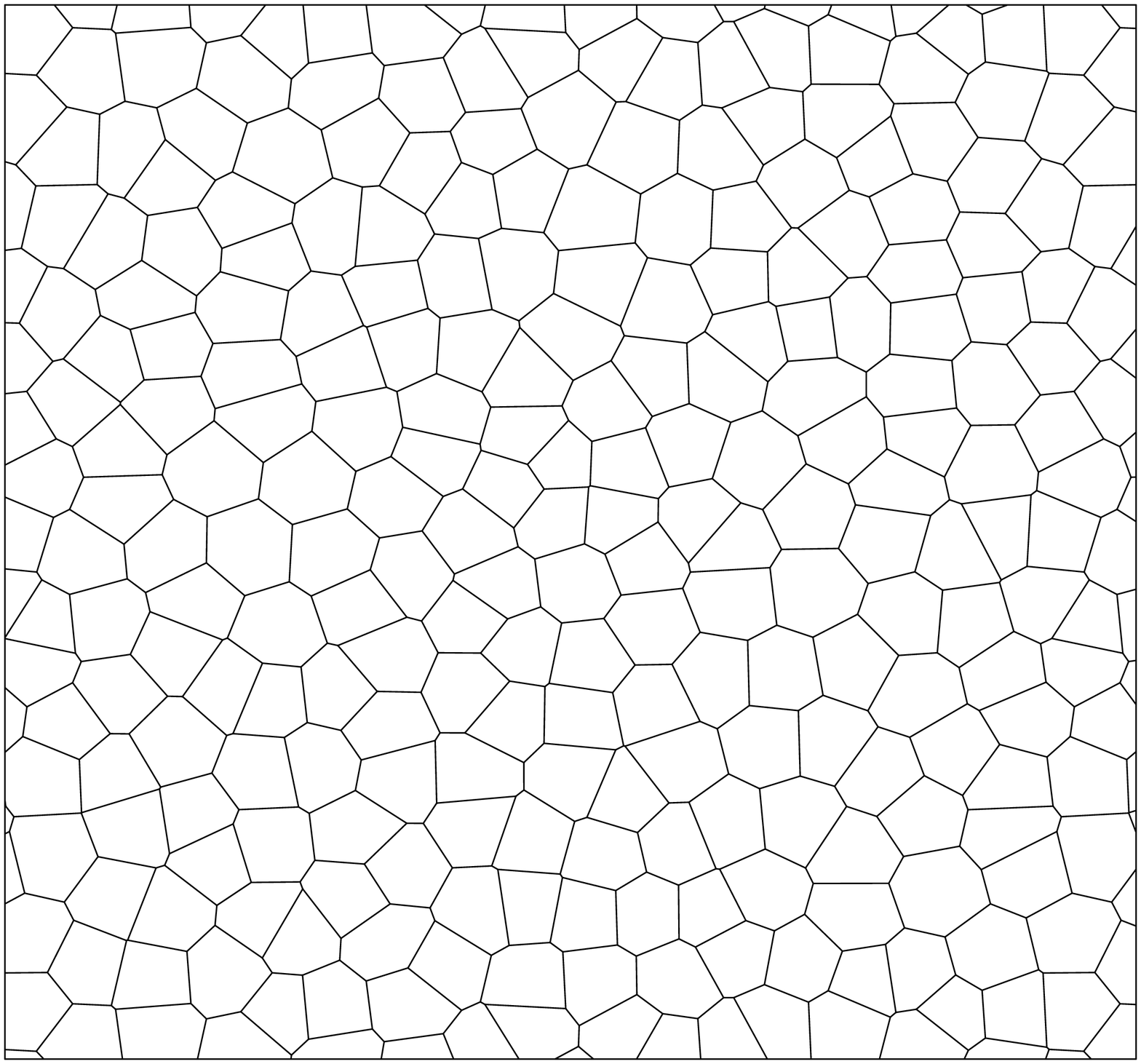}&
   \includegraphics[angle=0,scale=.29]{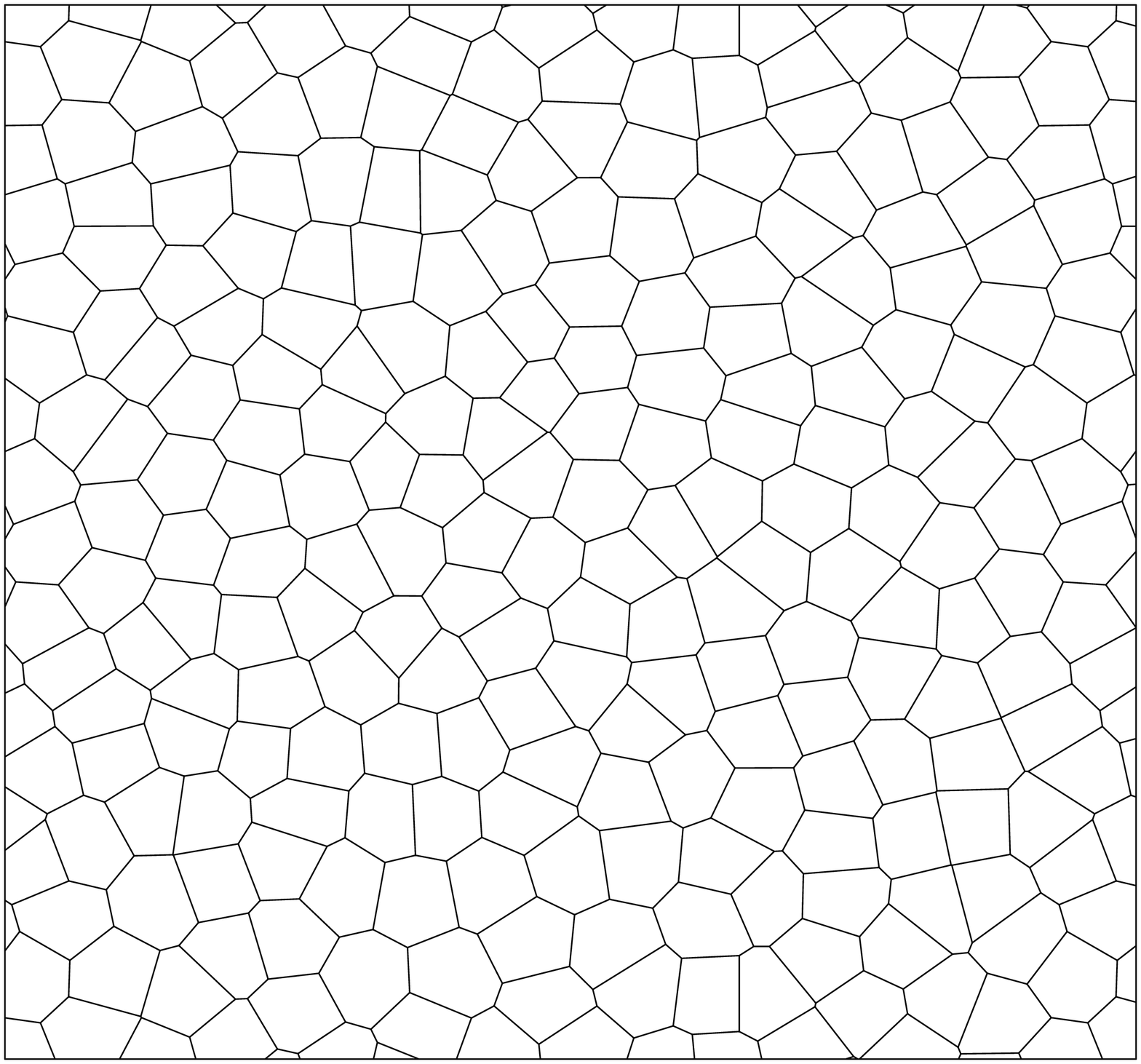}\\
          \footnotesize {$\theta=0.5$} &
          \footnotesize {$\theta=0.8$}
      \end{tabular}
  }
  \caption{{\small Simulation of Model 3 with $\alpha=0.05$, $B=0.625$, $z=100$ and $\theta=-0.8$ (top left), $\theta=-0.5$ (top right), $\theta=0.5$ (bottom left) and $\theta=0.8$ (bottom right). These are the final simulated tessellations after $2.10^5$ iterations when $|\theta|=0.5$ and $5.10^5$  iterations when $|\theta|=0.8$ (see the monitoring control in Figure \ref{model3-monitoring}).}}\label{model3}
  \end{figure}

  \begin{figure}[htbp]
  \begin{center}
  \hspace{0cm} \includegraphics[angle=0,scale=.28]{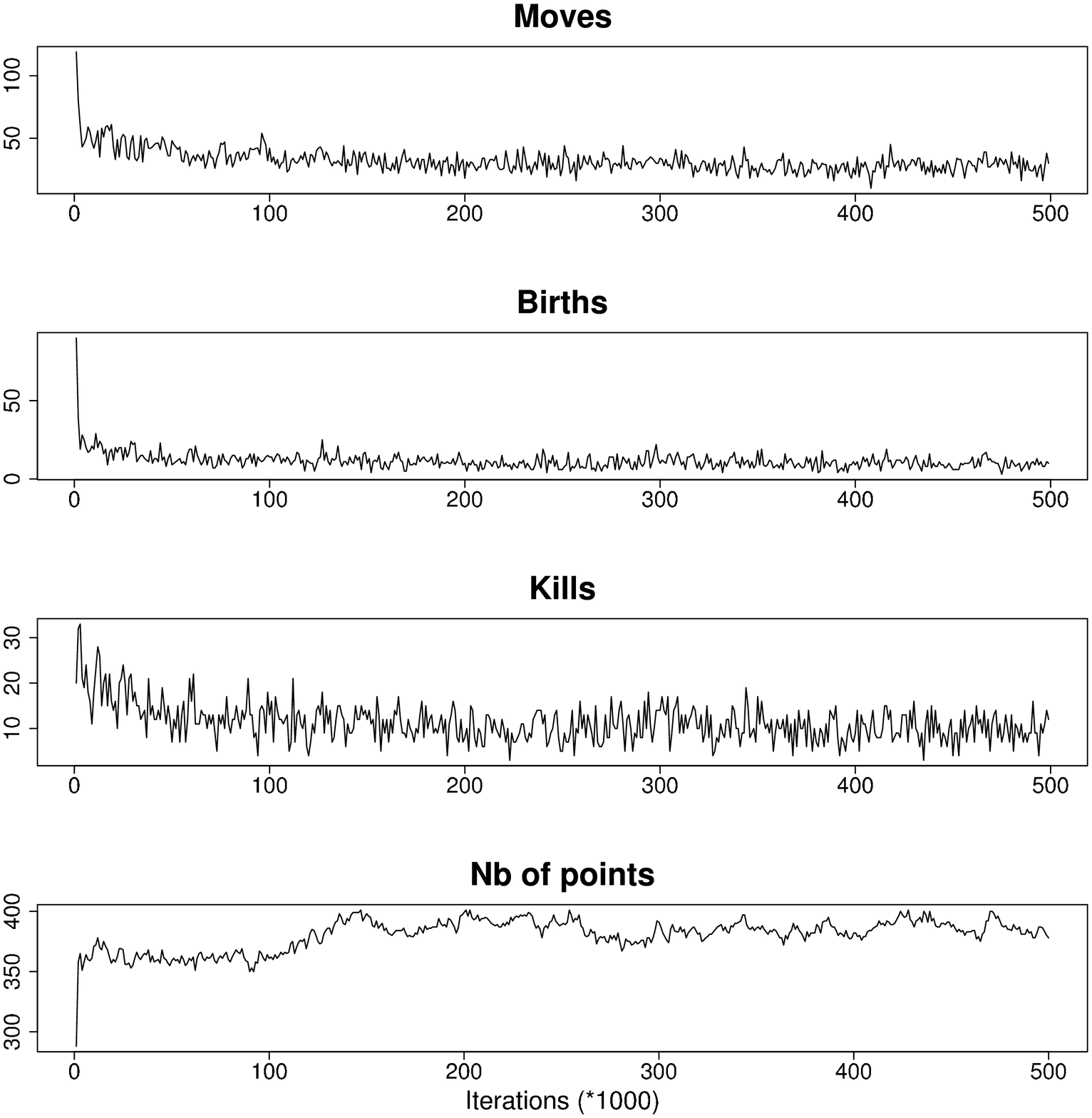}
  \hspace{0.4cm} \includegraphics[angle=0,scale=.28]{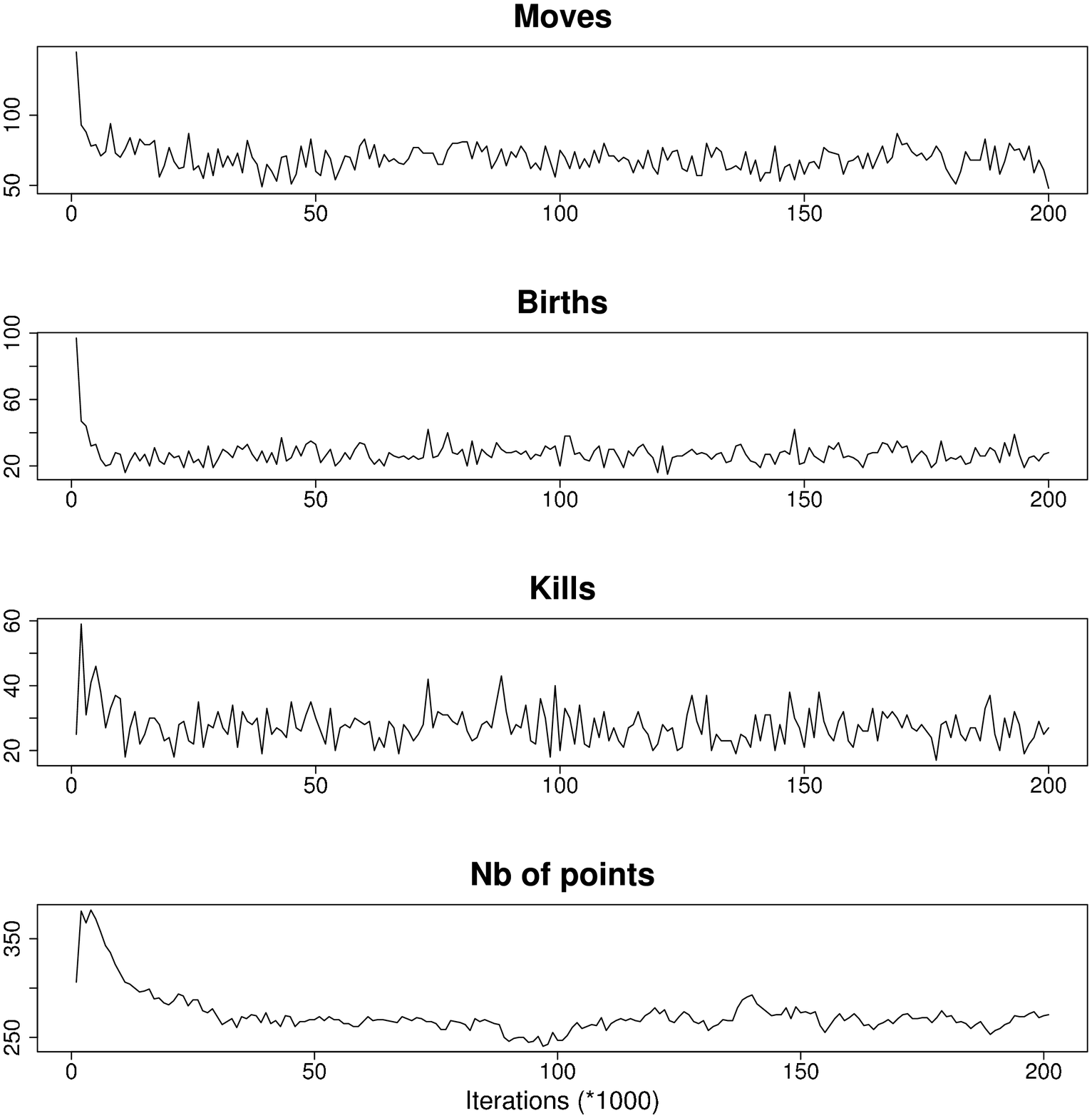}

  \vspace{1cm}

  \hspace{0cm} \includegraphics[angle=0,scale=.28]{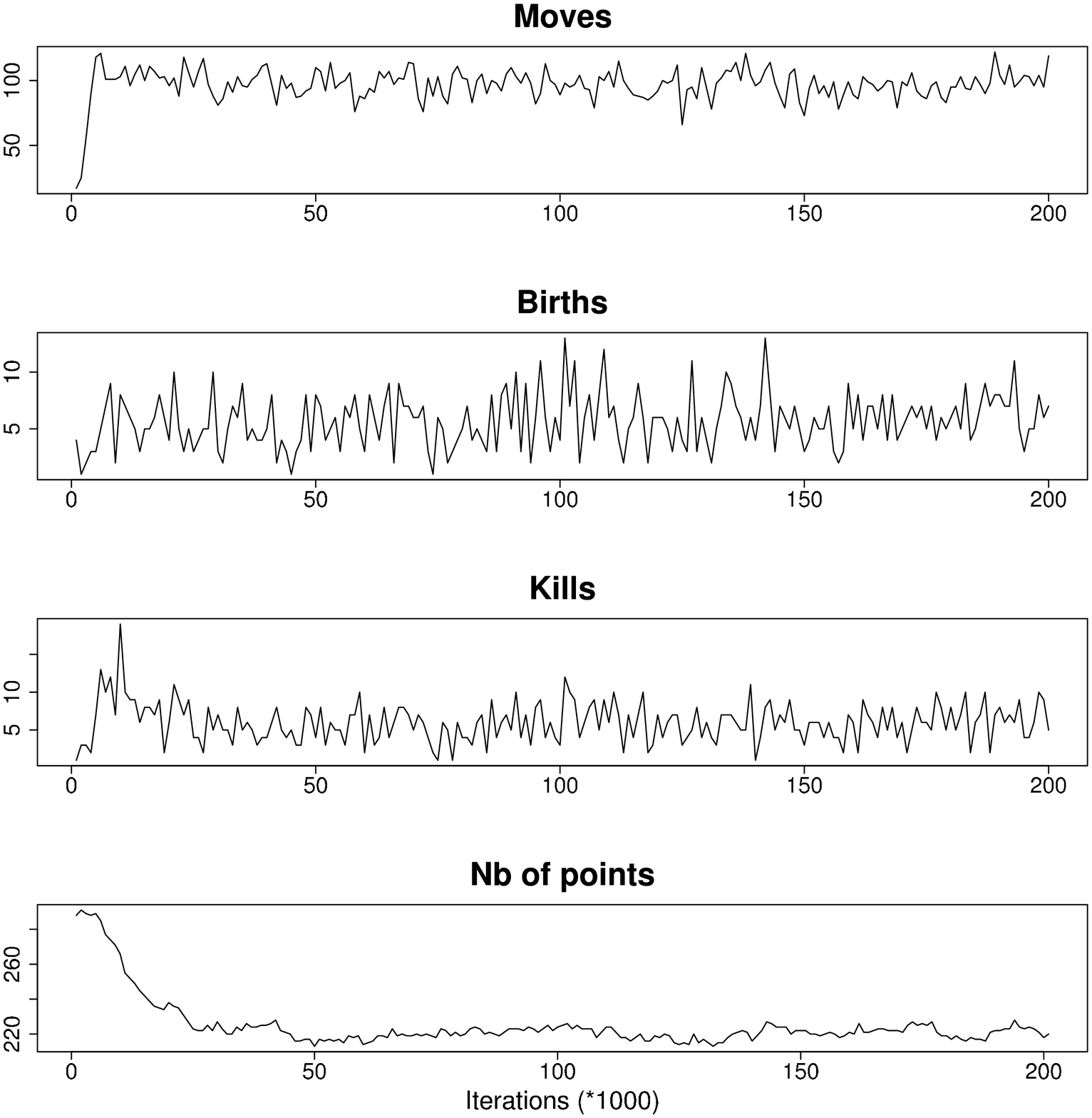}
  \hspace{0.4cm} \includegraphics[angle=0,scale=.28]{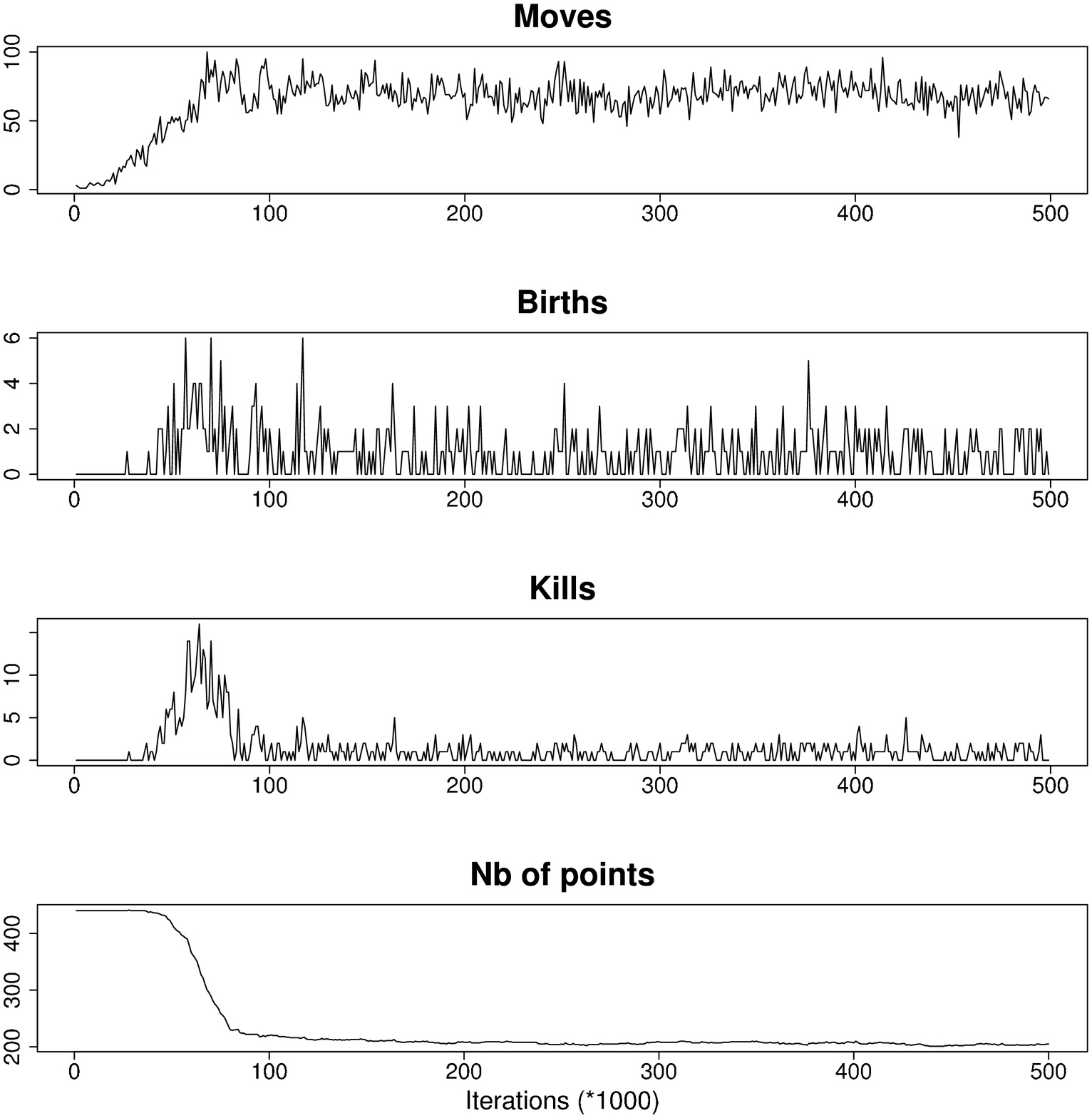}

  \caption{{\small Monitoring control for the simulations of Model 3 presented in Figure \ref{model3}, in the same order (from top left to bottom right: $\theta=-0.8,\ -0.5,\ 0.5,\ 0.8$). They consist in the same plots as for Figure \ref{model1}.}}\label{model3-monitoring}
  \end{center}
  \end{figure}

  \begin{figure}[htbp]
    \setlength{\tabcolsep}{0.1cm} \centerline{
      \begin{tabular}[]{cc}
  \includegraphics[angle=0,scale=.29]{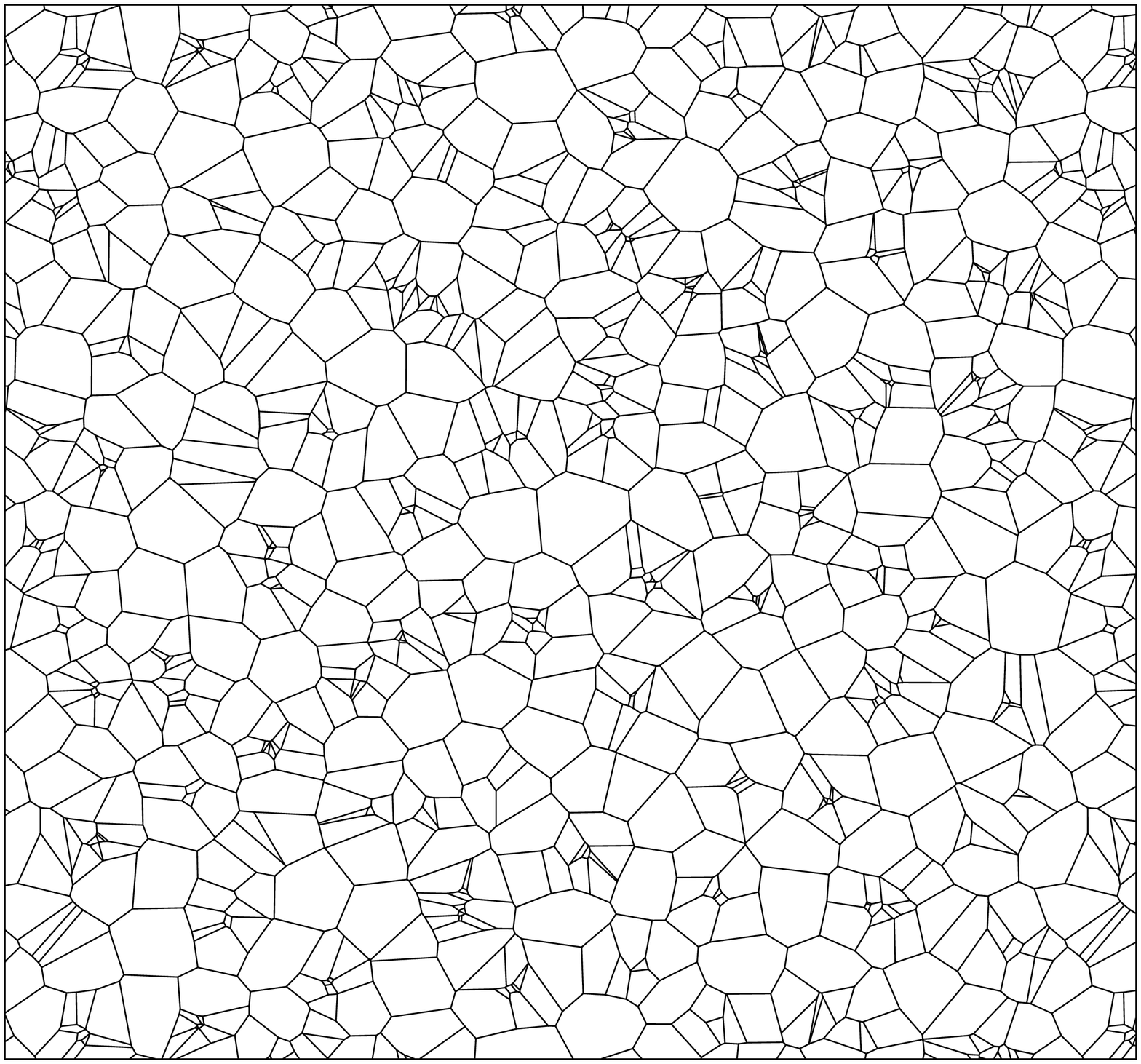} &
  \includegraphics[angle=0,scale=.29]{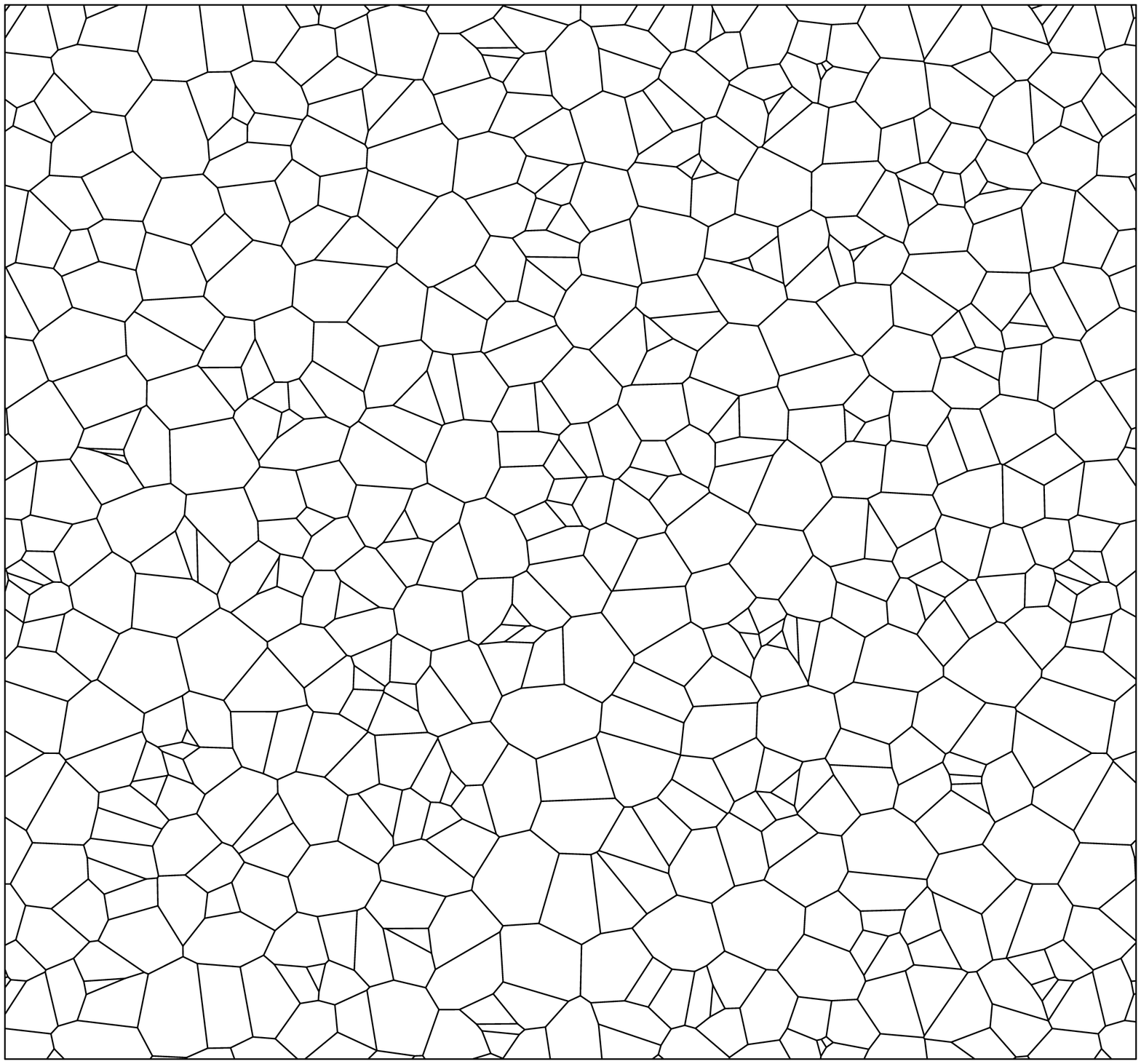} \\
  \footnotesize {$B=+\infty, \quad \theta=-0.5$} &
  \footnotesize {$B=1, \quad \theta=-0.5$} \\
  \includegraphics[angle=0,scale=.29]{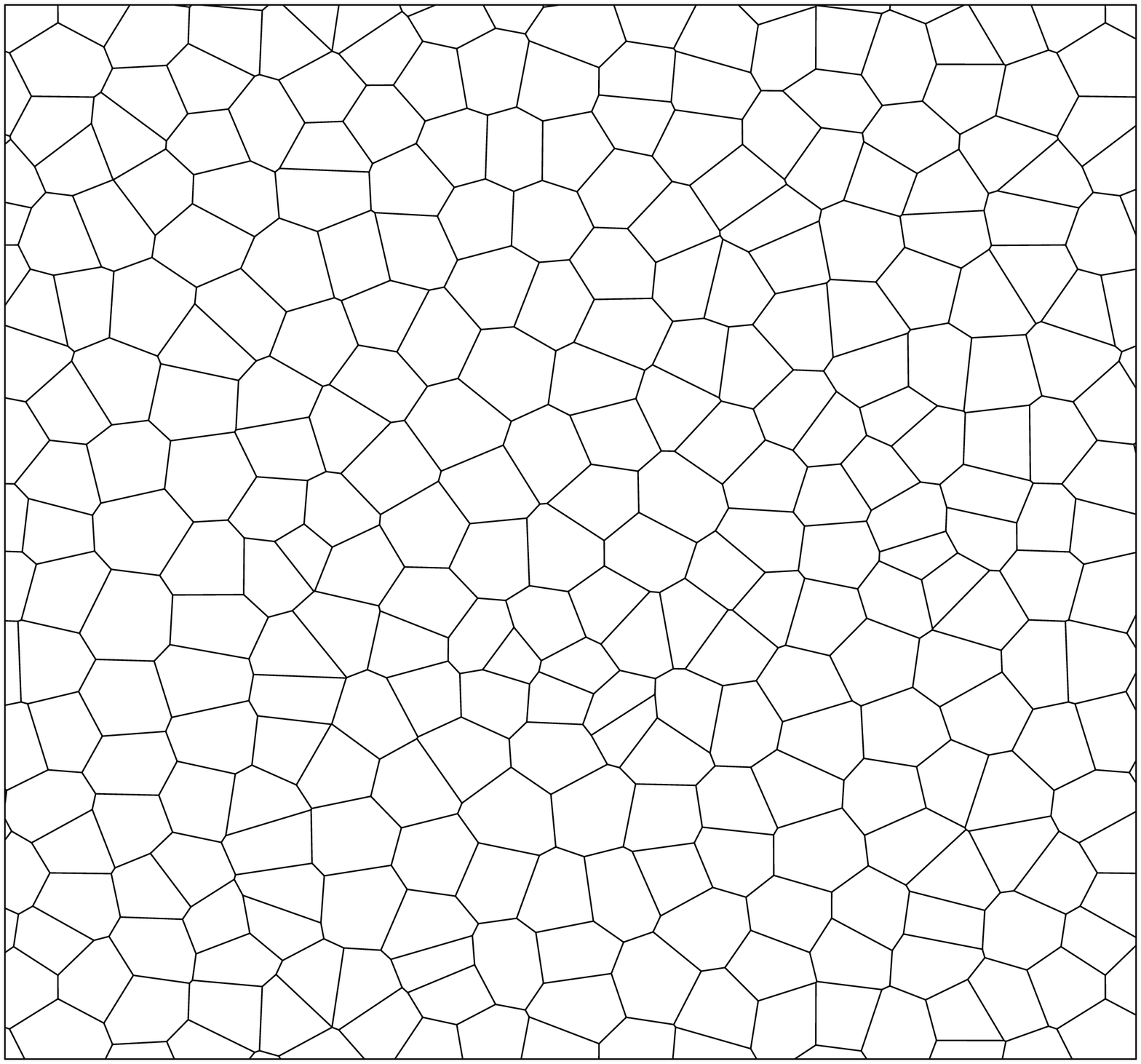}&
   \includegraphics[angle=0,scale=.29]{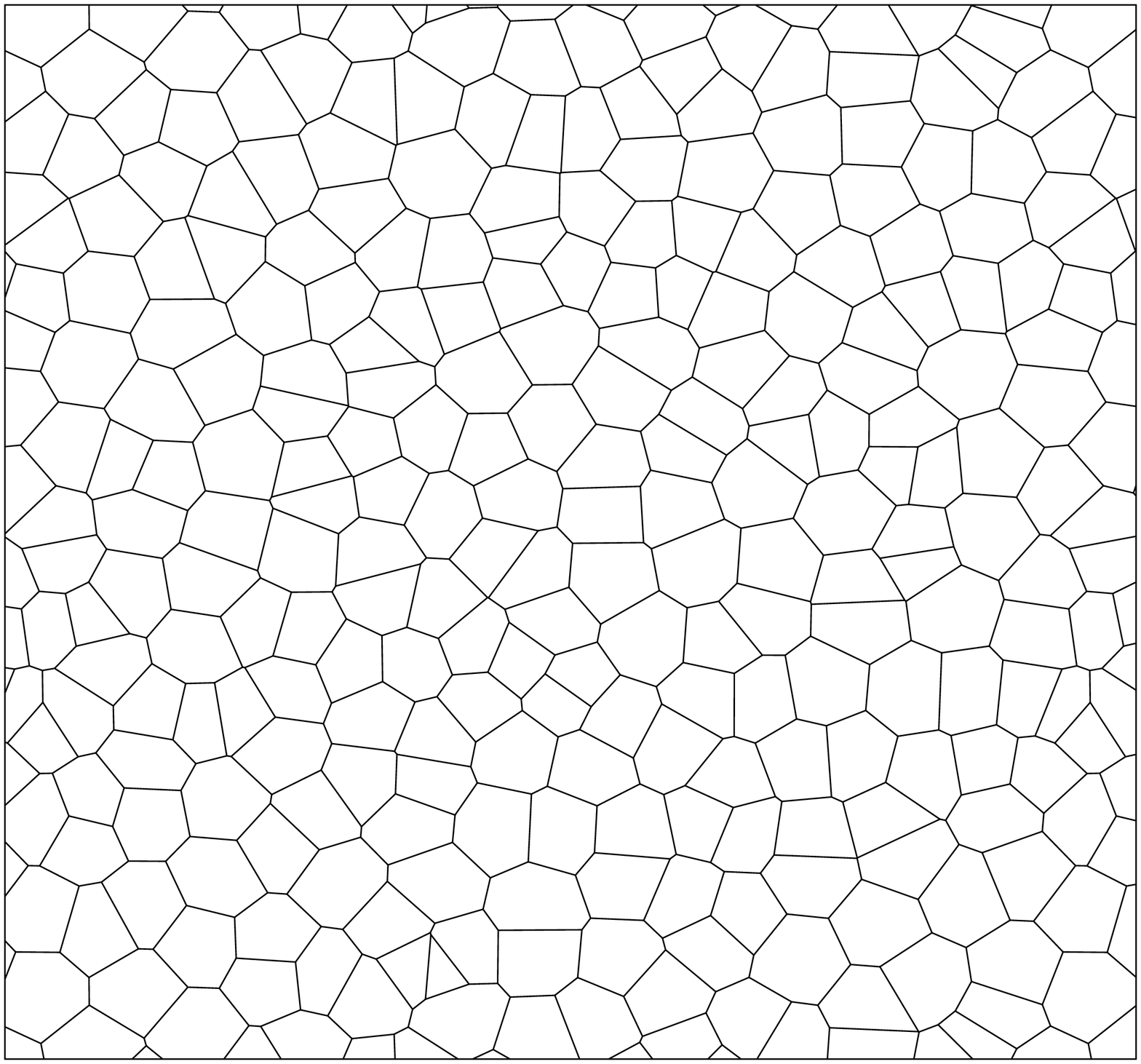}\\
          \footnotesize {$B=+\infty, \quad \theta=0.5$} &
          \footnotesize {$B=1, \quad \theta=0.5$}
      \end{tabular}
  }
        \caption{{\small Simulation of Model 3 with $\alpha=0.05$, $z=100$, $B=+\infty$ (left), $B=1$ (right), $\theta=-0.5$ (top), $\theta=0.5$ (bottom). These are the final simulated tessellations after $1.5.10^5$ iterations (see the monitoring control in Figure \ref{model3bis-monitoring}).}}\label{model3bis}

  \end{figure}

  \begin{figure}[htbp]
  \begin{center}
  \hspace{0cm} \includegraphics[angle=0,scale=.28]{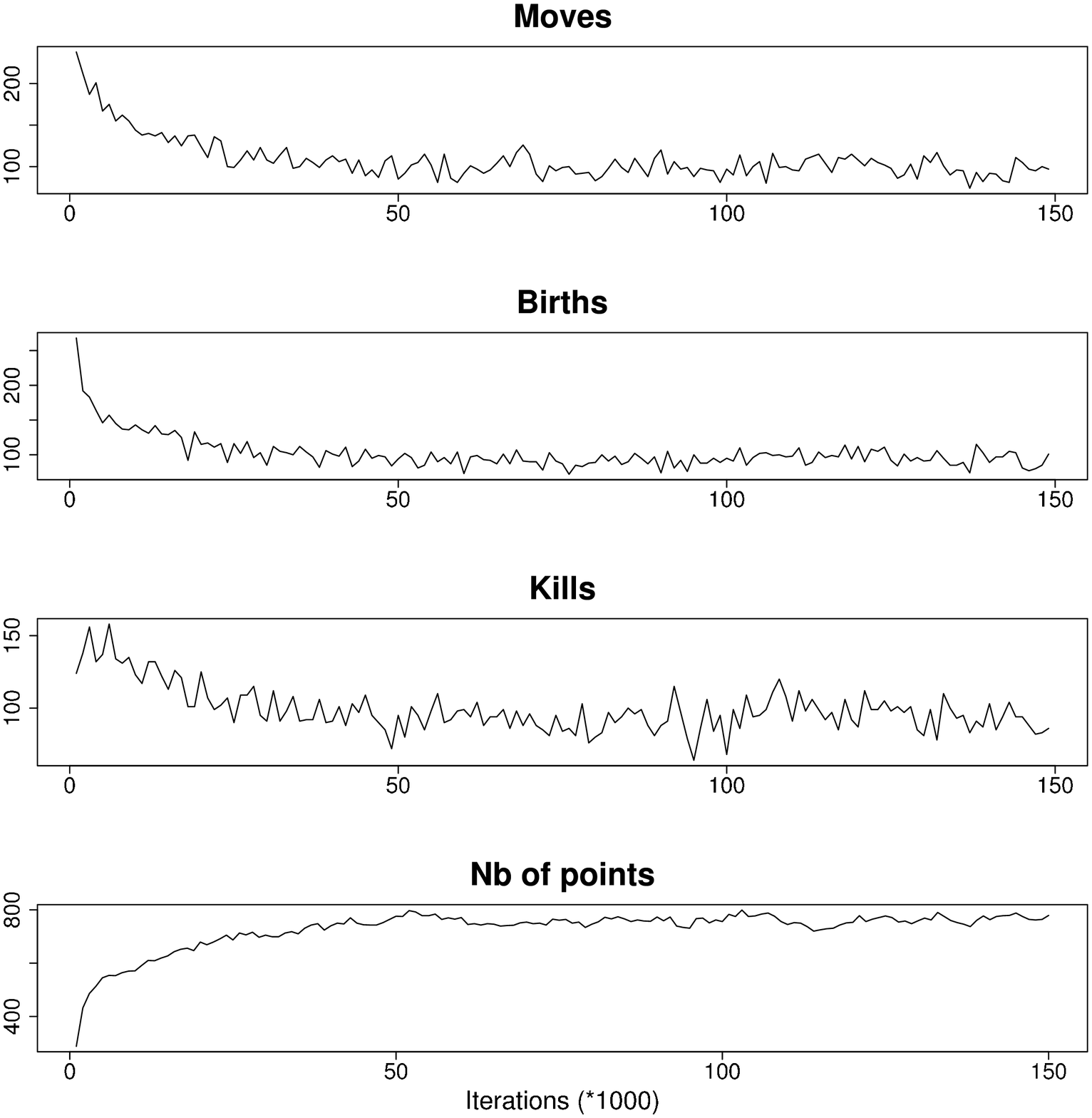}
  \hspace{0.4cm} \includegraphics[angle=0,scale=.28]{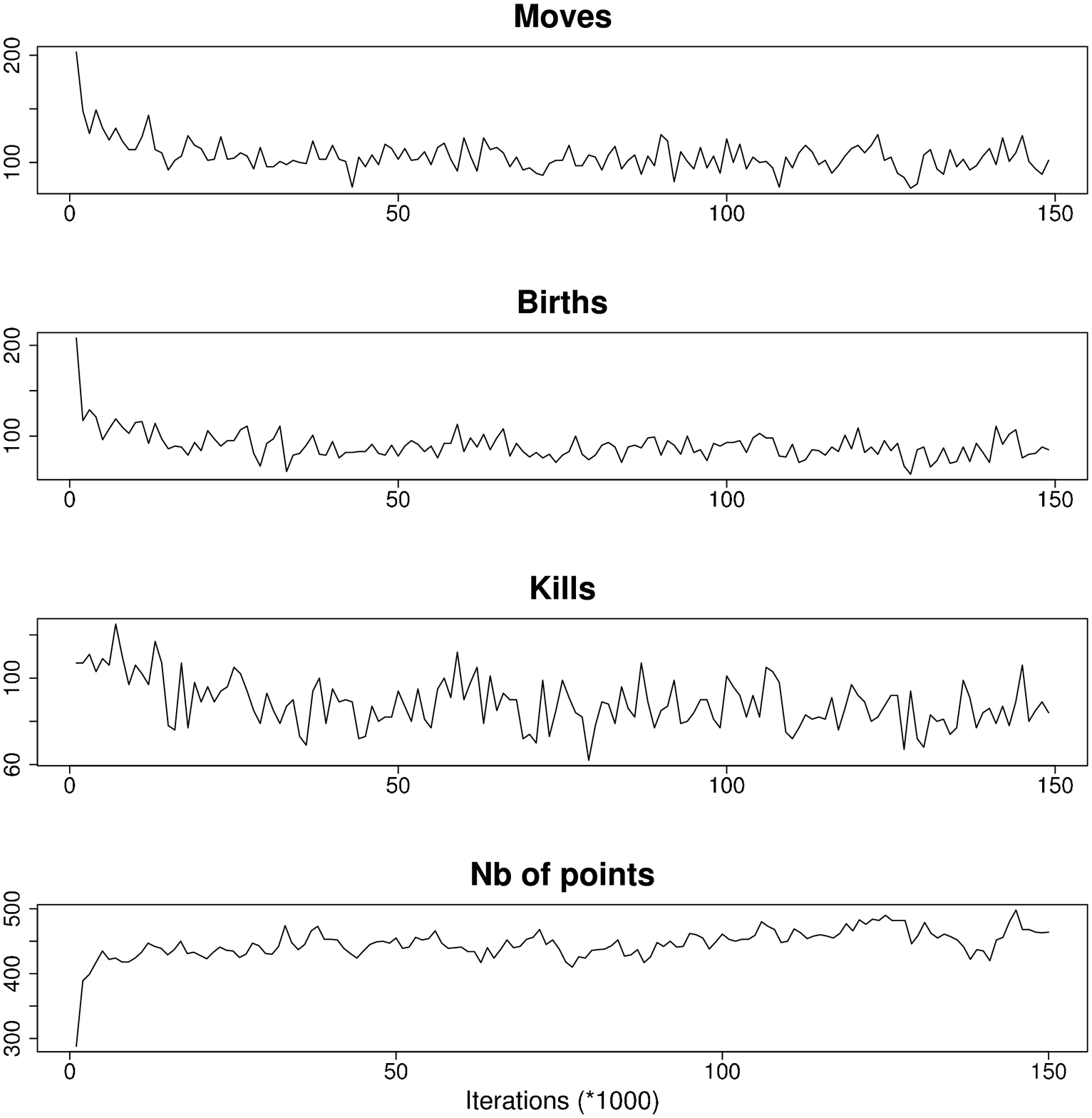}

  \vspace{1cm}

  \hspace{0cm} \includegraphics[angle=0,scale=.28]{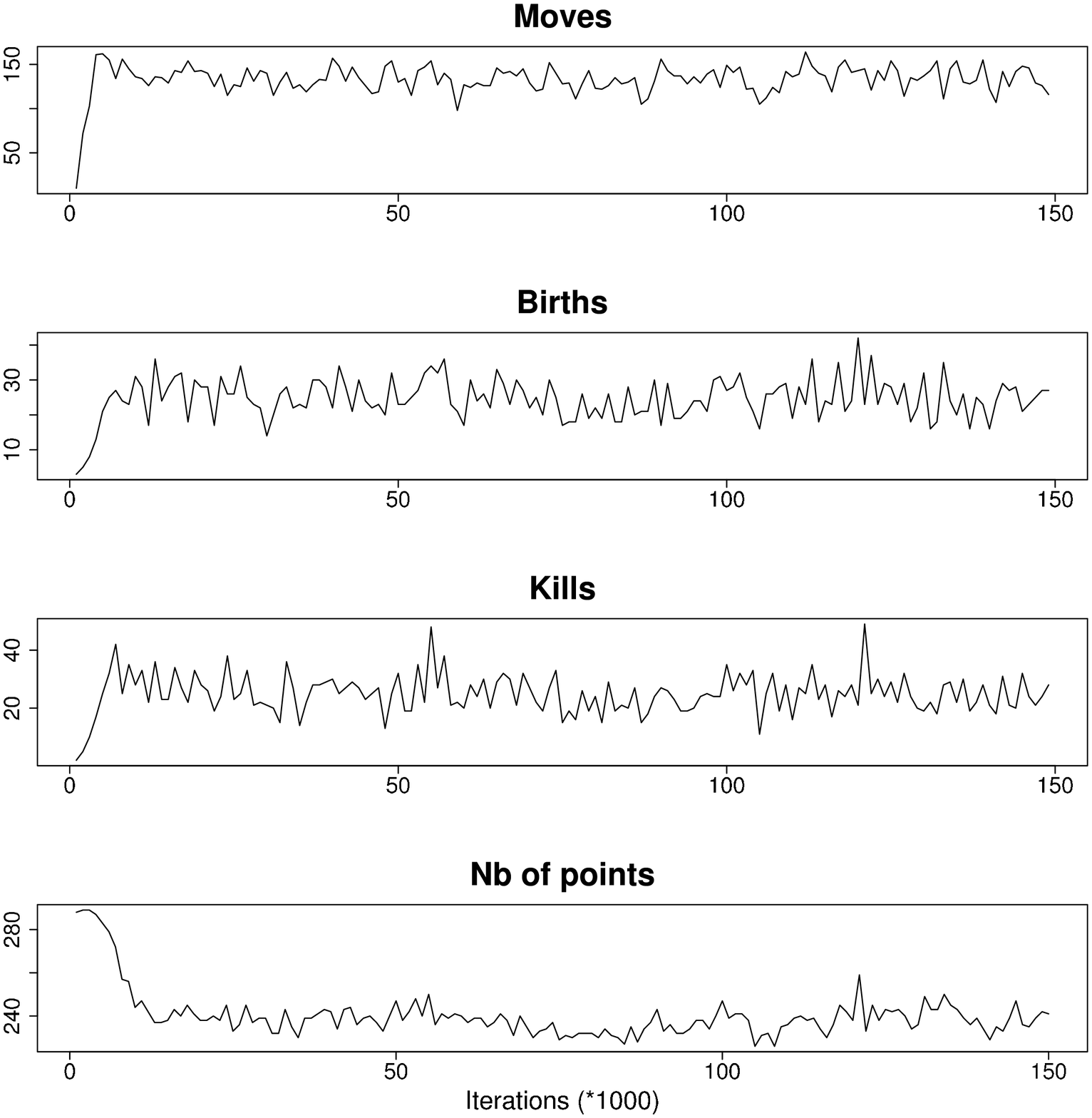}
  \hspace{0.4cm} \includegraphics[angle=0,scale=.28]{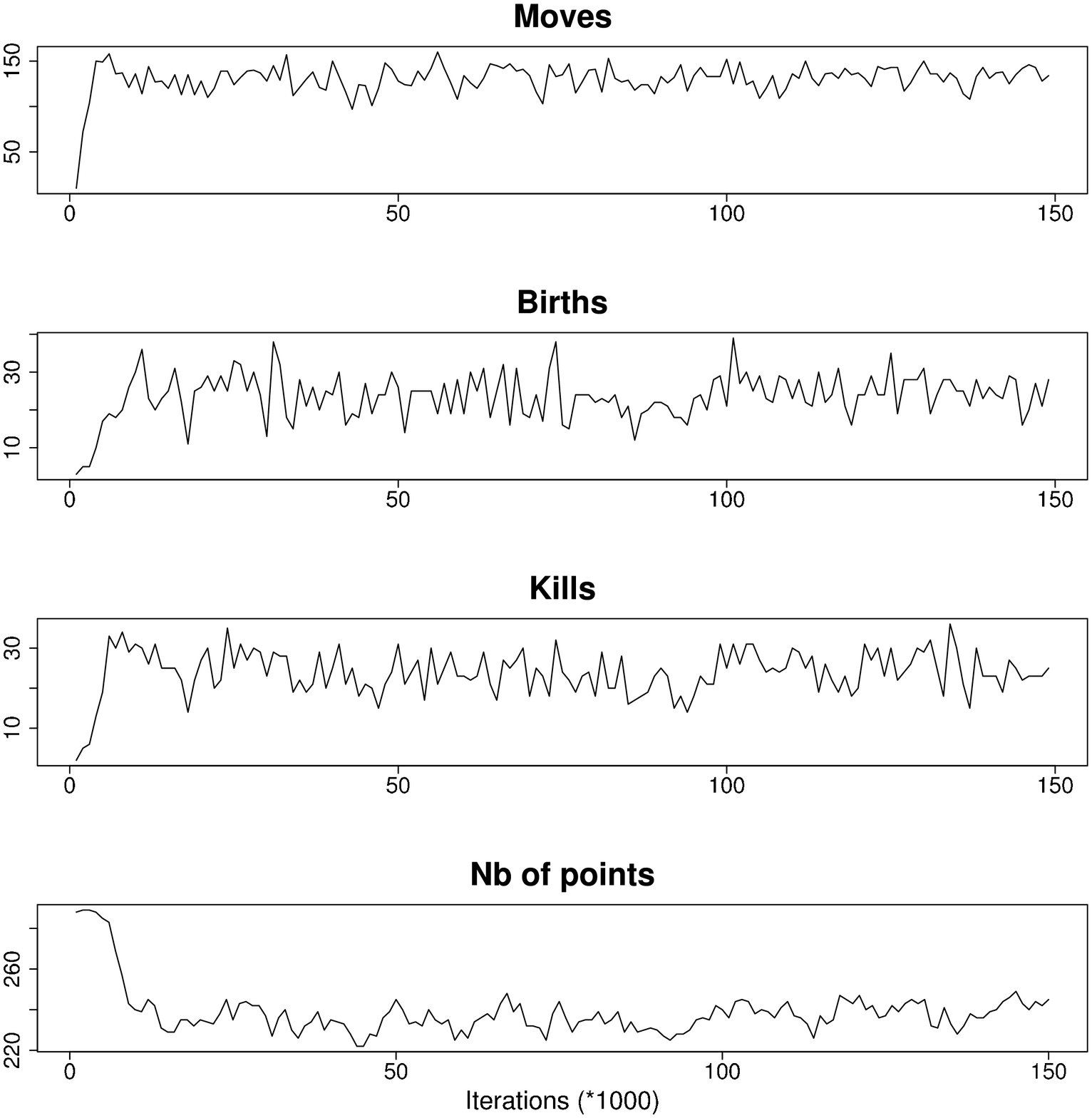}

  \caption{{\small Monitoring control for the simulations of Model 3 presented in Figure \ref{model3bis}, in the same order ($B=+\infty$ (left), $B=1$ (right), $\theta=-0.5$ (top), $\theta=0.5$ (bottom)). They consist in the same plots as for Figure \ref{model1}.}}\label{model3bis-monitoring}
  \end{center}
  \end{figure}

\newpage
\section{Estimation}\label{estimation}

In this section, we focus on stationary Gibbs tessellations, in particular we suppose that the underlying Poisson process is $\pi^z$, the stationary Poisson point Process with intensity $z\lambda$ where $\lambda$ is the Lebesgue measure. This is the case of Model 2 and 3 presented before. We apply our estimation procedure to the parameters $z$, $\epsilon$, $\alpha$, $B$ and $\theta$ involved in  these models. An interesting generalization would be to include the estimation of a non-stationary intensity (as in Model 1) to the estimation of the parameters of the interaction. We do not deal with this task in this paper.

Let us specify some notations. We assume that we observe a tessellation coming from a point configuration $\g$ on a window $\L_n=[-n,n]^2$. This tessellation is defined through an energy function $E_{\L}(\g_{\L},\g_{\L^c})$ as in (\ref{del}) or (\ref{vor}). In the following, we denote it by $E^{\beta, \theta}_\L(\g_\L,\g_{\L^c})$, since we assume a parametric form. We actually suppose that it depends on parameters $(\beta,\theta)$: $\beta$ is the hardcore parameter (parameterizing the finiteness of the energy function), while $\theta$ is the smooth interaction parameter. For instance, in Model 2: $\beta=(\epsilon, \alpha)$, $\theta=\theta$; in Model 3: $\beta=(\epsilon, \alpha, B)$, $\theta=\theta$. This distinction is  presented clearly in \ref{theorieestimation}, where some theoretical results for  the asymptotic consistency of our estimators are given.

We consider a classical two-step estimation procedure. We first estimate $\beta$,  then we estimate $\theta$ and $z$ by pseudo-likelihood where $\beta$ is replaced by its estimator. 

The choice of the pseudo-likelihood approach instead of the classical maximum likelihood estimator is mainly imposed by practical reasons. Indeed, maximum likelihood requires the estimation, by simulations, of an unknown normalizing constant. This approach demands to simulate several tessellations according to the model, which is extremely time-consuming in the situation when a hardcore interaction is involved (see previous section). Moreover, the pseudo-likelihood procedure has the advantage of being asymptotically consistent for a large class of models (see \cite{DL}), which has not been proved for the maximum likelihood estimator in such a general setting. However, when the hardcore interaction is not too strong, the maximum likelihood estimation may constitute a second step to refine a pseudo-likelihood approach.

There is a major difficulty to overcome in order to implement the pseudo-likelihood estimation in our case: the hardcore interactions are not necessarily hereditary. An interaction is hereditary if, for every forbidden point pattern $\g$, then, for every point $x$, the configuration $\g+x$ remains forbidden. This is equivalent to: for every allowed point configuration $\g$, then for every point $x\in\g$, the configuration $\g-x$ remains allowed. In other words, an interaction is hereditary if one can remove any point from $\g$. This property concerns only the hardcore interaction. So every interaction involving no hardcore part is necessarily hereditary. The models presented in Section \ref{troisexemples} are not hereditary. Indeed, if one removes a point from an allowed tessellation, the new tessellation may contain cells that are too large (for instance). As a consequence, we must modify the classical pseudo-likelihood contrast to take into account the so-called removable points 	as introduced in \cite{DL} (see Definition \ref{removable}).

\subsection{The two-step procedure}
The first step consists in estimating the hardcore parameter $\beta$. Let us first assume that $\beta$ is a one-dimensional parameter. We suppose the following inclusion 
\begin{equation}\label{inclusion1}
\textrm{if }\beta<\beta' \textrm{ then }\forall\L,\ E^{\beta, \theta}_\L(\g_\L,\g_{\L^c})<+\infty \Rightarrow E^{\beta', \theta}_\L(\g_\L,\g_{\L^c})<+\infty. 
\end{equation}
In this case, a consistent estimator of $\beta$ is 
\begin{equation}\label{esthardcore}
\hat \beta=\inf\{\beta>0,\ E^{\beta, \theta}_{\L_n}(\g_{\L_n},\g_{\L_n^c})<+\infty\}.\end{equation}
If instead of (\ref{inclusion1}), the converse implication holds, then it suffices to replace the infimum by a supremum in (\ref{esthardcore}).

In the case of a multi-dimensional hardcore parameter $\beta$, we estimate each of its components as above.

For instance, for Models 2 and 3 presented in Section \ref{troisexemples},  Property (\ref{inclusion1}) is satisfied by the hardcore parameters $\alpha$ and $B$, while the converse holds for  $\epsilon$. As a consequence, following (\ref{esthardcore}), natural estimators for these examples are (the notations are the same as in \ref{troisexemples}):
\begin{itemize}
\item  {\it For Model 2}: \begin{align*} \hat \epsilon&=min\{l(T),\ T\in \del_{\L_n}(\g)\},\\ \hat \alpha&=max\{R(T),\ T\in \del_{\L_n}(\g)\}.\end{align*}

\item  {\it For Model 3}: \begin{align*} \hat \epsilon&=min\{h_{\min}(C),\ C\in \vor_{\L_n}(\g)\}, \\ \hat \alpha&=max\{h_{\max}(C),\ C\in \vor_{\L_n}(\g)\}, \\ \hat B&=max\{h_{\max}^2(C)/\vol(C),\ C\in \vor_{\L_n}(\g)\}.\end{align*}

\end{itemize}

The second step consists in estimating the smooth interaction parameter $\theta$ and the intensity parameter $z$. We use the pseudo-likelihood procedure for the reasons explained before. To deal with the non-hereditary problem, we must introduce the concept of removable points.
\begin{definition}\label{removable}   
Let $\g$ be in $\Md$ and $x$ be a point of $\g$, then $x$ is removable from $\g$ if there exists  $\L\in\Bd$ such that $x\in\L$ and 
\begin{equation}    
E^{\beta, \theta}_\L(\g_\L-x,\g_{\L^c})<+\infty.
\end{equation}
\end{definition}

The following proposition, proved in \cite{DL}, gives a more intuitive approach and justifies the name of removable points.

\begin{proposition}\label{equivalence}
Let $\g$ be in $\Mdi$ and $x$ be a point of $\g$, then $x$ is removable from $\g$ if and only if $\g-x$ is in $\Mdi$. 
\end{proposition}
From the definition, it is clear that the property of being a removable point from $\g$ depends only on the hardcore parameter $\beta$ and not on $\theta$ or $z$. Thus, we denote by $\rem^{\beta} (\g)$ the set of removable points in $\g$.

The more rigid the tessellation, the less removable points there are. In particular, if there is no hardcore part in the interaction function, every point of $\g$ is removable. In Figure \ref{figrem}, the removable points of previous simulations are encircled.

  \begin{figure}[htbp]
  \begin{center}
  \hspace{0cm} \includegraphics[angle=0,scale=.233]{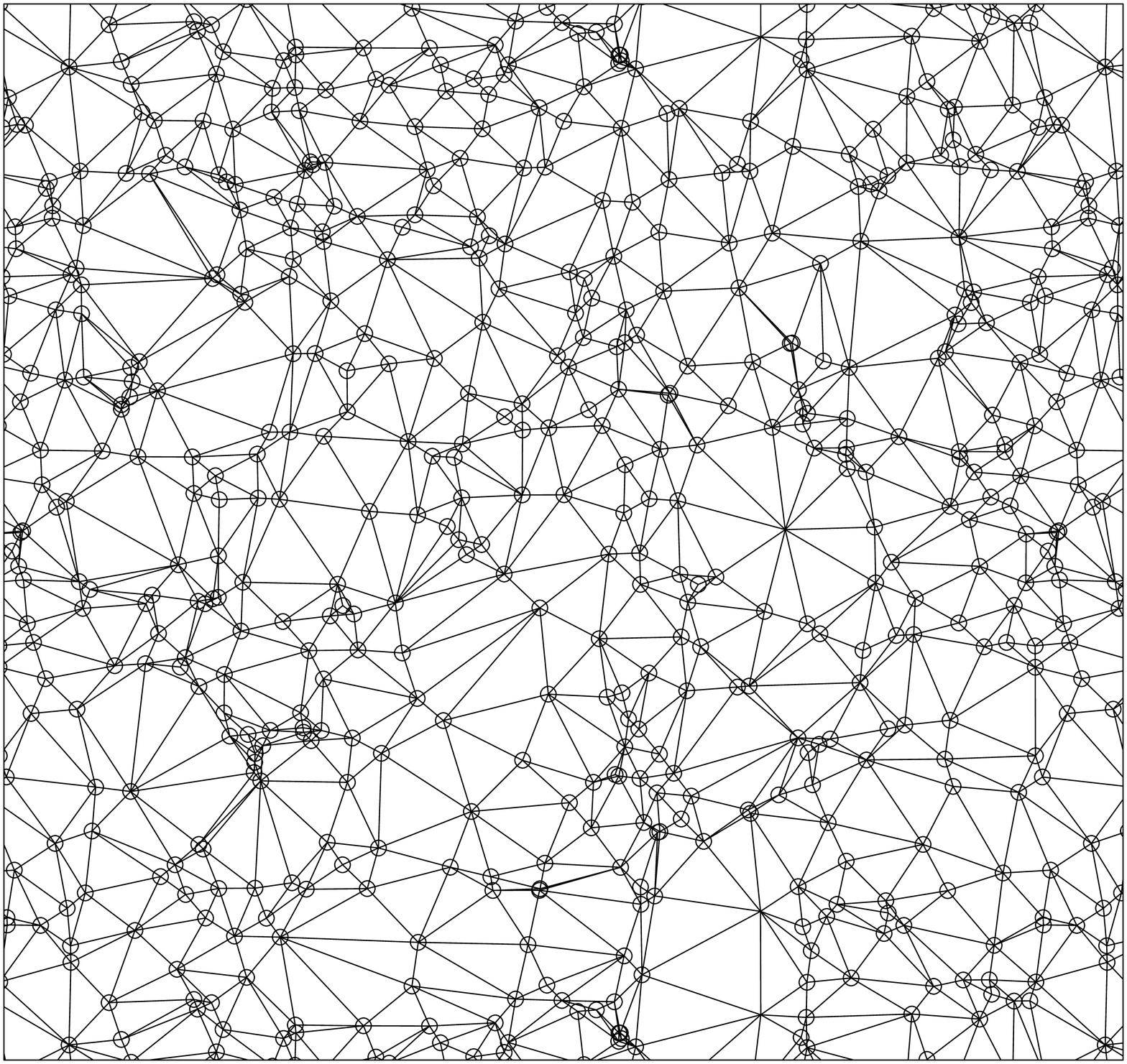}
  \hspace{0.4cm} \includegraphics[angle=0,scale=.233]{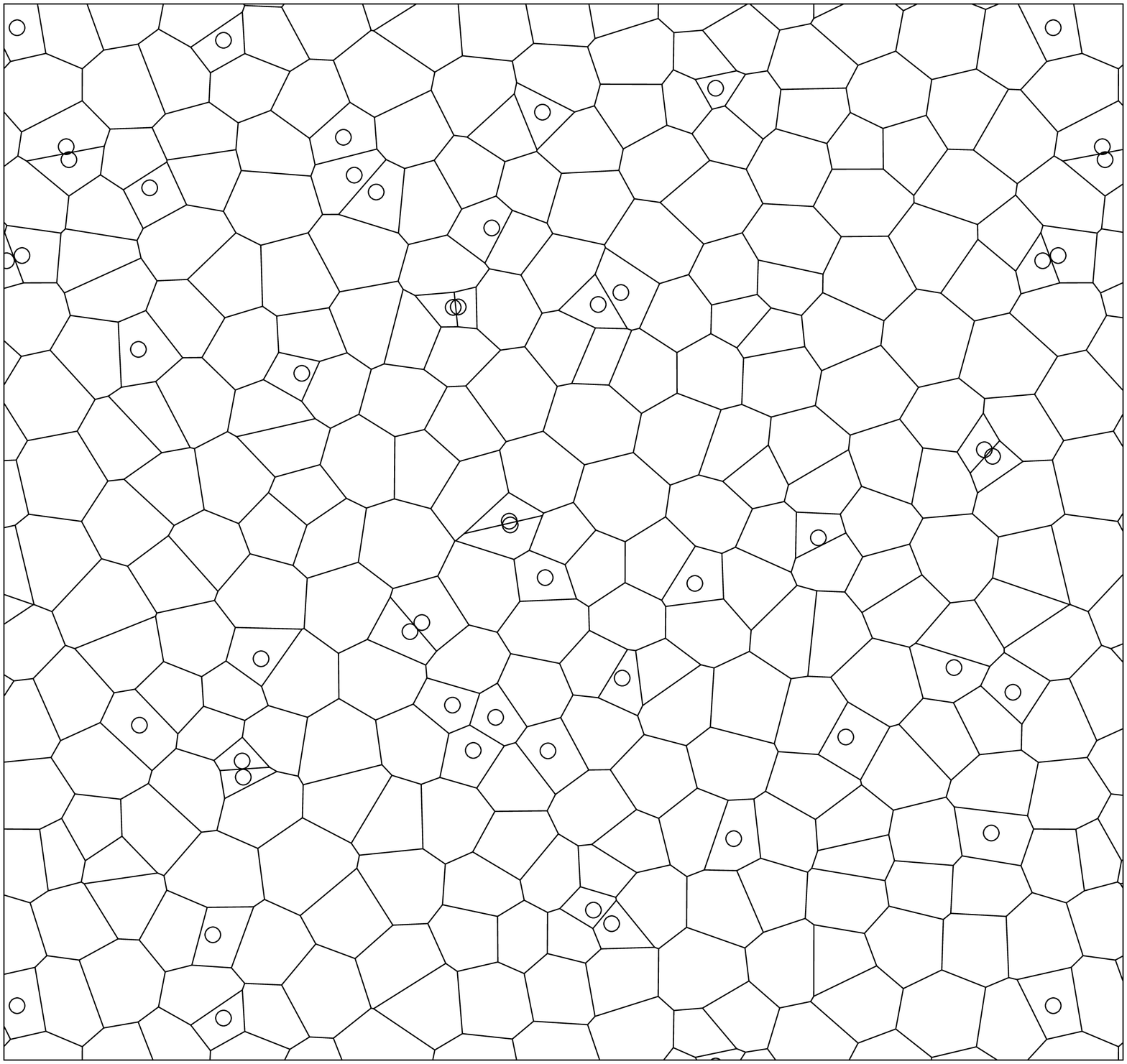}
  \caption{{\small Removable points (encircled) from: The Delaunay tessellation simulated in Figure \ref{model2} where $\theta=5$ (left);  The Voronoi tessellation simulated in Figure \ref{model3} where $\theta=-0.5$ (right).}}\label{figrem}
  \end{center}
  \end{figure}

We are now in position to introduce the pseudo-likelihood contrast function, adapted to the non-hereditary case:

\begin{equation*}
PLL_{\L_n}(\gamma,z,\beta, \theta)= \\
\int_{\Lambda_n}z \exp\left(-h^{\beta,\theta}(x,\gamma)\right)dx + \sum_{x\in\rem^{\beta}(\g)\cap\L_n}\big(h^{\beta, \theta}(x,\gamma-x)-\ln(z)\big),
\end{equation*}
where $h^{\beta, \theta}(x,\gamma-x)$ is the local energy of $x$ in $\g$, defined for every $x\in\rem^{\beta}(\g)$ by 
\begin{equation}\label{defh}
h^{\beta, \theta}(x,\gamma-x)=E^{\beta, \theta}_\L(\g_\L,\g_{\L^c})-E^{\beta, \theta}_\L(\g_\L-x,\g_{\L^c}),
\end{equation}
where $\L$ is a set containing $x$ as in Definition \ref{removable}. Let us point out that $h^{\beta, \theta}(x,\gamma)$ is just equal to $h^{\beta, \theta}(x,(\gamma+x) - x)$ and is always well-defined for $\g$ in $\Mdi$.

The parameters $\theta$ and $z$ are estimated by minimizing $PLL_{\L_n}$, where the hardcore parameter $\beta$ is replaced by its estimator $\hat \beta$ obtained in the first step:
\begin{equation} \label{estsmooth}
(\hat z,\ \hat \theta)=argmin_{z,\theta} PLL_{\L_n}(\gamma,z,\hat \beta, \theta).
\end{equation}

The consistency of this estimation procedure is considered in \ref{theorieestimation}.

\subsection{Practical implementation}

The optimization of  $PLL_{\L_n}$ requires the calculus of the local energy $h^{\beta, \theta}(x,\gamma)$ for any $x\in\Lambda_n$. This is the same calculus as the one needed in step 3 of the algorithm presented in Section \ref{algorithm} and, as explained there, it can be achieved by focusing on a window around $x$. Moreover, this computation requires the knowledge of $\g_{\L^c}$, the configuration outside this window. To prevent boundary problems, it is actually necessary to compute $PLL$ on a sub-window of the initial observation window $\L_n$. We denote abusively $\L_n$ this sub-window in the following.

The derivative of  $PLL_{\L_n}$ with respect to $z$ yields the following estimator for $z$:
\begin{equation}\label{estz}
\hat z = \frac{\int_{\Lambda_n} \exp\left(-h^{\hat \beta,\theta}(x,\gamma)\right)dx}{card(\rem^{\hat \beta}(\g)\cap\L_n)},
\end{equation}
where $card(\rem^{\hat \beta}(\g)\cap\L_n)$ is the number of removable points from the observed point pattern $\g$ in $\L_n$.

In the simple case where $\theta$ is a one-dimensional parameter and  $h^{\beta, \theta}(x,\gamma)$ is a sufficiently regular function, the minimization of   $PLL_{\L_n}$ in $\theta$ can be reduced to the determination of the root of an equation. Indeed, from the derivative of $PLL_{\L_n}$ with respect to $\theta$, we obtain in this case that $\hat \theta$ is the solution of
\begin{equation}\label{esttheta}
z \int_{\Lambda'_n} \frac{\partial h^{\hat \beta,\theta}}{\partial \theta} (x,\gamma) \exp\left(-h^{\hat \beta,\theta}(x,\gamma)\right)dx = \sum_{x\in\rem^{\hat \beta}(\g)\cap\L_n} \frac{\partial  h^{\hat \beta, \theta}}{\partial \theta}(x,\gamma-x),
\end{equation}
where $\L'_n = \{x\in\L_n,\  \g+x\in\Mdi\}$.

Moreover, when $h^{\beta, \theta}(x,\gamma)$ depends linearly on $\theta$ (as in Models 2 and 3), for all $x$ such that $\g+x\in\Mdi$, $$\frac{\partial h^{\hat \beta,\theta}}{\partial \theta} (x,\gamma)=\theta h^{\hat \beta,1}(x,\g),$$
which simplifies equation (\ref{esttheta}) above.

From a practical point of view, we first estimate $\theta$ thanks to (\ref{esttheta}), where $z$ is replaced by $\hat z$ given by (\ref{estz}). Then we deduce $\hat z$ by plugging $\hat \theta$ into (\ref{estz}).
In both these estimations, the involved integrals are approximated by Monte Carlo (this is the most time-consuming step of the estimation procedure).

\subsection{Some examples}

\subsubsection{For Model 2}

We implement the estimation procedure on simulations of Model 2. We do not introduce the hardcore parameter $\epsilon$ here. The estimation of $\alpha$, $\theta$ and $z$ has been done from 200 replications of Model 2 when $\alpha=0.08$, $z=1000$ and $\theta=\pm 5$, simulated as in Section \ref{algorithm}. The results are shown in Figure \ref{estmodel2-theta_5} and \ref{estmodel2-theta5}. We have distinguished two cases: first estimating $\theta$ by supposing $z=1000$ known, then estimating both $\theta$ and $z$. In this last case, one can note in the bottom right plot of these figures the closed relation between $\hat z$ and $\hat \theta$. Although the models that we consider are well identifiable, it is not surprising to observe this closed relation: it is implied  by the Euler's formula, which connects linearly the number of cells and the number of vertices in a tessellation (see 3.2.11 in \cite{M}). Therefore, if $z$ is under-estimated, $\theta$ will tend to be under-estimated as well, in order to respect this linear relation.

When $\theta=-5$ (Figure \ref{estmodel2-theta_5}), the simulated tessellations rely on about 1500 points and all of them are removable. The estimation of $\alpha$ when $\theta=-5$ actually shows that this hardcore parameter is useless in this case: the cells of the tessellation naturally satisfy the hardcore condition. It is interesting to note that this misspecification does not affect the estimation of the smooth interaction parameter $\theta$. The average of $\hat \theta$ is about $-5$, while its standard deviation is $0.4$ when $z$ is known and $1.6$ when $z$ is estimated. The average of $\hat z$ is $1002$ and its standard deviation $145$. 

  \begin{figure}[htbp]
    \setlength{\tabcolsep}{0.1cm} \centerline{
  \begin{tabular}[]{ccc}
  \includegraphics[angle=0,scale=.16]{Figures/perim-theta_5} &
  \includegraphics[angle=0,scale=.22]{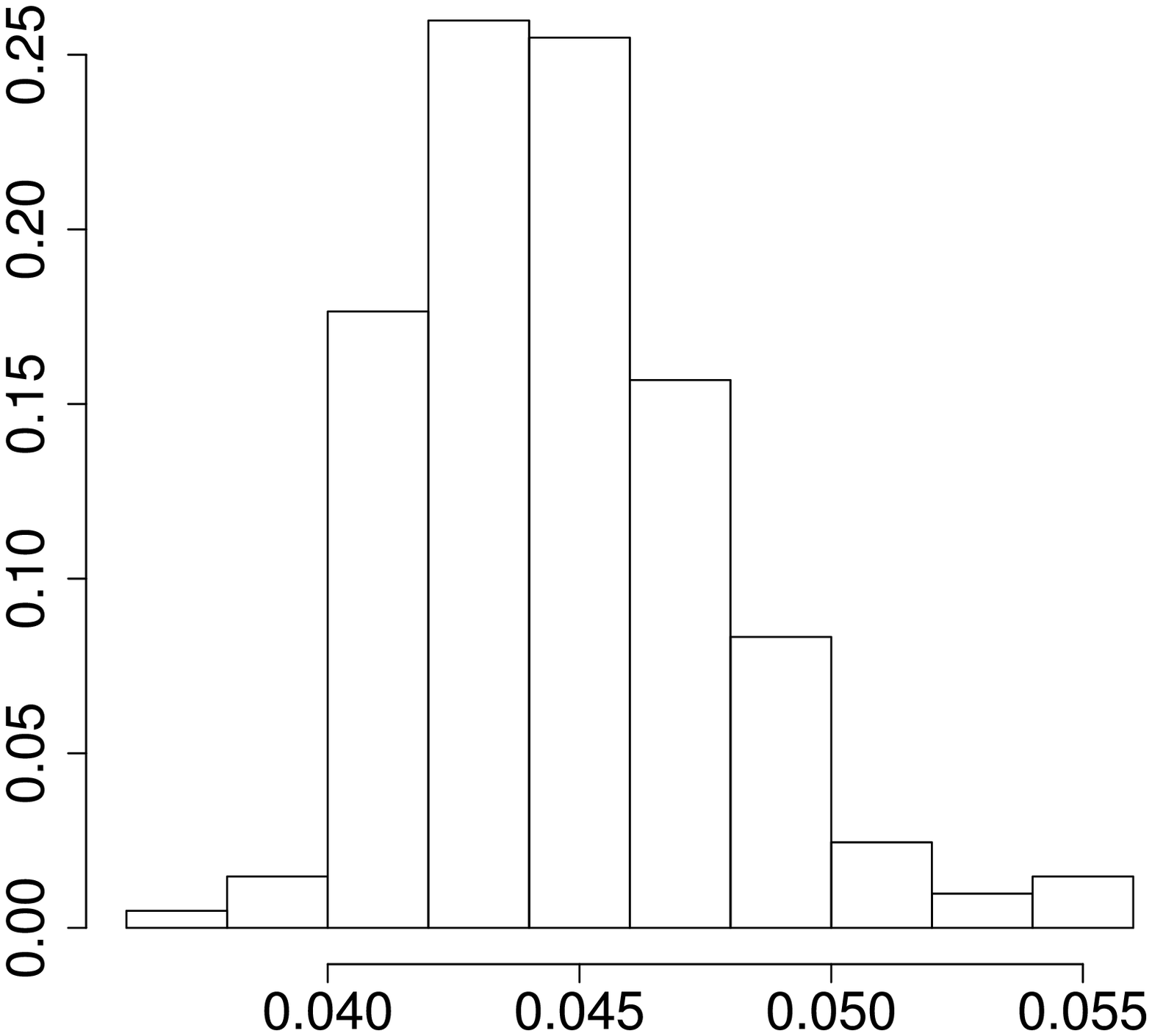} &
  \includegraphics[angle=0,scale=.22]{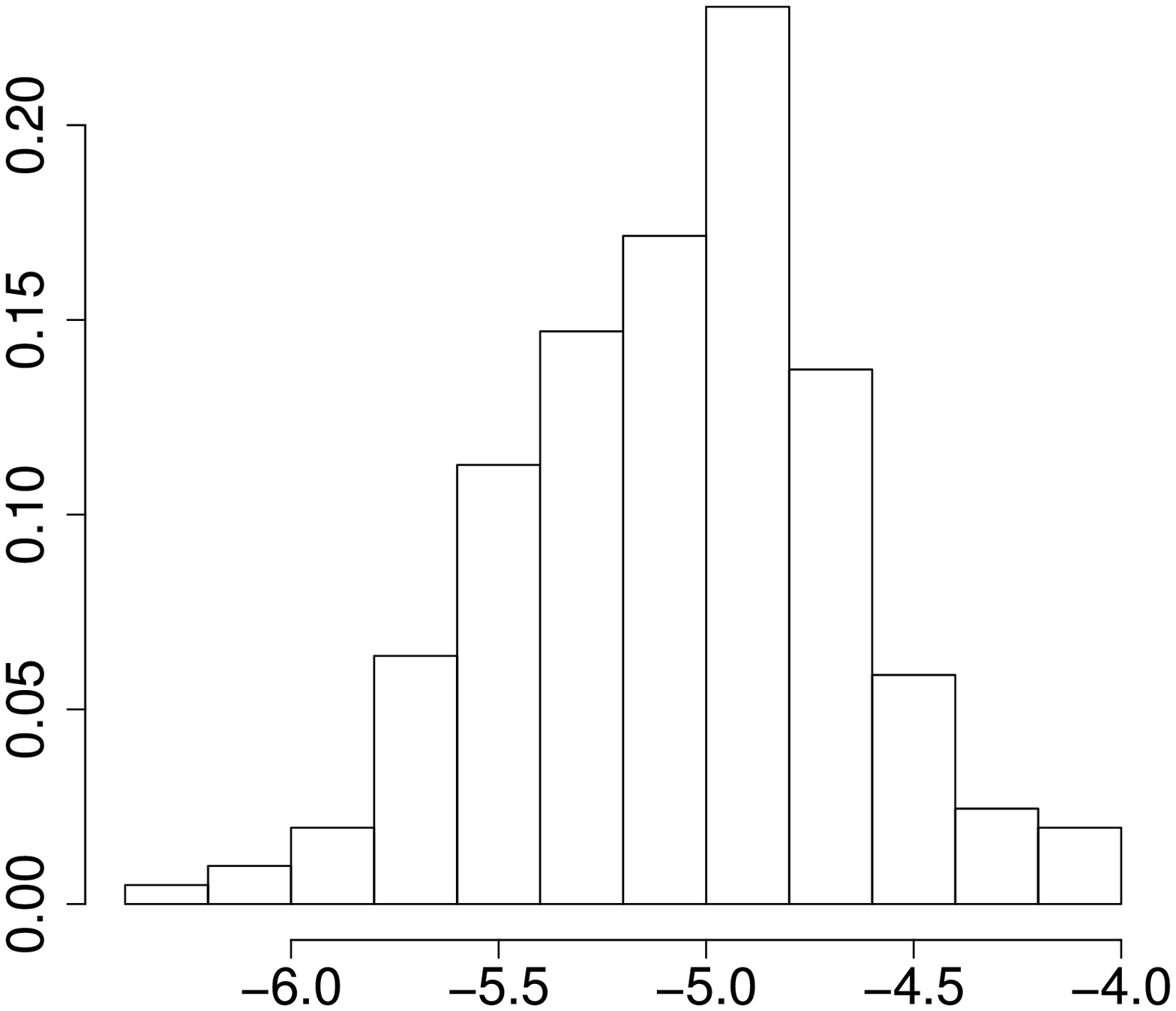} \\
  \footnotesize {Example of tessellation} &
  \footnotesize {$\hat \alpha$} & \footnotesize {$\hat \theta$ when $z$ is known}\\
  \includegraphics[angle=0,scale=.22]{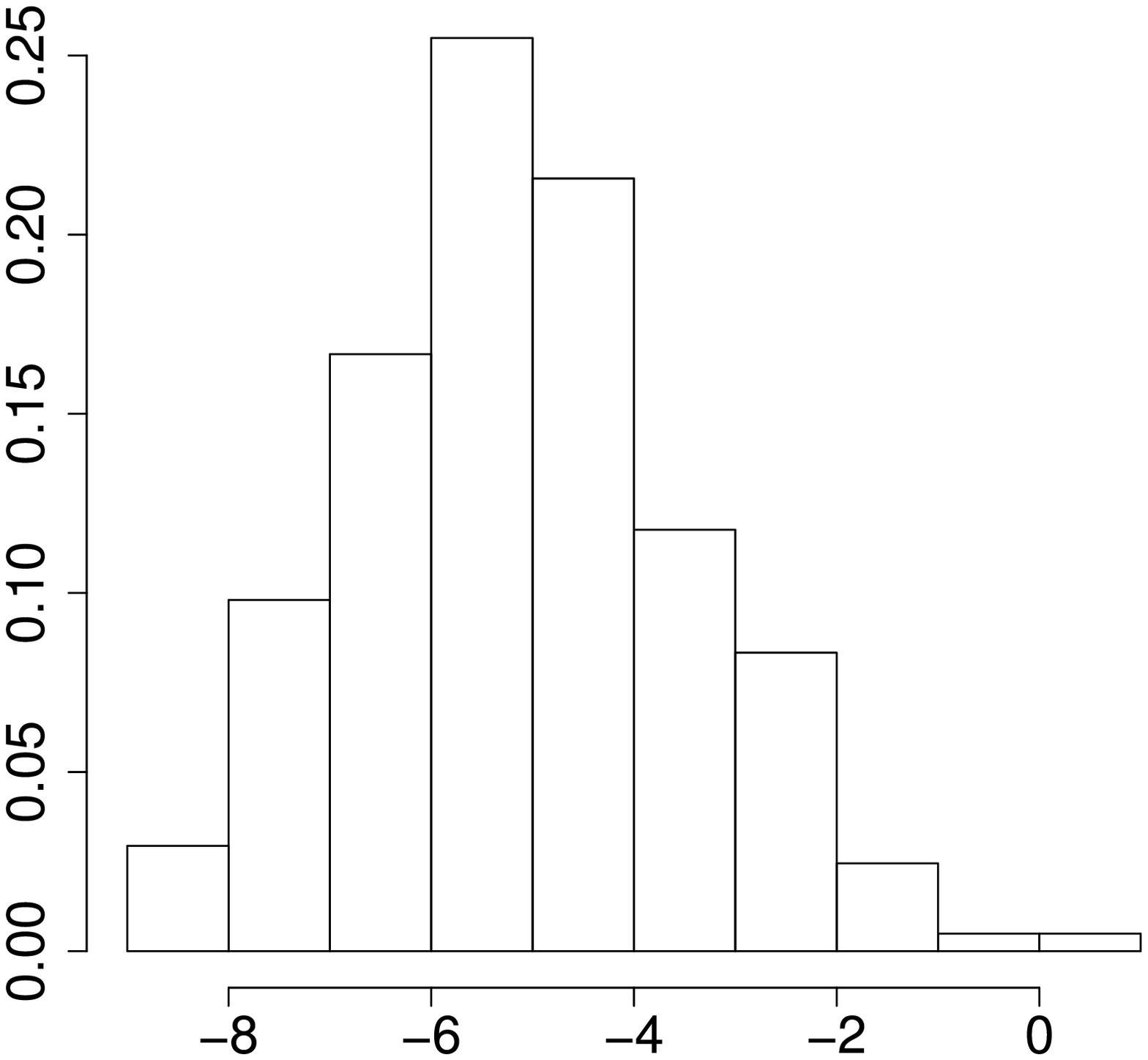}&
   \includegraphics[angle=0,scale=.22]{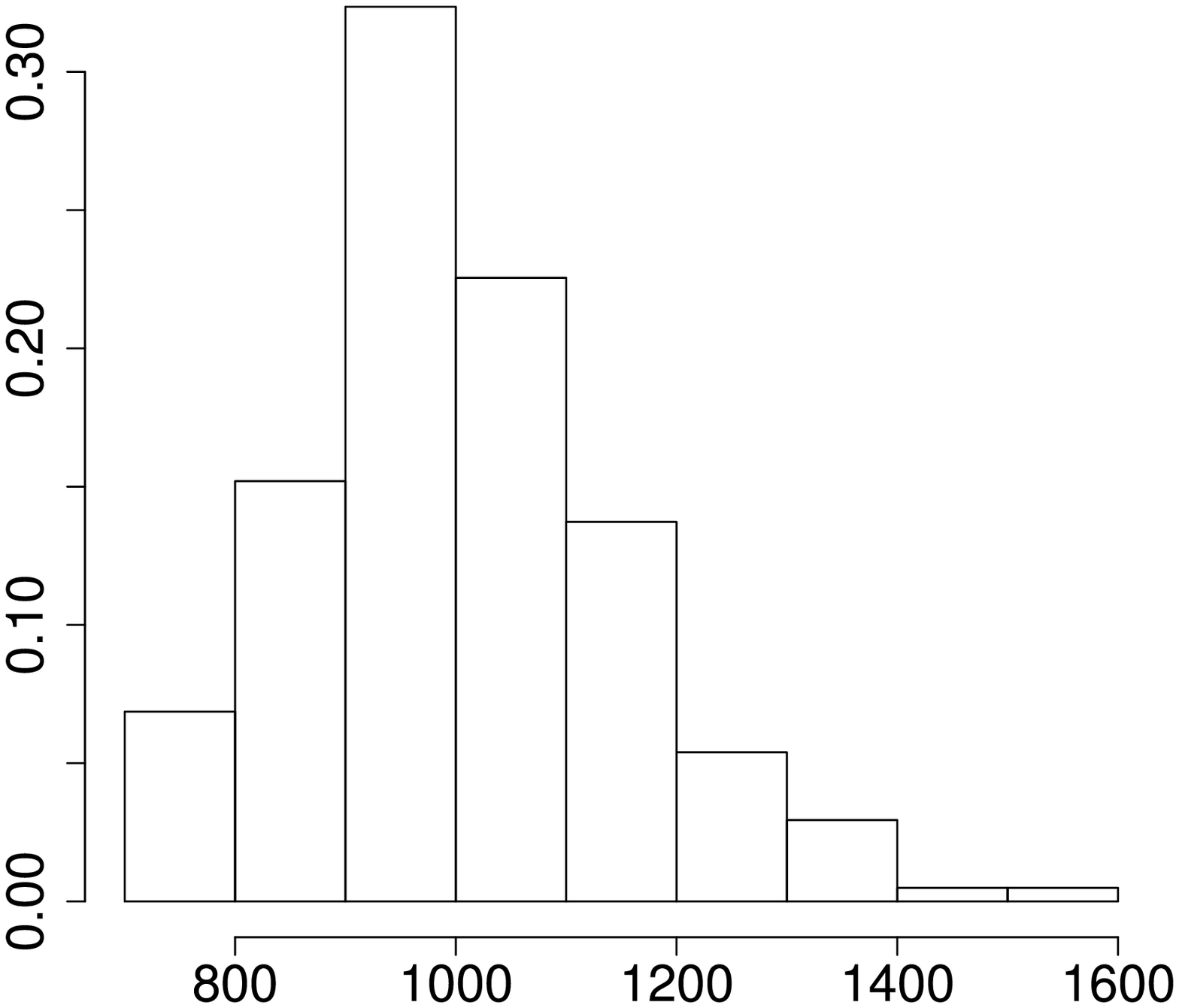}&
   \includegraphics[angle=0,scale=.22]{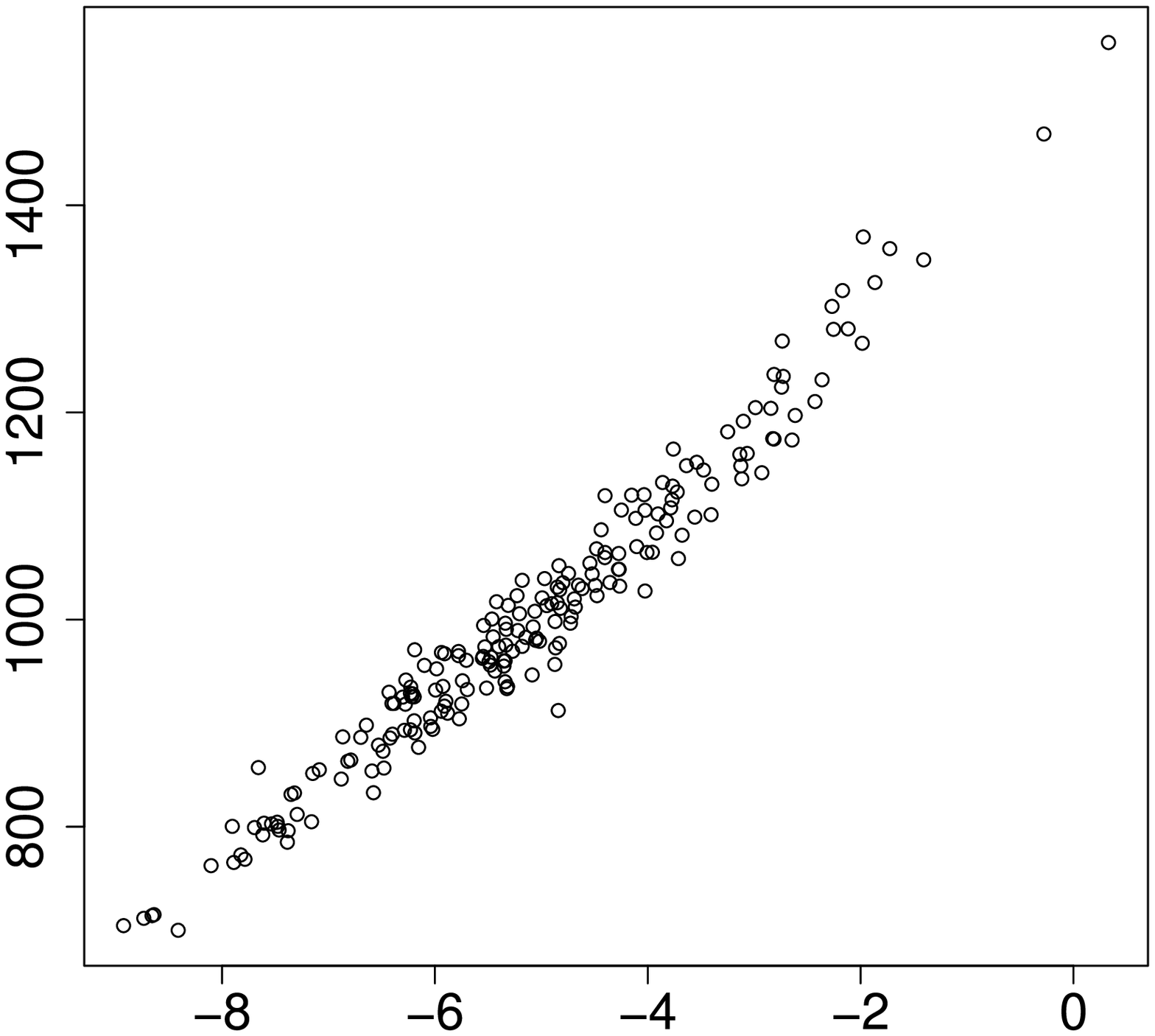}\\
  \footnotesize {$\hat \theta$ when $z$ is estimated} & \footnotesize {$\hat z$} & 
  \footnotesize {Scatterplot of $(\hat\theta,\ \hat z)$} 
  \end{tabular}
  }
     \caption{{\small Estimation of Model 2 when $\alpha=0.08$, $\theta=-5$, $z=1000$, from 200 replications.}}\label{estmodel2-theta_5}

  \end{figure}

  \begin{figure}[htbp]
    \setlength{\tabcolsep}{0.1cm} \centerline{
  \begin{tabular}[]{ccc}
  \includegraphics[angle=0,scale=.16]{Figures/perim-theta5} &
  \includegraphics[angle=0,scale=.22]{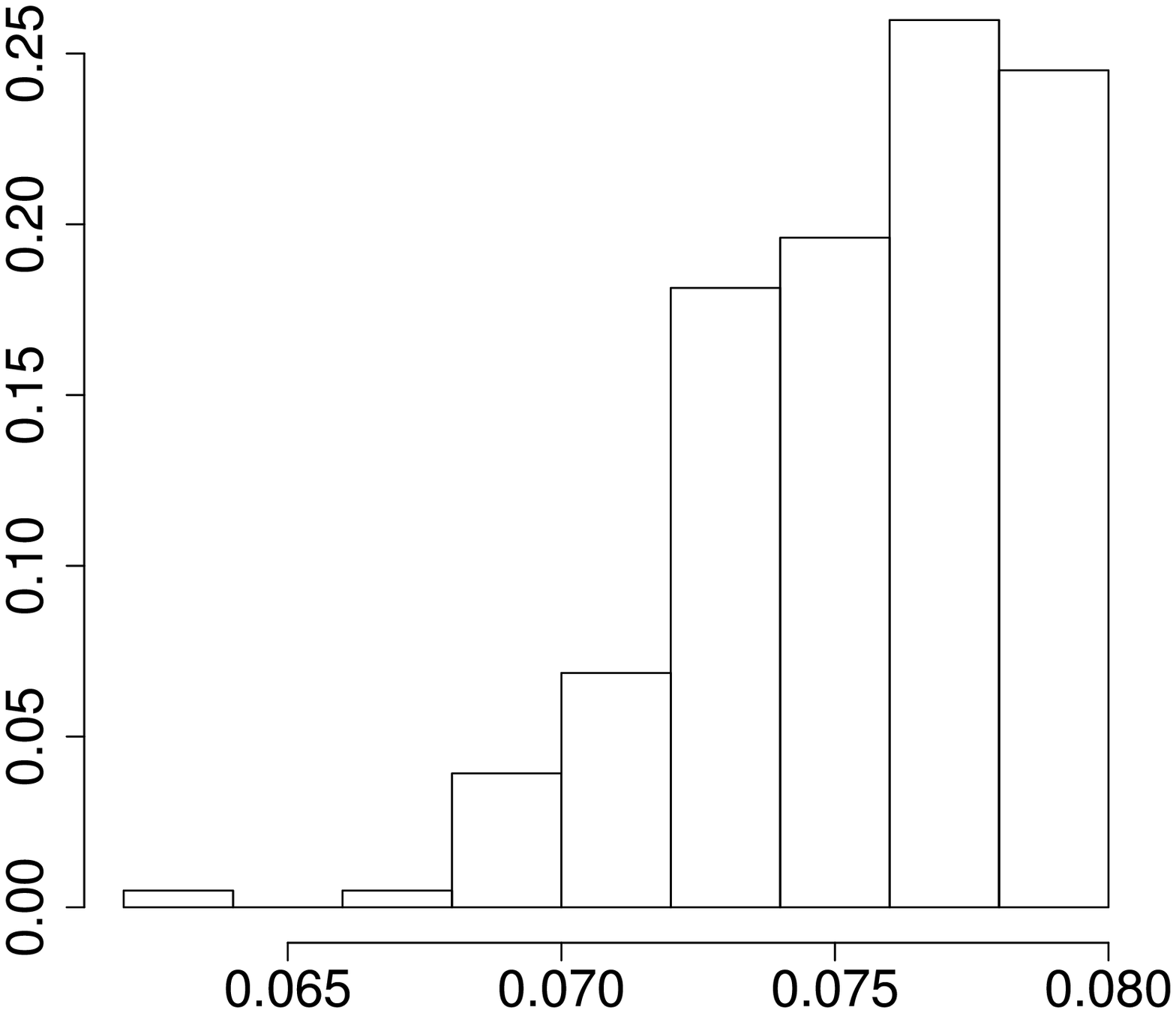} &
  \includegraphics[angle=0,scale=.22]{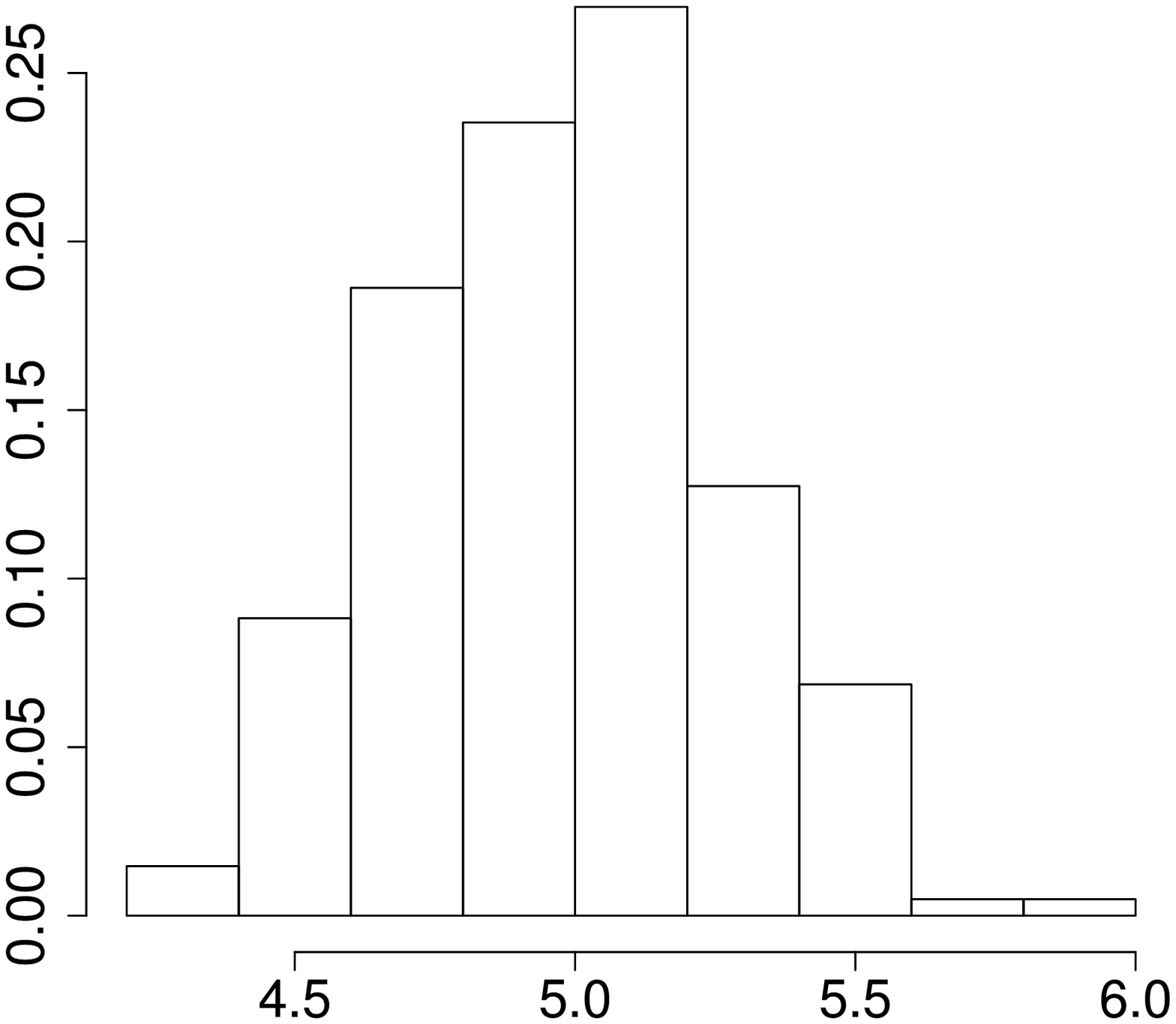} \\
  \footnotesize {Example of tessellation} &
  \footnotesize {$\hat \alpha$} & \footnotesize {$\hat \theta$ when $z$ is known}\\
  \includegraphics[angle=0,scale=.22]{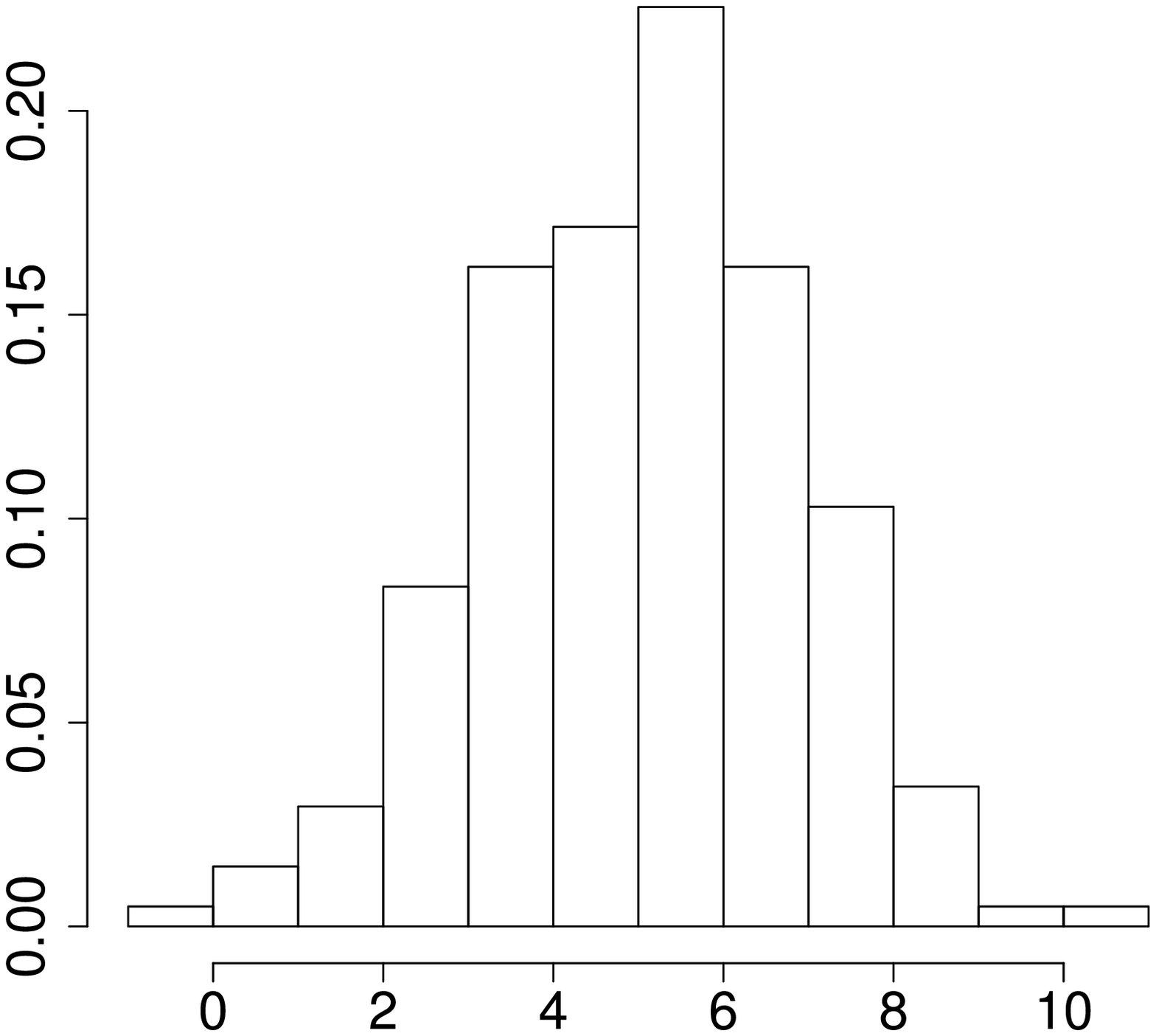}&
   \includegraphics[angle=0,scale=.22]{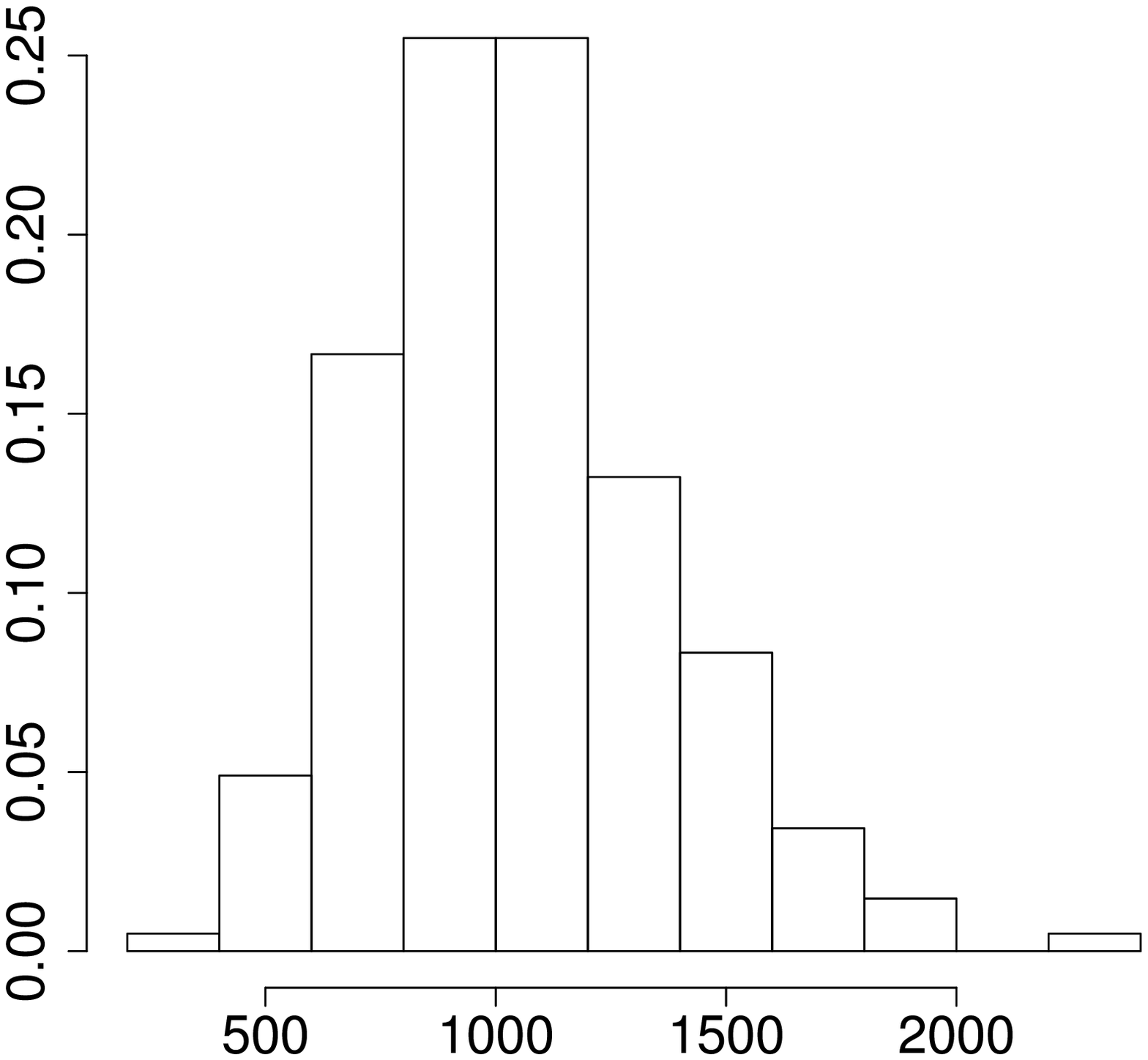}&
   \includegraphics[angle=0,scale=.22]{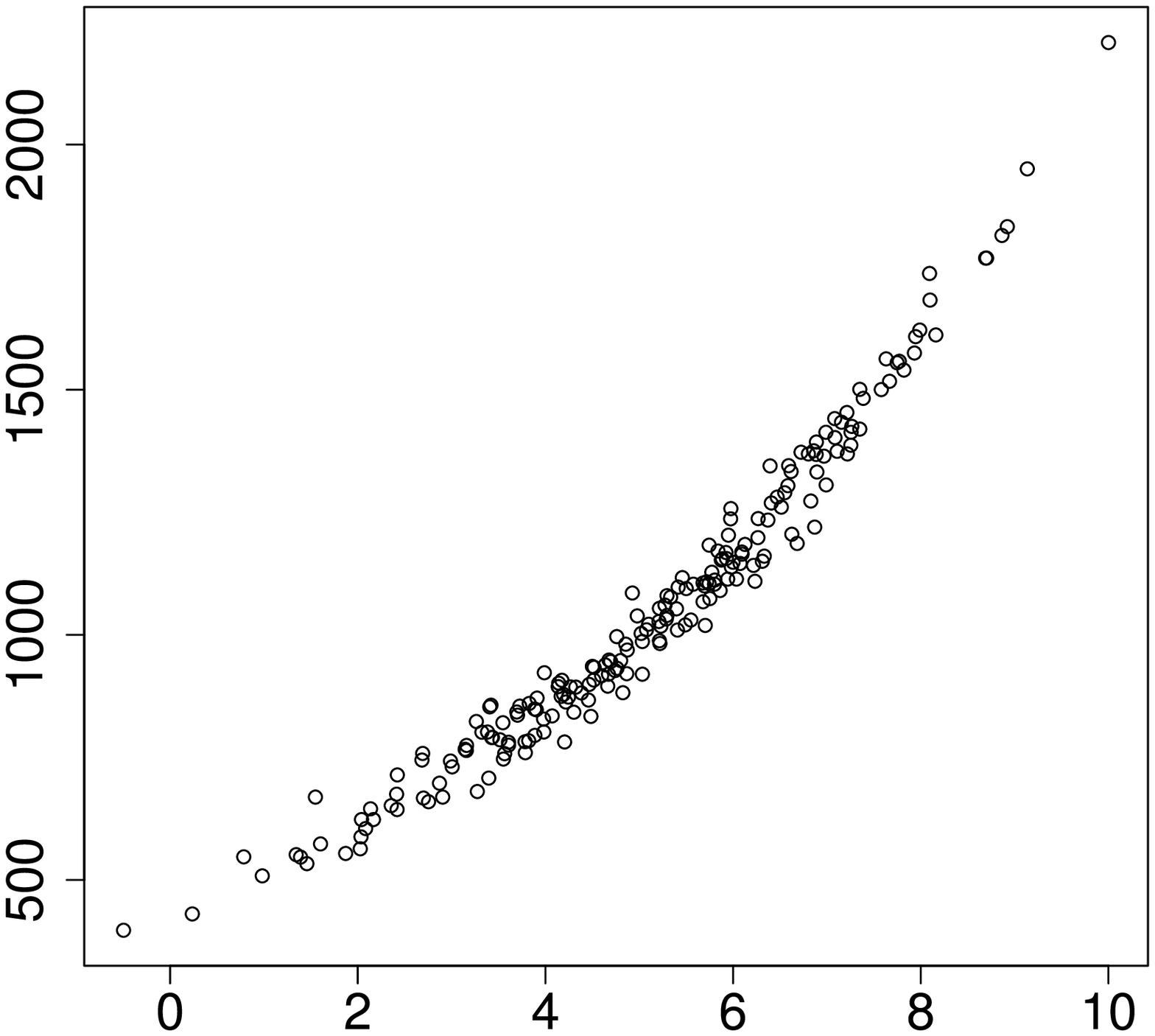}\\
  \footnotesize {$\hat \theta$ when $z$ is estimated} & \footnotesize {$\hat z$} & 
  \footnotesize {Scatterplot of $(\hat\theta,\ \hat z)$} 
  \end{tabular}
  }
      \caption{{\small Estimation of Model 2 when $\alpha=0.08$, $\theta=5$, $z=1000$, from 200 replications.}}\label{estmodel2-theta5}

  \end{figure}

When $\theta=5$ (Figure \ref{estmodel2-theta5}), the hardcore plays an important role in the model. It is well estimated with a standard deviation of $3.10^{-2}$. The standard deviation of $\hat \theta$ is $0.3$ when $z$ is known and $1.9$ when $z$ is estimated. The average of $\hat z$ is $1049$ and its standard deviation $313$. These estimations seem less accurate than when $\theta=-5$. This certainly comes from the fact that, when $\theta=5$, our simulated tessellations on $[0,1]\times[0,1]$ rely only on $500$ points. Most of these points are removable (more than $90\%$), as showed in the left example of Figure \ref{figrem}.

\subsubsection{For Model 3}

Two hundred replications of Model 3 where $\alpha=0.05$, $B=0.625$, $z=100$ and $\theta=\pm 0.5$ have been simulated according to the algorithm presented in Section \ref{algorithm} (see Figure \ref{model3} for an example). As above, the hardcore parameter $\epsilon$ was not introduced here. The results of the estimations are shown in Figure \ref{estmodel3-theta_5} when $\theta=-0.5$ and in Figure \ref{estmodel3-theta5} when $\theta=0.5$. Two situations are considered, assuming $z=100$ is known or not. The particularity of these simulated Voronoi tessellations is their rigidity. The hardcore interactions are strong, forcing the cells not to be too large (through $\alpha$) neither too flat (through $B$). This is confirmed by the accuracy of their estimation in both cases $\theta=\pm 0.5$ (see the histograms in Figures \ref{estmodel3-theta_5} and \ref{estmodel3-theta5}). But, as a consequence, there are only a few removable points, making the estimation of the smooth interaction parameters more difficult. Yet, it appears from these simulations that, in spite of the apparent similarity of the tessellations when $\theta=-0.5$ and $\theta=0.5$ (see Figure \ref{model3}) and in spite of the few number of removable points, the estimation procedure is mostly available to properly distinguish them.

When $\theta=-0.5$ (Figure  \ref{estmodel3-theta_5}), there are in average $45$ removable points on $265$ points. The estimation of $\theta$ remains correct: the average and the standard deviation of $\hat \theta$ are respectively $-0.52$ and $6.4$ {\small $10^{-2}$} when $z$ is known, and $-0.56$ and $14.5$ {\small $10^{-2}$} when $z$ is estimated. The average of $\hat z$ is $94$ while its standard deviation is $45$.

  \begin{figure}[h]
    \setlength{\tabcolsep}{0.1cm} \centerline{
  \begin{tabular}[]{ccc}
  \includegraphics[angle=0,scale=.22]{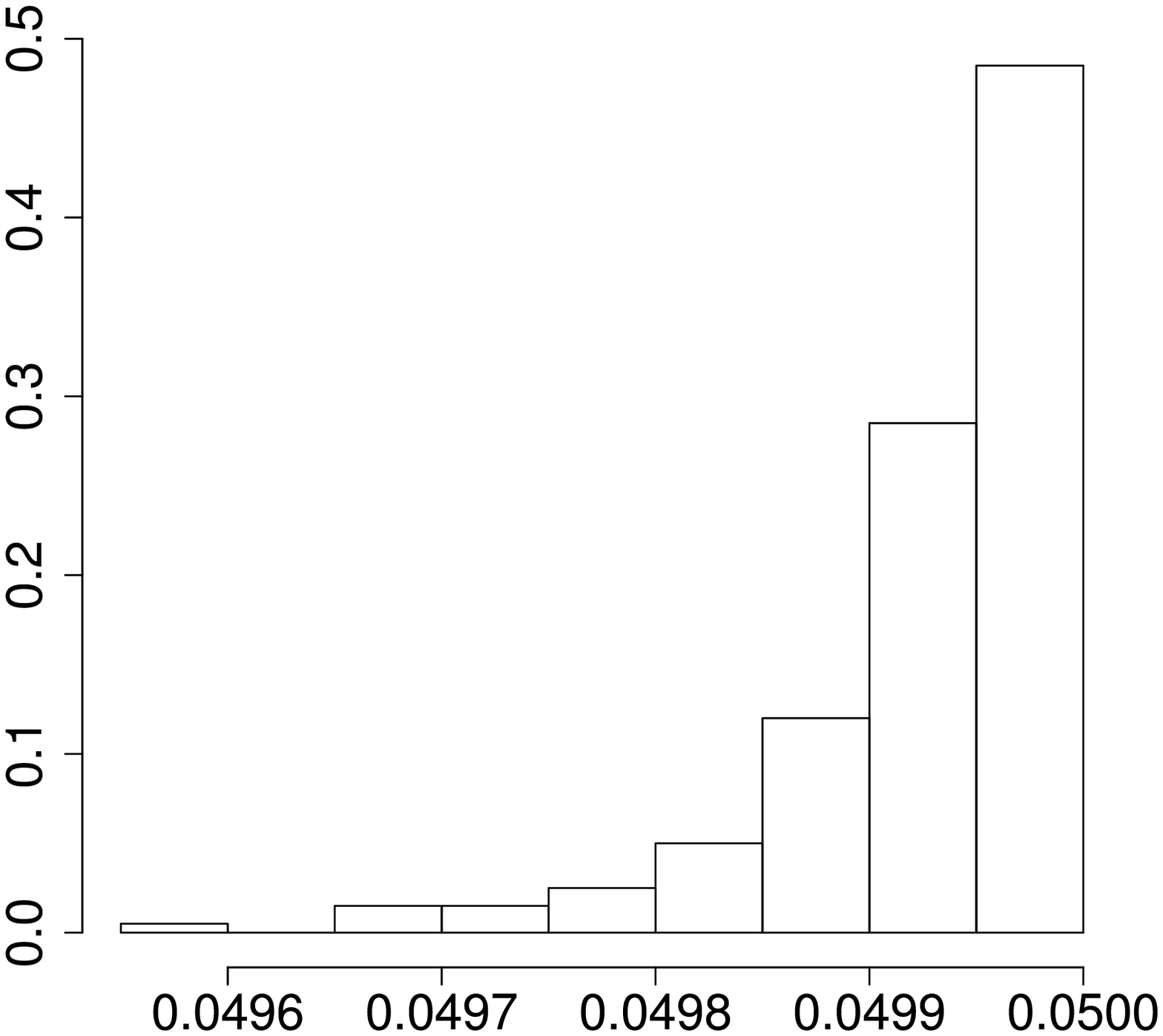} &
  \includegraphics[angle=0,scale=.22]{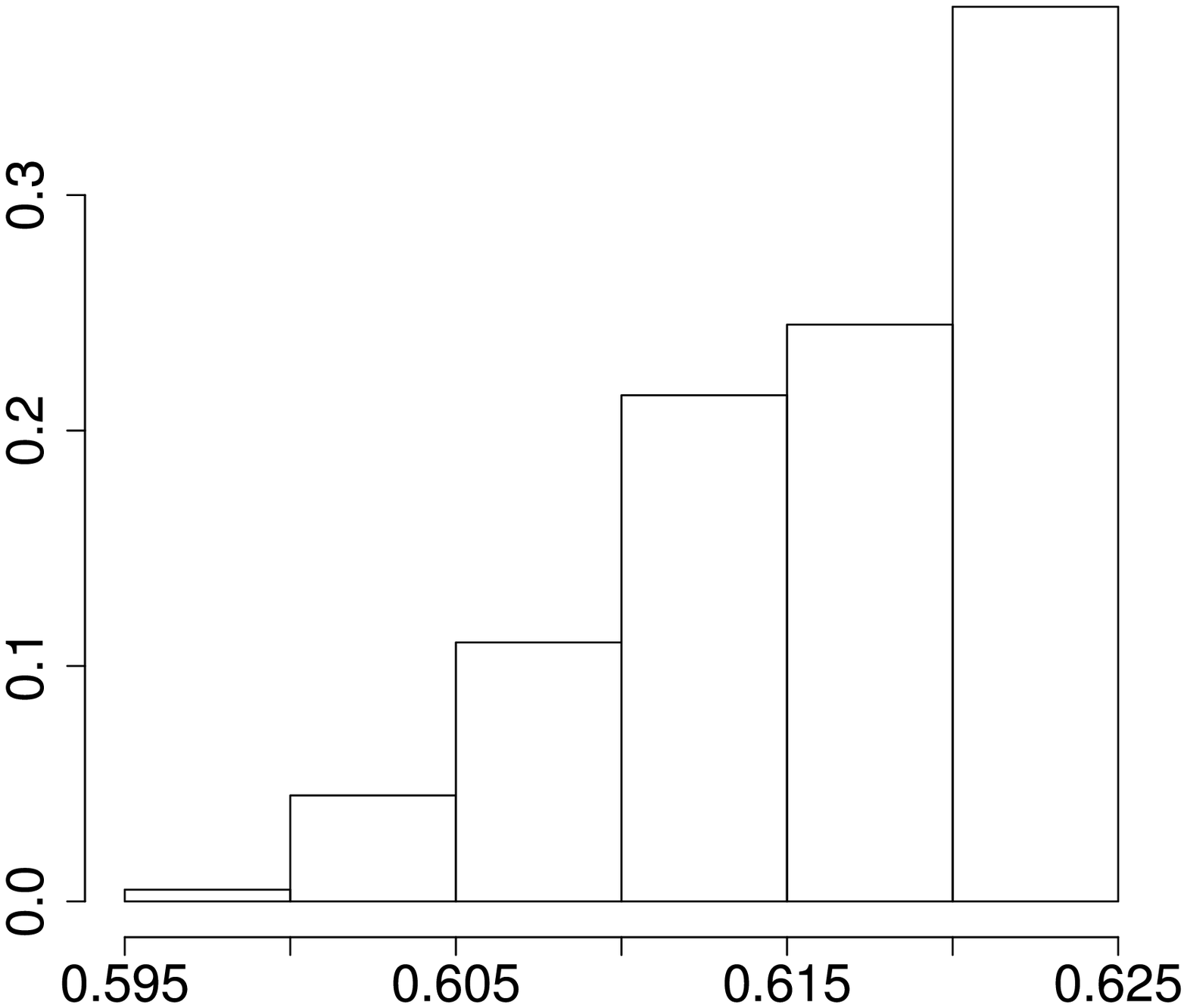} &
  \includegraphics[angle=0,scale=.22]{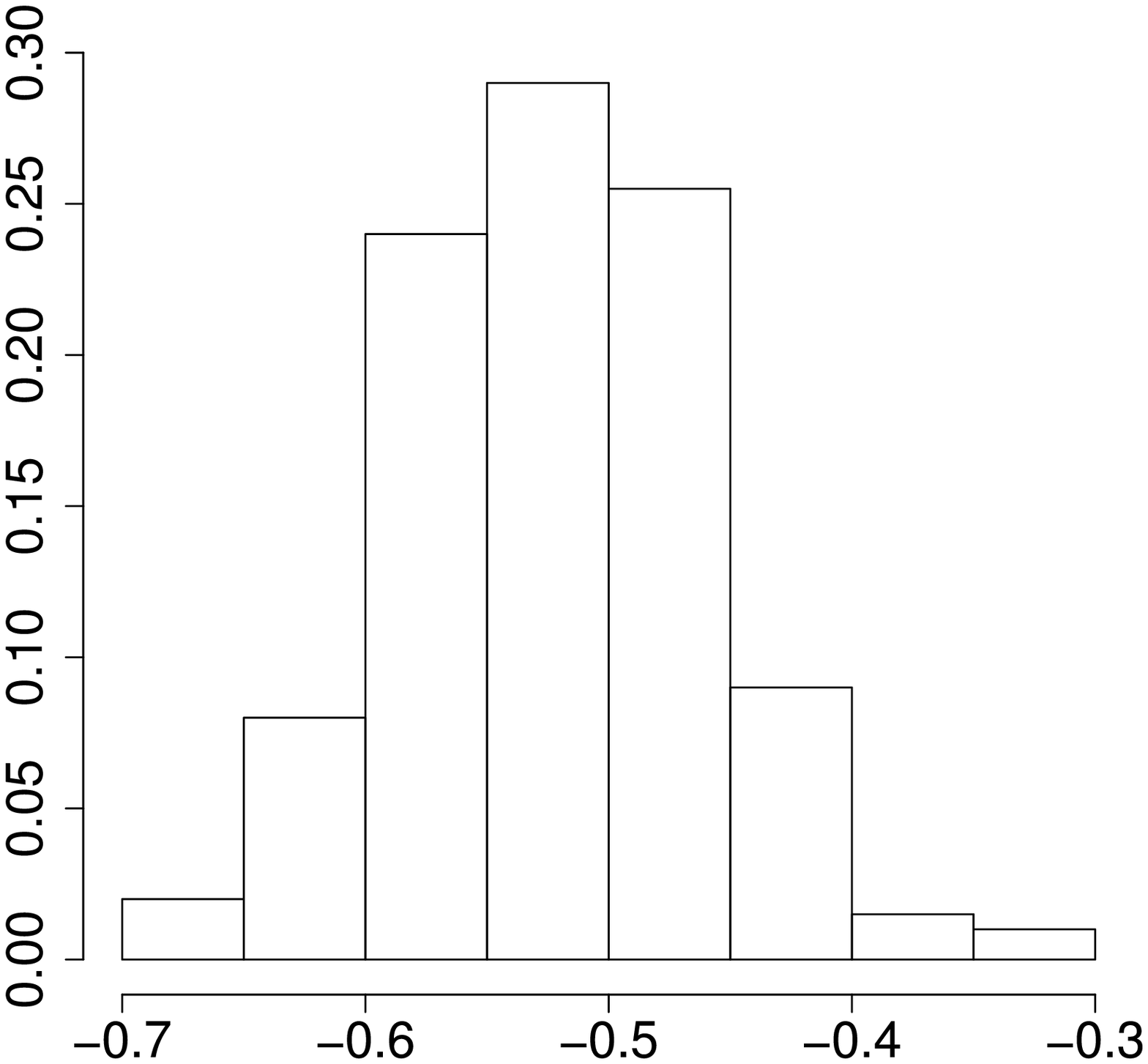} \\
  \footnotesize {$\hat \alpha$} &
  \footnotesize {$\hat B$} & \footnotesize {$\hat \theta$ when $z$ is known}\\
  \includegraphics[angle=0,scale=.22]{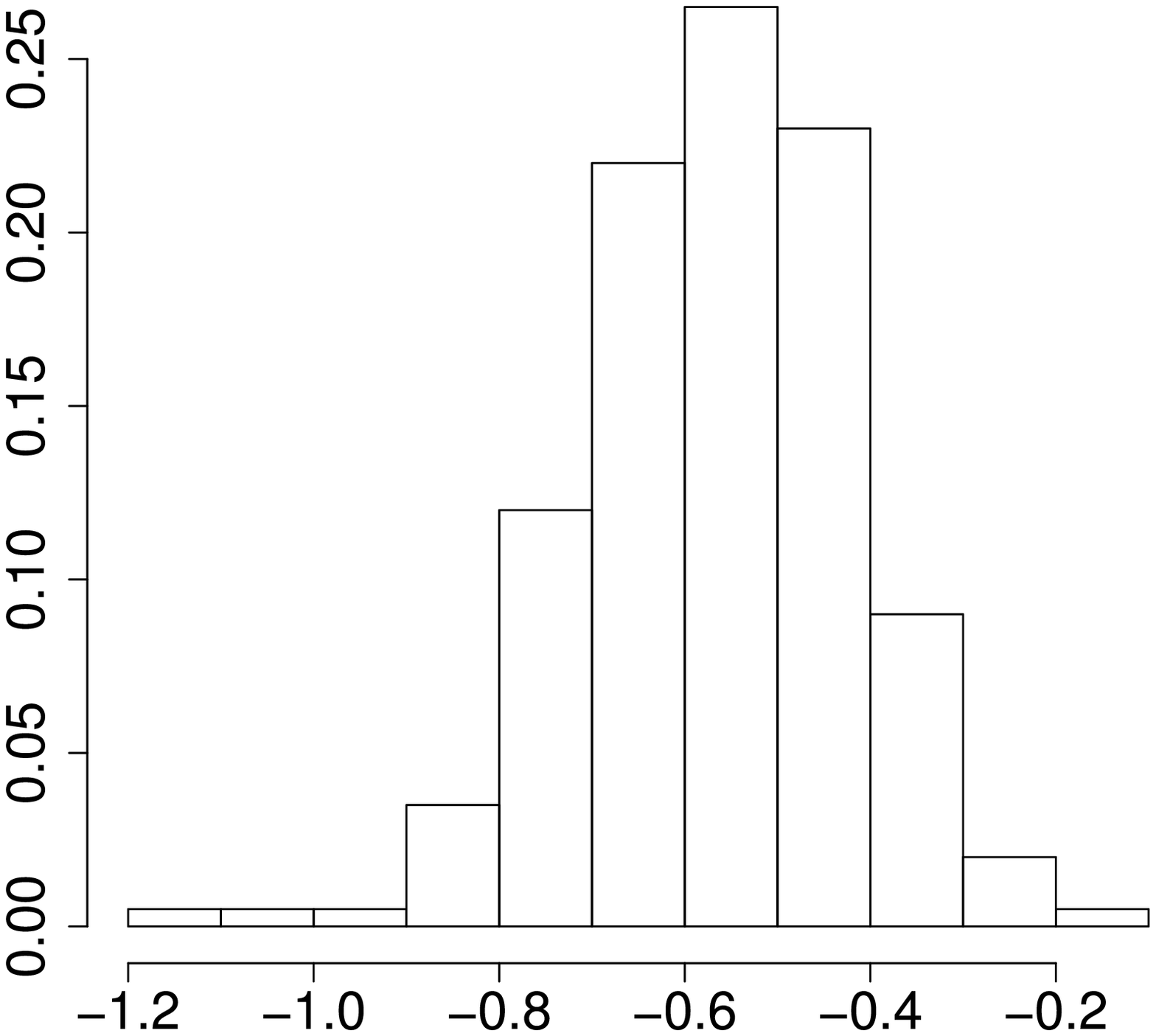}&
   \includegraphics[angle=0,scale=.22]{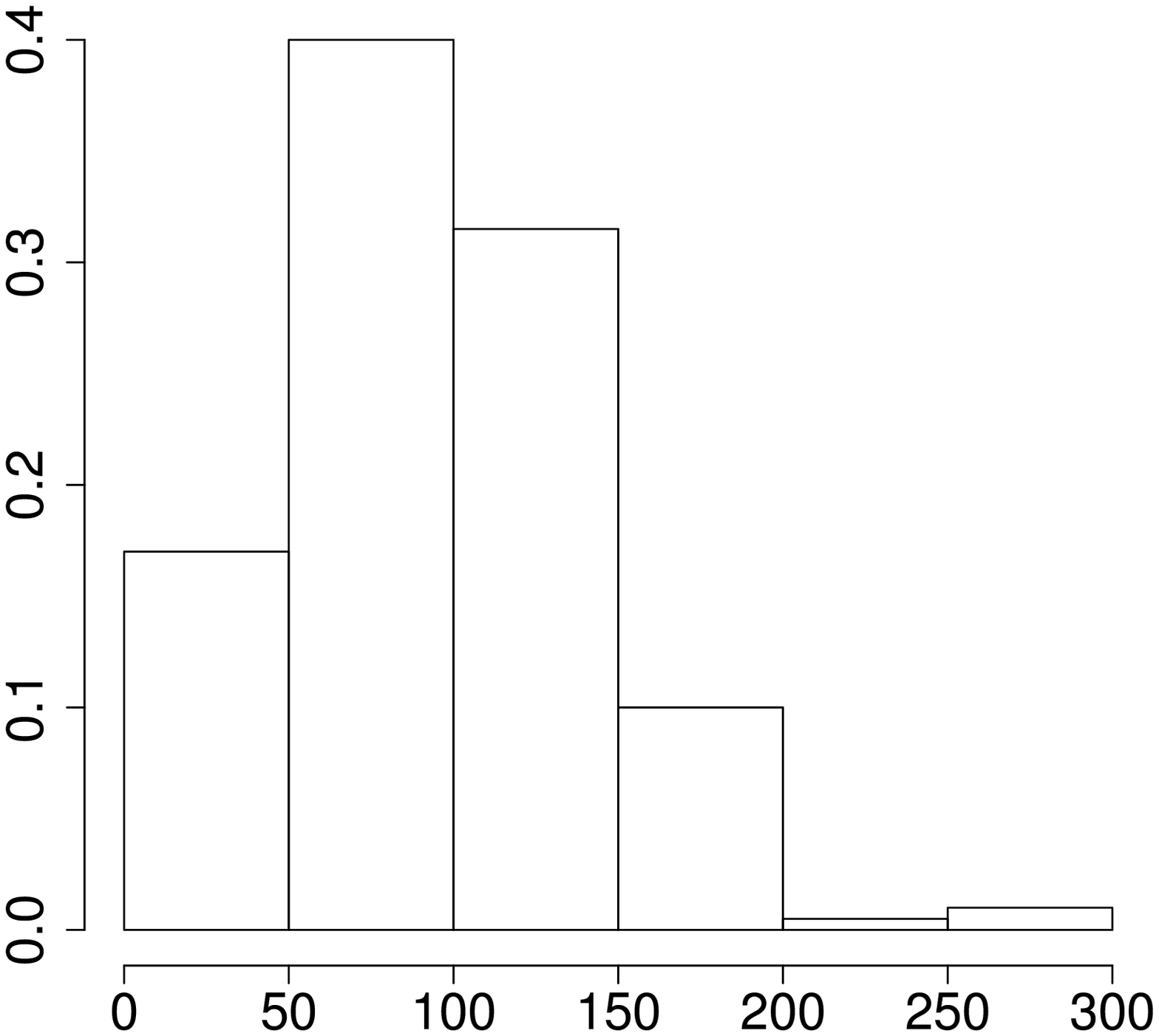}&
   \includegraphics[angle=0,scale=.22]{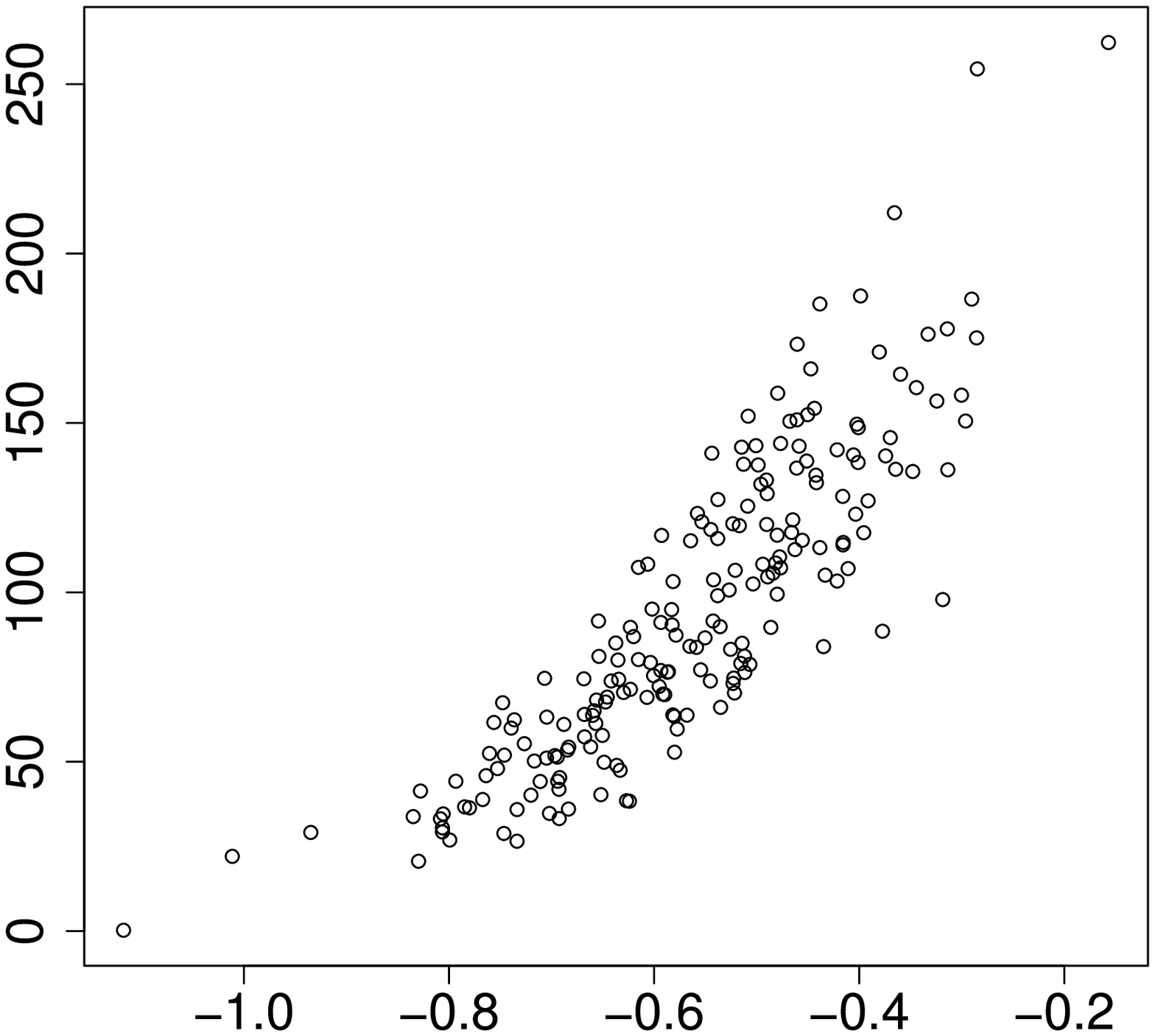}\\
  \footnotesize {$\hat \theta$ when $z$ is estimated} & \footnotesize {$\hat z$} & 
  \footnotesize {Scatterplot of $(\hat\theta,\ \hat z)$} 
  \end{tabular}
  }
        \caption{{\small Estimation of Model 3 when $\alpha=0.05$, $B=0.625$, $\theta=-0.5$, $z=100$, from 200 replications.}}\label{estmodel3-theta_5}

  \end{figure}

When $\theta=0.5$ (Figure  \ref{estmodel3-theta5}), there are only $3.5$ removable points in average on about $215$ points. In this latest extreme case, some estimations of $\theta$ and $z$ were even impossible since there were no removable points at all (in $5$ percent of the simulations). This shows the limit of the estimation procedure in presence of a very rigid tessellation. The average of $\hat \theta$ is $0.55$ and its standard deviation is $22.8$ {\small $10^{-2}$} when $z$ is not estimated. This is surprisingly reasonable in view of the few numbers of removable points. When both $z$ and $\theta$ are estimated, the results become bad: the average of $\hat \theta$ is $0.53$ with a standard deviation of $48$~{\small $10^{-2}$} and the average of $\hat z$ is $189$ with a standard deviation of $345$. Their joint distribution is plotted on bottom left of Figure  \ref{estmodel3-theta5}. A zoom in is plotted on bottom middle, where more than $90\%$ of the points are remaining. The last plot on bottom right shows the repartition of $\hat \theta$ according to the number of removable points in the tessellation, when $z=100$ is assumed to be known. There is a clear bias when the number of removable points is low. Since $z=100$ is fixed, this low number of removable points is associated with  a strong rigidity, so $\theta$ is most likely to be high. Moreover, the standard deviation of $\hat \theta$ seems to decrease with the number of removable points. This is confirmed by Table \ref{table} which contains, for a fixed number of removable points $card(\rem^{\hat \beta}(\g))$, the number of tessellations from our simulations having this number of removable points (named $replications$) and the standard deviation of $\hat \theta_n$ calculated from these tessellations (denoted $sd(\hat\theta)$). 

\begin{table}[h]
\begin{center}
\begin{tabular}[]{c|cccccccc|}
$card(\rem^{\hat \beta}(\g))$ & 1 & 2 & 3 & 4 & 5 & 6 &7 & $> 7$ \\
\hline
$sd(\hat\theta)$ {\footnotesize ($\times10^{-2}$)} & 27.6 & 14.9 & 18.4 & 15.1 &  9.6 & 10.3 &  8.8 & 8.7 \\
\hline
$replications$ & 23 & 41 & 41 &46 &29 &19 & 8 & 7 
\end{tabular}
\end{center}
\caption{{\small Standard deviation of $\hat\theta$ according to the number of removable points, from replications of Model 3 with $\theta=0.5$.}}
\label{table}
\end{table}


  \begin{figure}[htbp]
    \setlength{\tabcolsep}{0.1cm} \centerline{
  \begin{tabular}[]{ccc}
  \includegraphics[angle=0,scale=.16]{Figures/vor-theta5} &
  \includegraphics[angle=0,scale=.22]{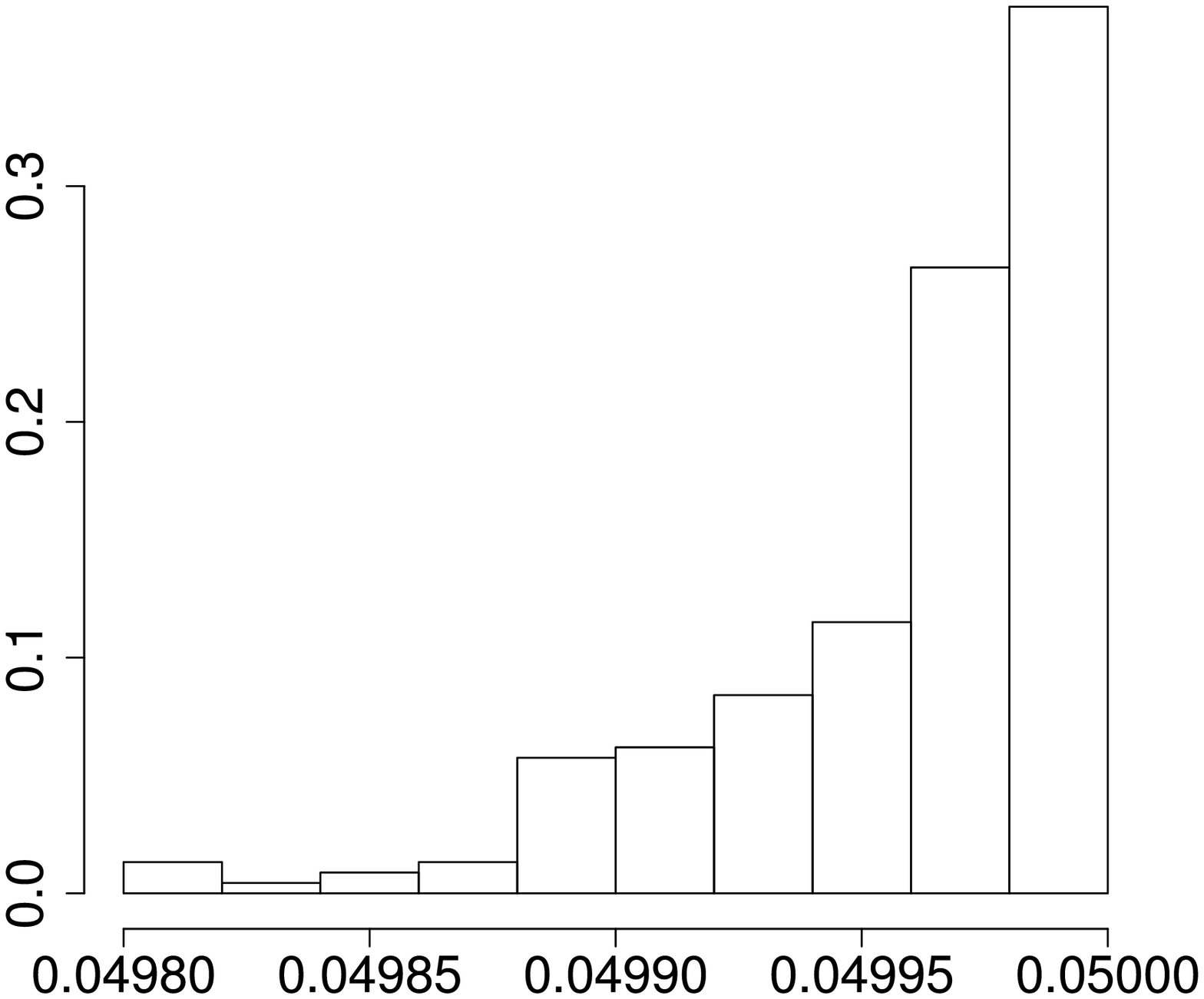} &
  \includegraphics[angle=0,scale=.22]{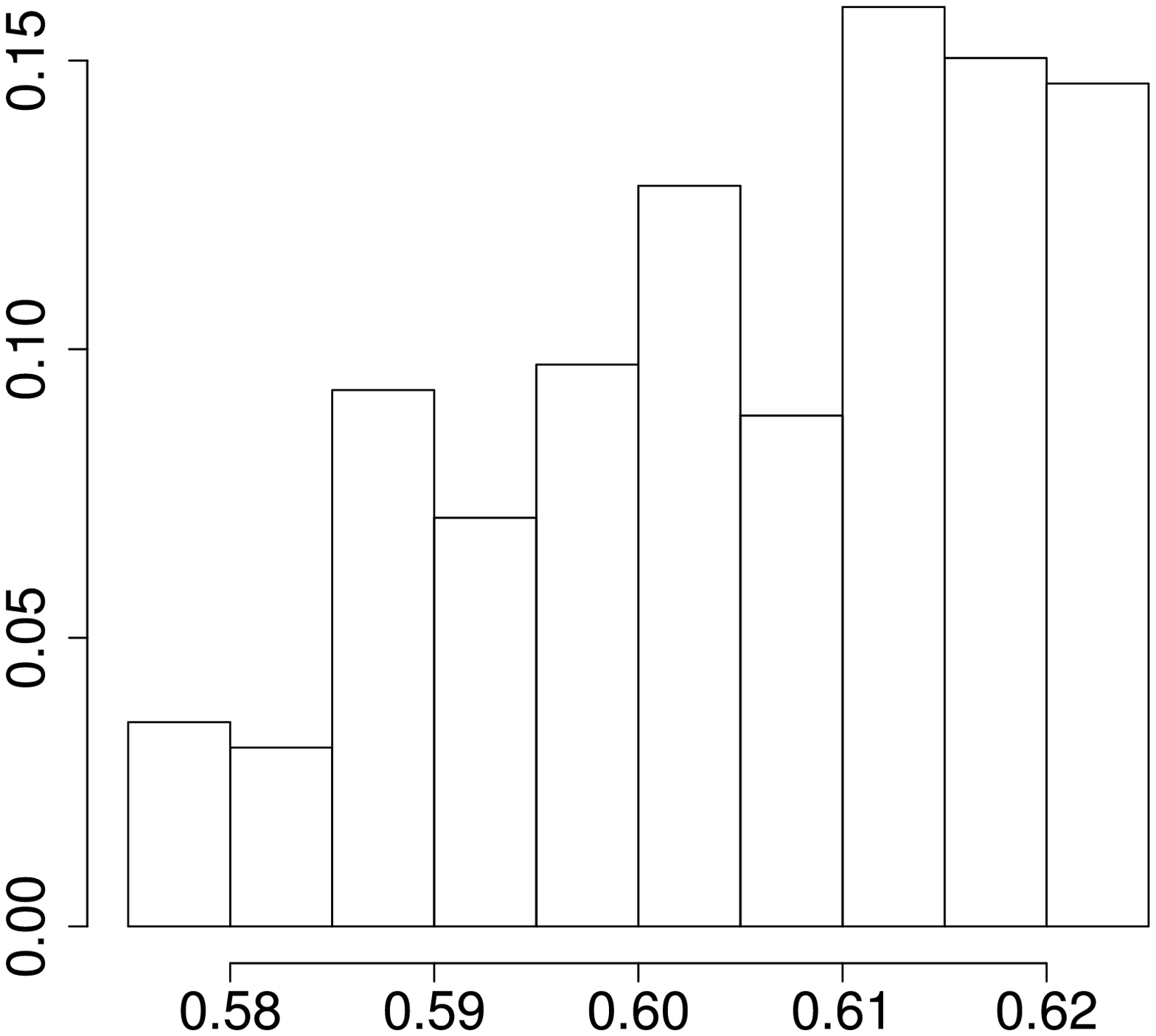} \\
  \footnotesize {Example of tessellation} &
  \footnotesize {$\hat \alpha$} & \footnotesize {$\hat B$}\\
  \includegraphics[angle=0,scale=.22]{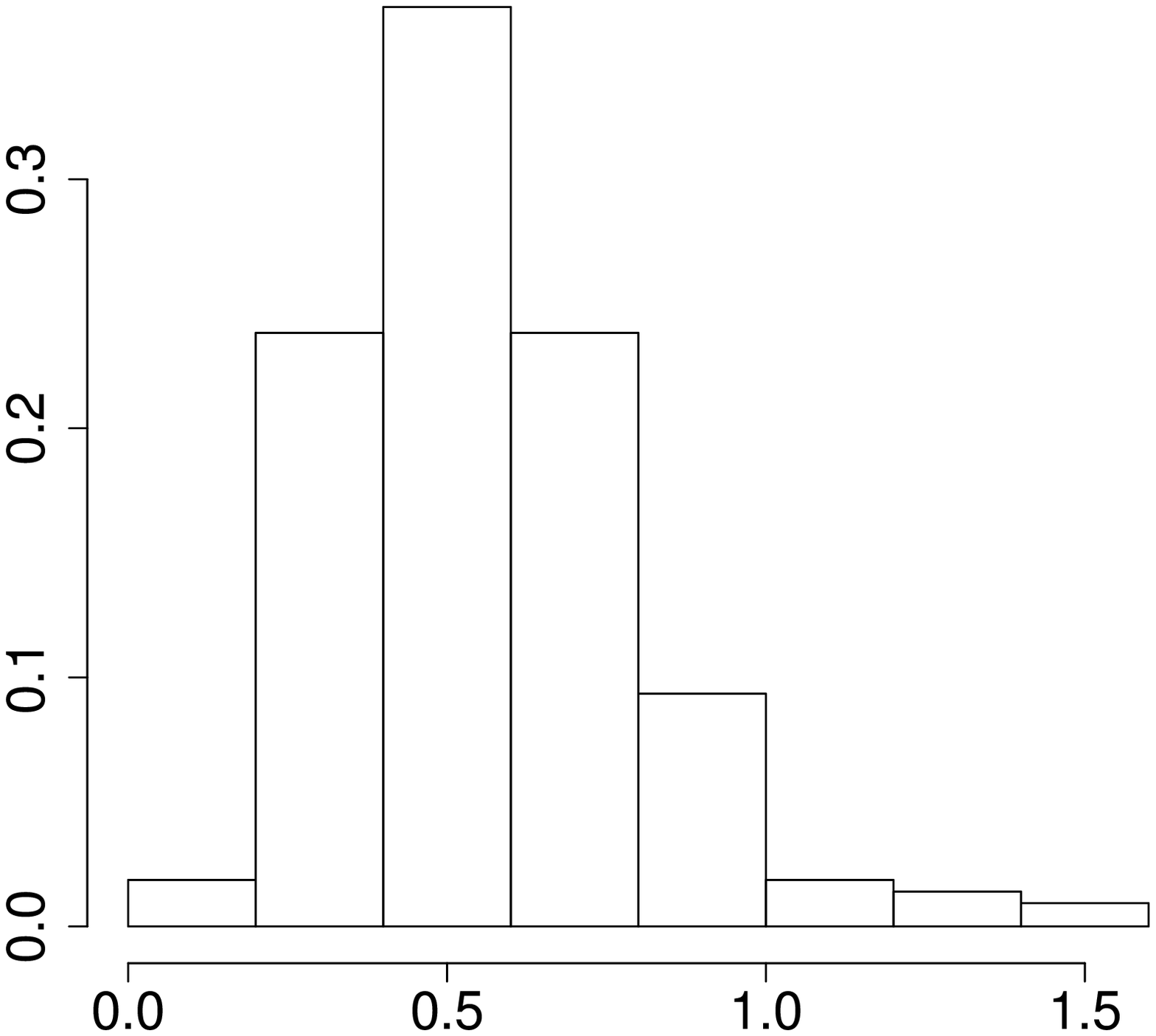}&
   \includegraphics[angle=0,scale=.22]{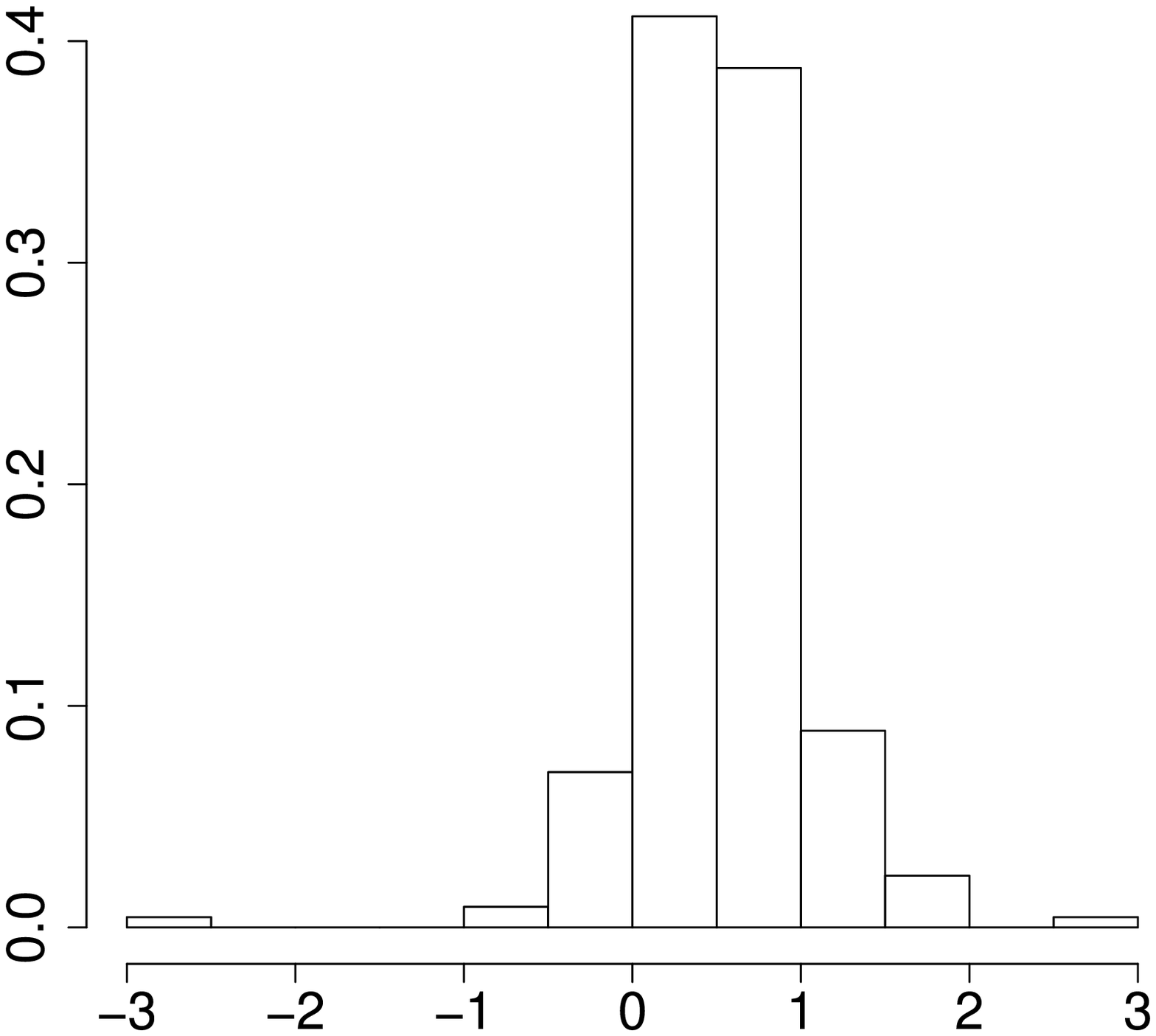}&
   \includegraphics[angle=0,scale=.22]{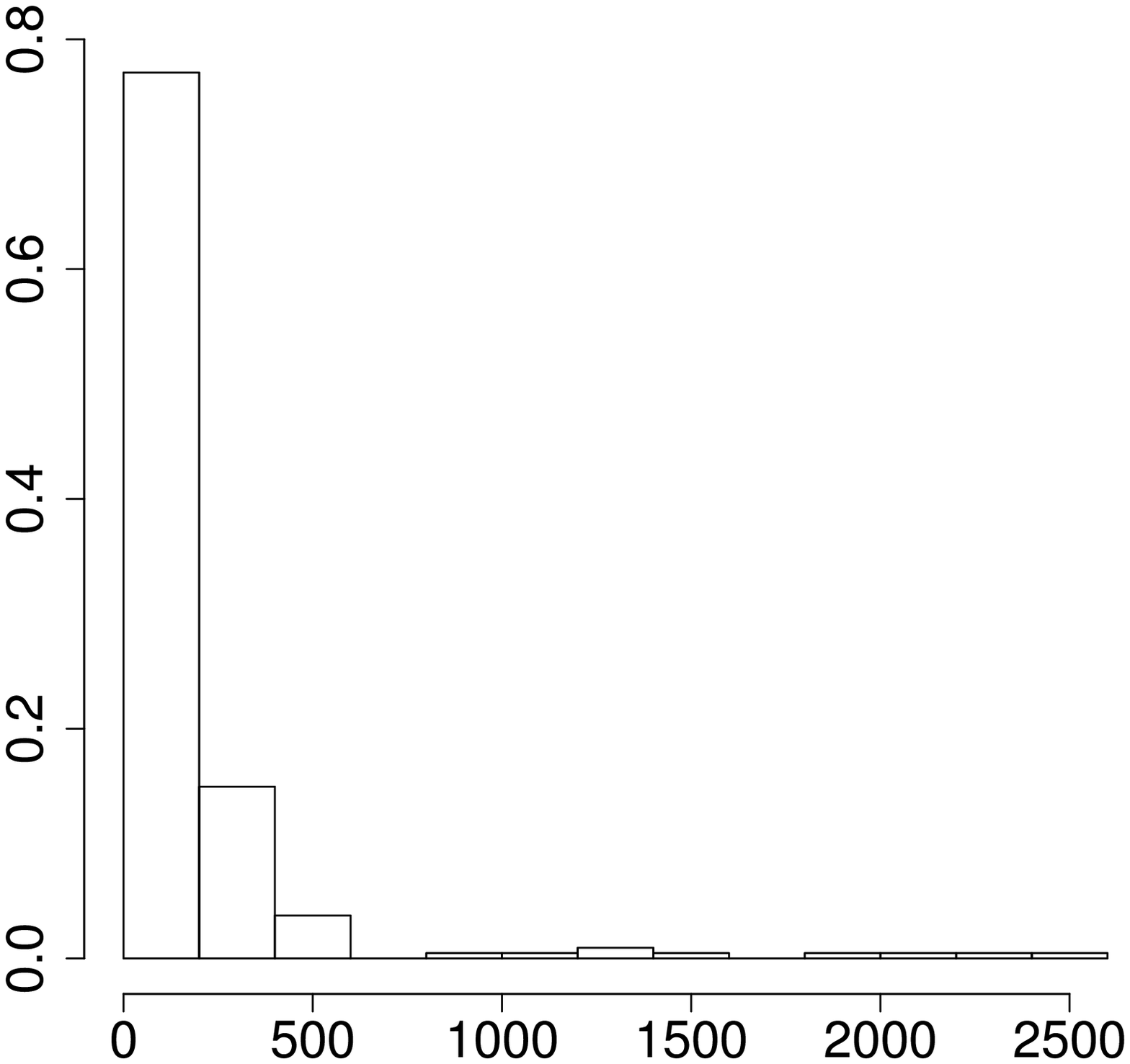}\\
  \footnotesize {$\hat \theta$ when $z$ is known} & \footnotesize {$\hat \theta$ when $z$ is estimated} & 
  \footnotesize {$\hat z$} \\
  \includegraphics[angle=0,scale=.22]{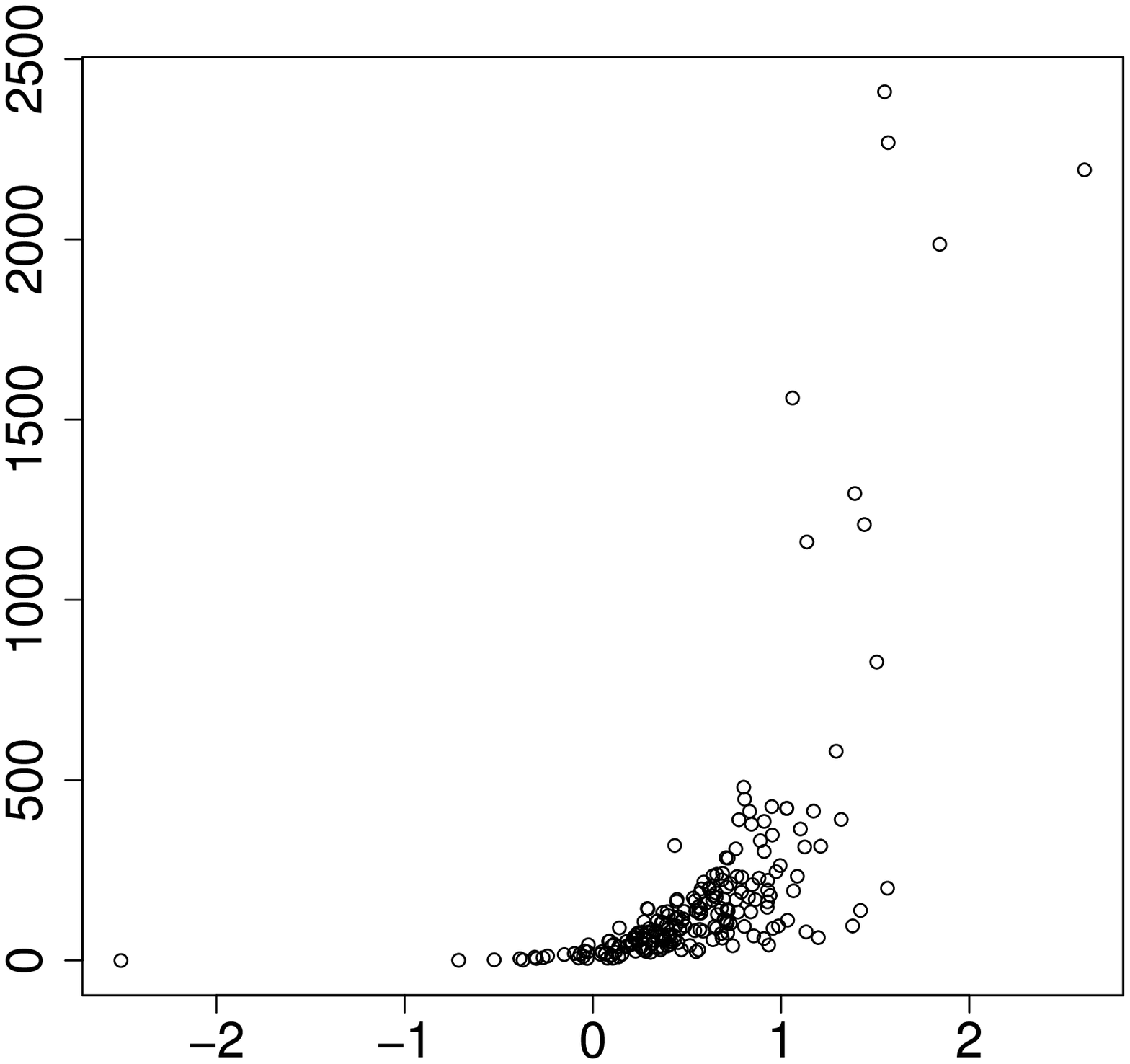}&
   \includegraphics[angle=0,scale=.22]{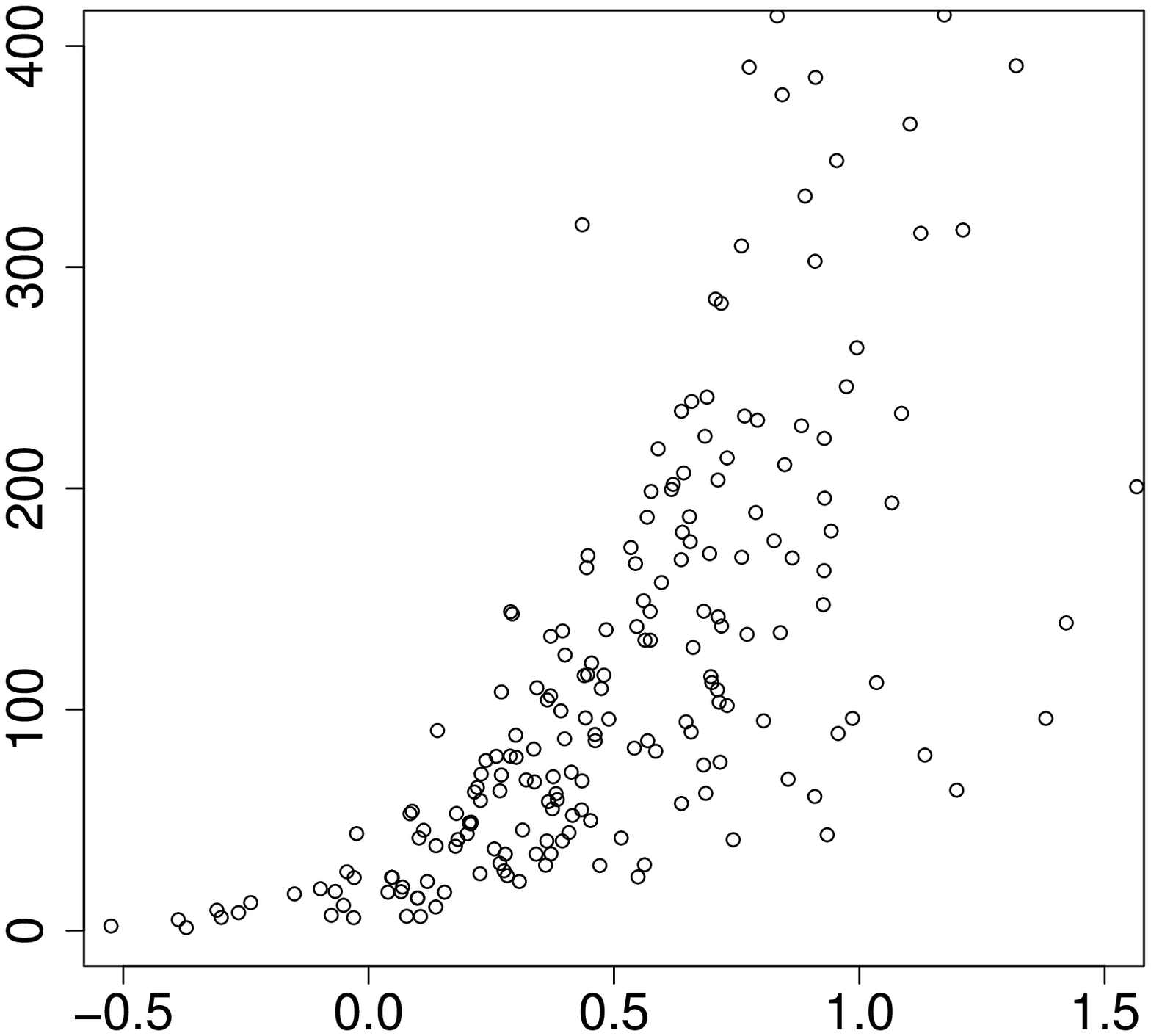}&
   \includegraphics[angle=0,scale=.22]{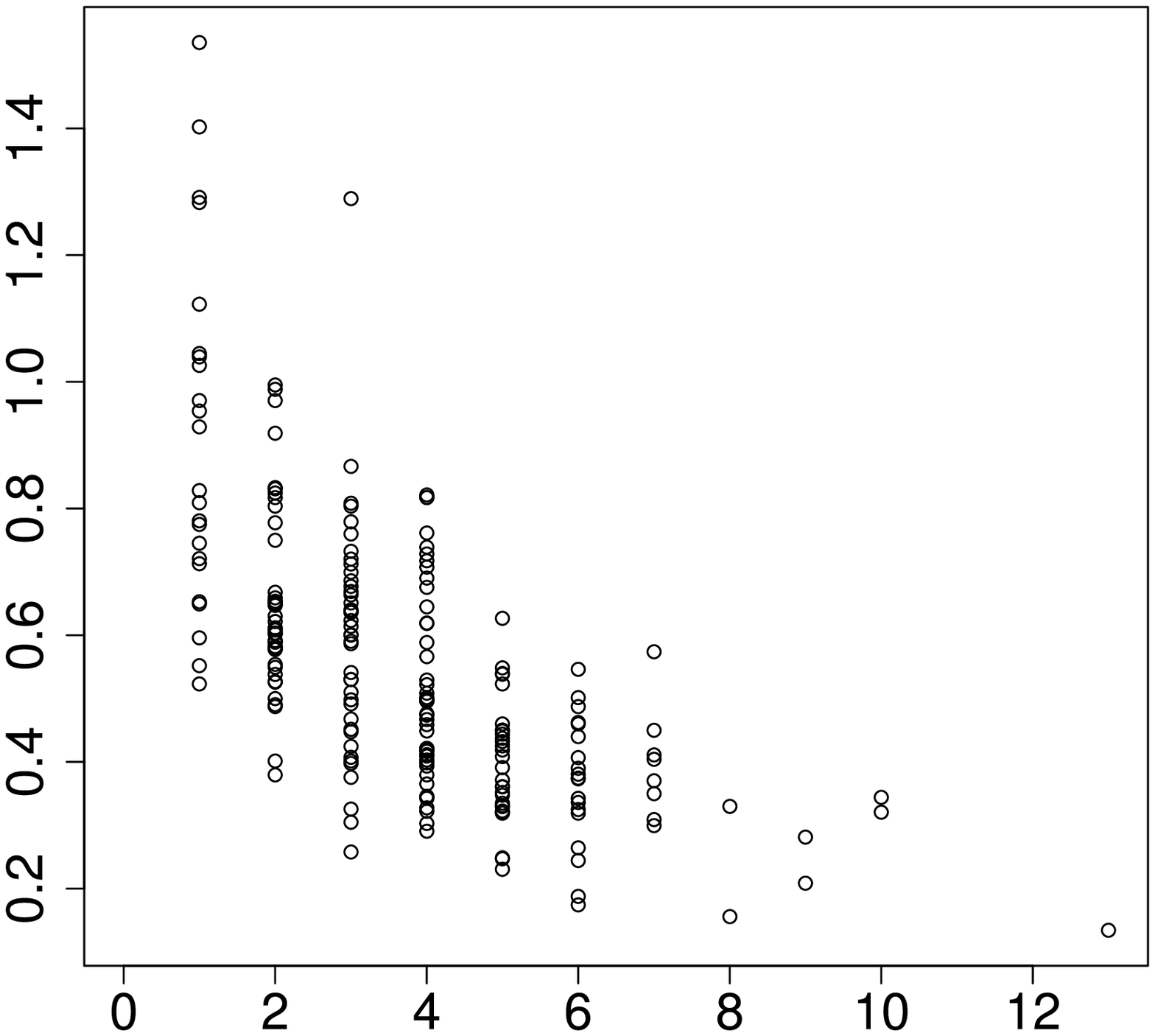}\\
  \footnotesize {Scatterplot of $(\hat\theta,\ \hat z)$} & \footnotesize {$(\hat\theta,\ \hat z)$ zoomed in} & 
  \footnotesize { $\big(card(\rem^{\hat \beta}(\g)),\ \hat \theta\big)$,  $z$ known} 
  \end{tabular}
  }
        \caption{{\small Estimation of Model 3 when $\alpha=0.05$, $B=0.625$, $\theta=0.5$, $z=100$, from 200 replications. On top left: A typical tessellation which is estimated. Bottom right: Repartition of $\hat\theta$ according to the number of removable points observed in the tessellation (see also Table \ref{table}).}}\label{estmodel3-theta5}

  \end{figure}

\subsection{Analysis of residuals}

When fitting a model to a data set, the analysis of the residuals is a standard way to check the quality of the model. The concept of residuals for spatial point processes is not simple. A general definition is proposed in \cite{BTMH}, where the authors also present several diagnostic tools based on residuals. The definition relies on the Campbell equilibrium equation due to Nguyen and Zessin (see \cite{NZ}), where the Papangelou conditional intensity is involved. In our context, the Papangelou conditional intensity does not always exist, because the hardcore interactions are not necessarily hereditary (see Remark 2 in \cite{DL}). Yet, a Campbell equilibrium equation still holds, provided we restrict the support to the set of removable points.

\begin{proposition}
Let $P$ be a stationary Gibbs Delaunay-Voronoi tessellation as defined in Definition \ref{defgibbs}. For every bounded non negative measurable function $\psi$ from $\Rd\times\Md$ to $\R$, we have
$$E_P\left(\sum_{x\in\rem^{\beta}(\g)} \psi(x,\g-x)\right)=E_P\left(\int_{\Rd} \psi(x,\g) e^{-h^{\beta, \theta}(x,\g)}\nu(dx)\right),$$ where $h^{\beta, \theta}$ is defined in (\ref{defh}) and $E_P$ denotes the expectation under $P$.
\end{proposition}

This proposition is proved in \cite{DL}. From this equation, following \cite{BTMH}, we can define the innovation process, for  any bounded set $\Delta$ in $\Rd$ and for every function $\psi$ as above:
$$I\left(\Delta,\psi,h^{\beta,\theta},\nu\right)=\sum_{x\in\rem^{\beta}(\g)\cap \Delta} \psi(x,\g-x)-\int_\Delta \psi(x,\g) e^{-h^{\beta, \theta}(x,\g)}\nu(dx).$$

The residuals are then defined as an estimation of the innovations:
$$R\left(\Delta,\psi,h^{\hat\beta,\hat\theta},\hat\nu\right)=\sum_{x\in\rem^{\hat\beta}(\g)\cap \Delta} \psi(x,\g-x)-\int_\Delta \psi(x,\g) e^{-h^{\hat\beta, \hat\theta}(x,\g)}\hat\nu(dx),$$
where $\hat\nu$ is an estimation of the intensity measure $\nu$, which is simply $\hat z\lambda$ in the stationary case.

This generalization of the residuals to the setting of possible non-hereditary interactions allows to perform several diagnostic plots. We refer to \cite{BTMH} for a presentation of different relevant choices for $\psi$, and for some diagnostic tools. A smoothed version of the residuals is also proposed, leading to more appealing graphics. The main purpose of the residuals analysis is to check whether the fitted model is misspecified.

As an illustration, let us assess the effect of a misspecified model to the tessellation simulated in top left of Figure \ref{model3bis}.  It actually corresponds to a sample from Model 3 where $\alpha=0.05$, $B=+\infty$, $z=100$ and $\theta=-0.5$. But we will improperly fit a stationary Poisson process to this sample, then we will fit Model 2 (where the interaction relies on the Delaunay triangulation). Finally the correct Model 3 will be fitted for a sake of comparison. Figure \ref{fig-tess-for-residuals} represents the sample according to these three points of view.

\begin{figure}[htbp]
    \setlength{\tabcolsep}{0.1cm} \centerline{
  \begin{tabular}[]{ccc}
  \includegraphics[angle=0,scale=.23]{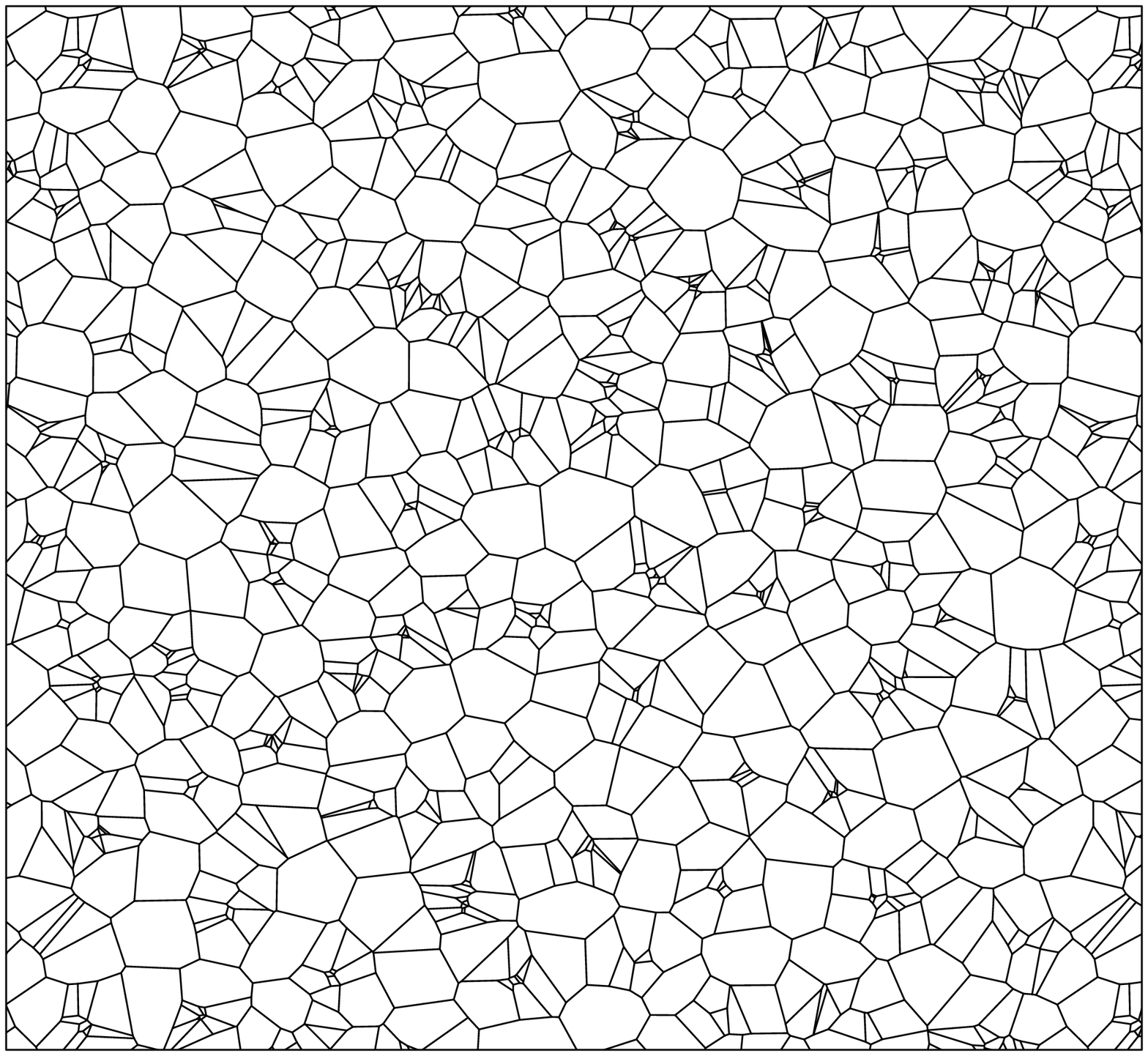} &
  \includegraphics[angle=0,scale=.23]{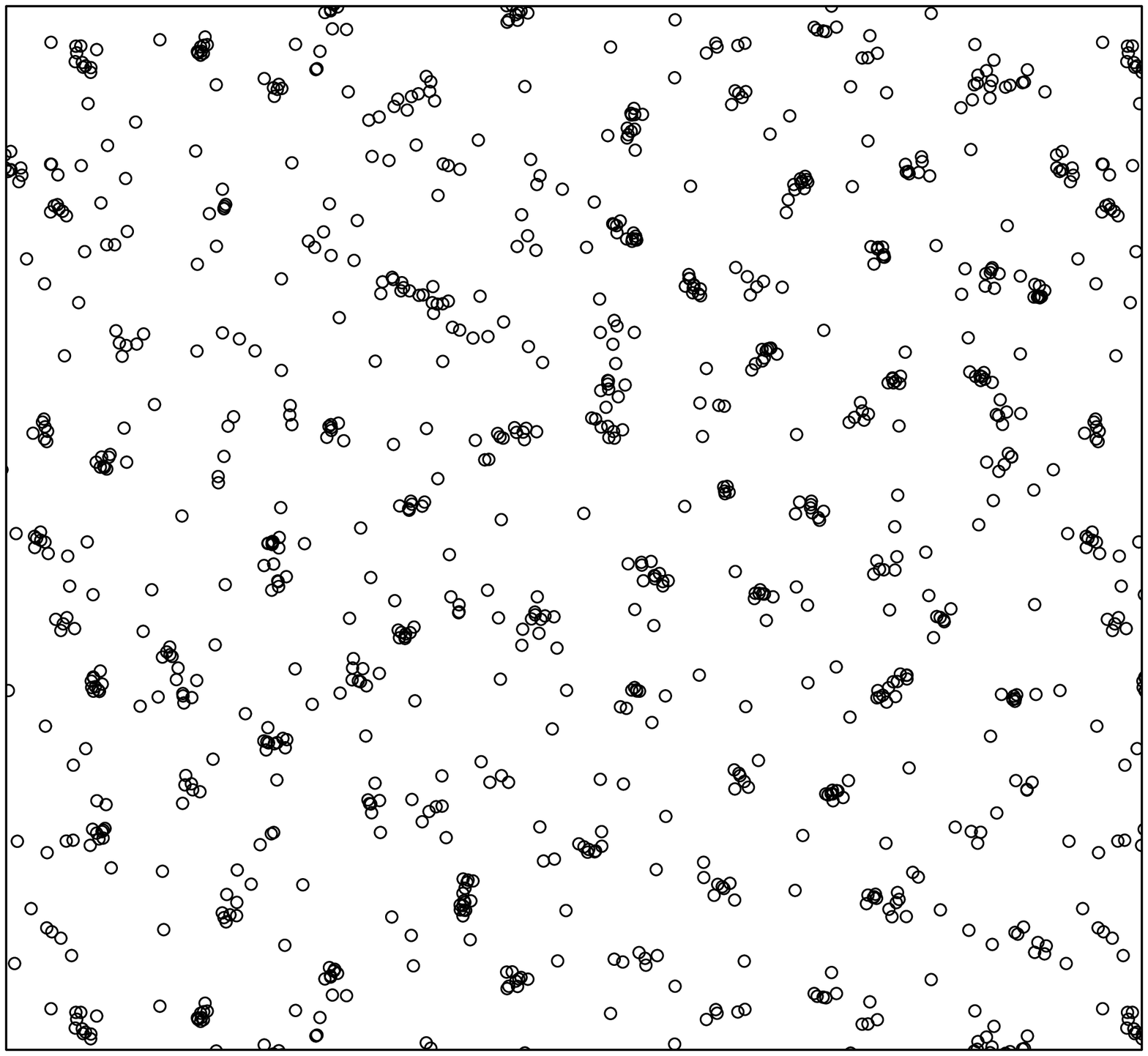} &
  \includegraphics[angle=0,scale=.23]{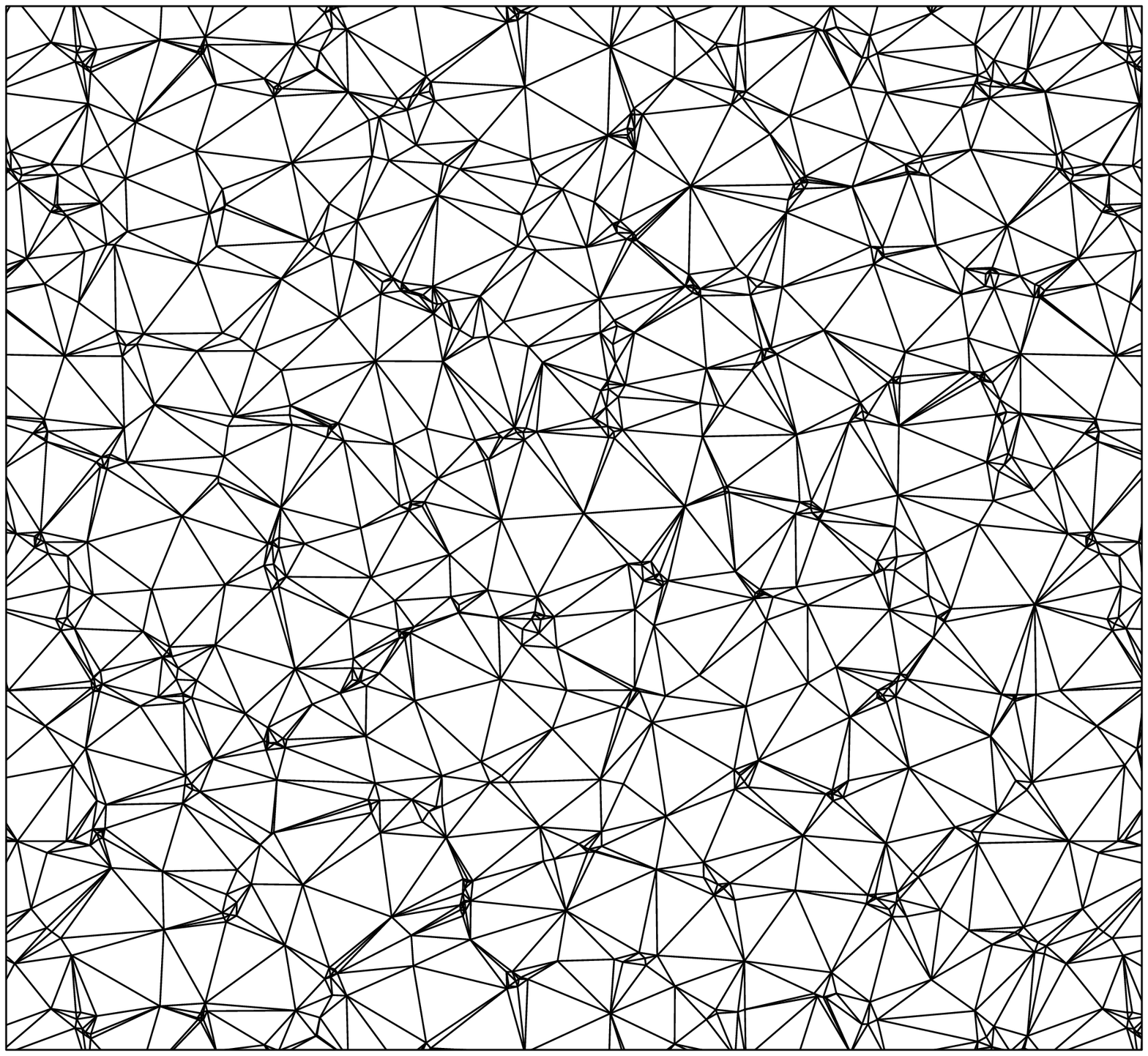} \\
  \footnotesize {Voronoi} &
  \footnotesize {Points} & \footnotesize {Delaunay}
   \end{tabular}
  }
        \caption{{\small Voronoi tessellation (left) and Delaunay tessellation (right) from the same point configuration (middle), coming from a simulation of Model 3  (top right of Figure \ref{model3bis}).}}\label{fig-tess-for-residuals}
  \end{figure}

We consider the simple case when $\psi=1$. This corresponds to the so-called raw residuals, which have the following form in the stationary case:
$$R\left(\Delta,1,h^{\hat\beta,\hat\theta},\hat z\right)=card\left(\rem^{\hat\beta}(\g)\cap \Delta\right)-\hat z\int_\Delta e^{-h^{\hat\beta, \hat\theta}(x,\g)}dx.$$
To check the fitted model,  we use the QQ-plot diagnostic presented in \cite{BTMH}. It consists in comparing the empirical quantiles of the fitted residuals to the empirical quantiles of bootstrapped residuals.

If we fit a stationary Poisson process to the sample, we obtain an estimated intensity $\hat z=833$. The raw residuals, computed on squares $\Delta$ with side $0.01$, are shown in top left of Figure \ref{fig-residuals}. The same kind of residuals have then been computed on 100 simulated Poisson process with intensity $\hat z$. A QQ-plot of these residuals with a $95\%$-confidence interval is shown on top right of  Figure \ref{fig-residuals}, where the residuals of the original sample have been added (crosses). An example of raw residuals from a simulated Poisson process is represented in the middle of this plot.
It appears that the residuals of our sample do not behave as those from the simulated Poisson point processes. The stationary Poisson model is then misspecified. 

Similarly, if we fit Model 2 to the same sample, we obtain $\hat \alpha=0.055$ and $\hat \theta=4.49$ when $z=1000$. The raw residuals 
for this model are shown in bottom left of Figure \ref{fig-residuals}. We have bootstrapped residuals from 100 simulated samples from Model 2 with the same parameters. One example of such residuals is shown on bottom middle. The QQ-plot, in bottom right, shows that the original sample does not seem to follow Model 2. 

Finally, Model 3 is fitted. The estimation gives $\hat \alpha=0.049$, $\hat B=97.8$ and $\hat \theta=-0.56$ when $z=100$. The same plots as before are represented in Figure \ref{fig-bons-residuals}. According to the QQ-plot, one should not reject the fitted model. Let us remark that the estimation of $\theta$ is rather bad for our sample: the error is $-0.06$ although other simulations  shows that the standard deviation of the errors is about $0.02$. This is the reason why the distribution of the residuals is on the edge of the confidence interval in the QQ-plot. The two residuals images represented on the left show that some residuals may be very negative on some squares (the black ones). This is confirmed by the dispersion of the lowest quantiles in the QQ-plot. Thus, the distribution of the residuals can not be Gaussian in this example. This differs from the asymptotic gaussianity of most residuals conjectured in \cite{BTMH}.

  \begin{figure}[htbp]
    \setlength{\tabcolsep}{0.1cm} \centerline{
  \begin{tabular}[]{ccc}
  \includegraphics[angle=0,scale=.25]{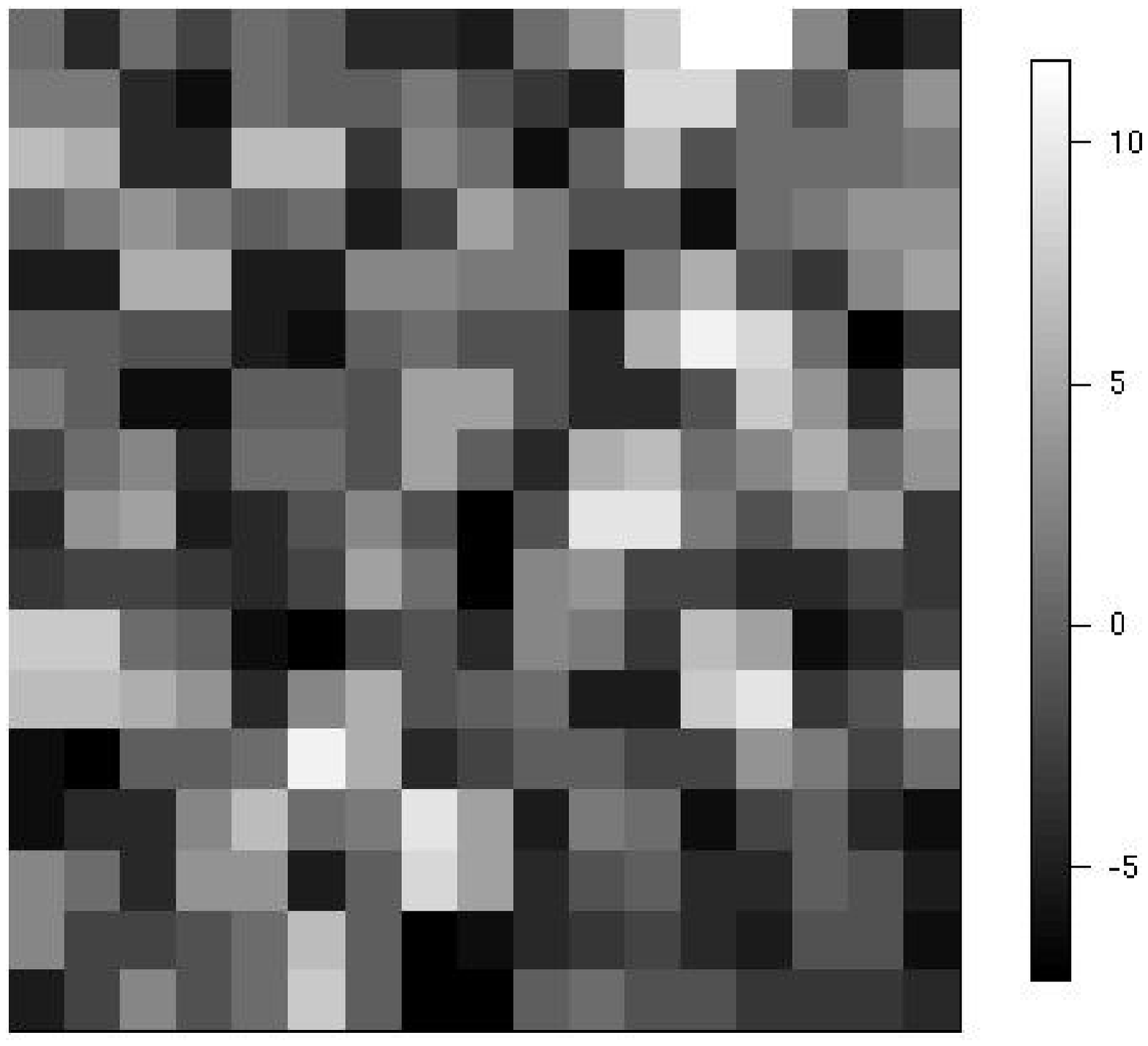} &
  \includegraphics[angle=0,scale=.25]{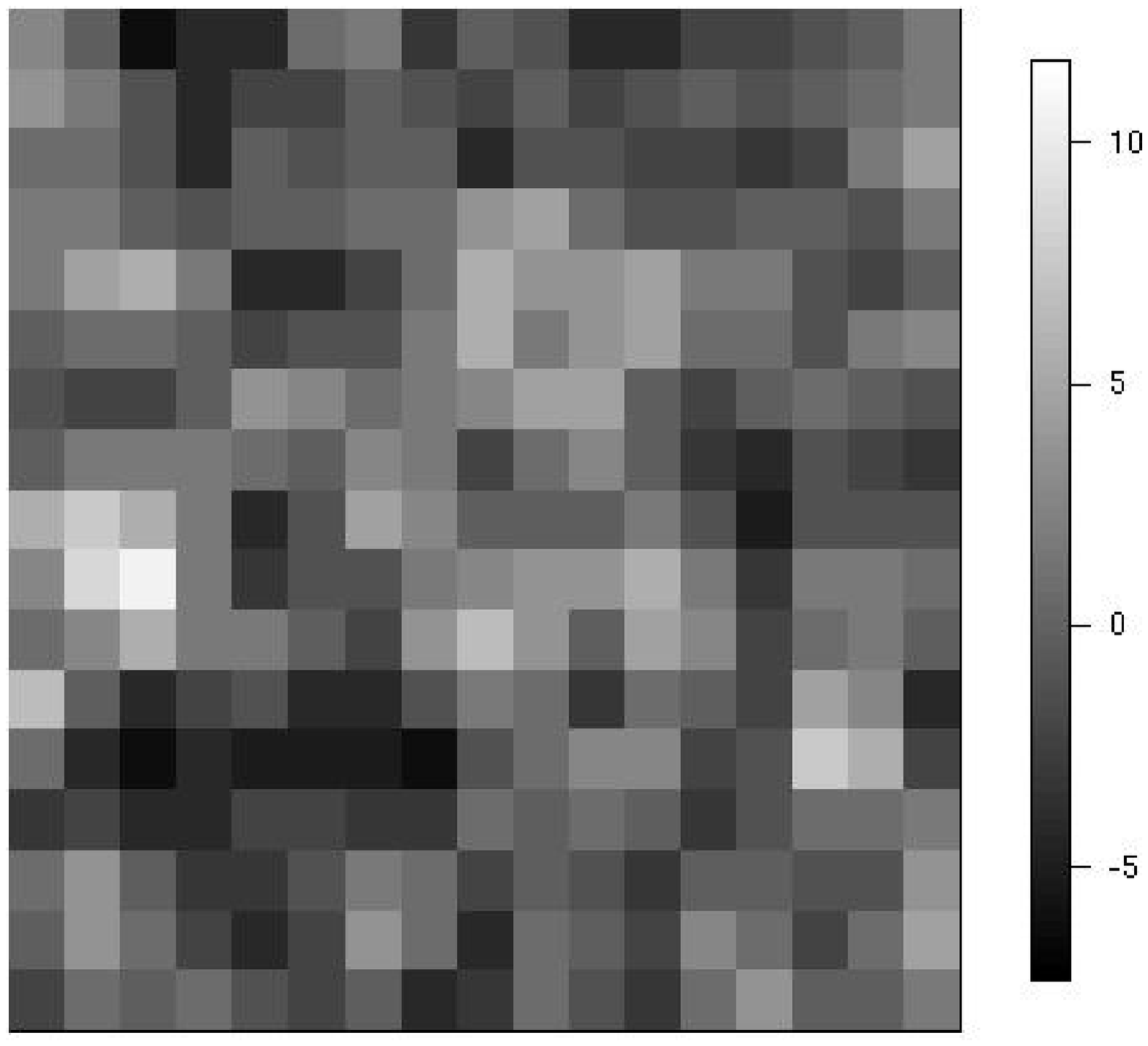} &
  \includegraphics[angle=0,scale=.25]{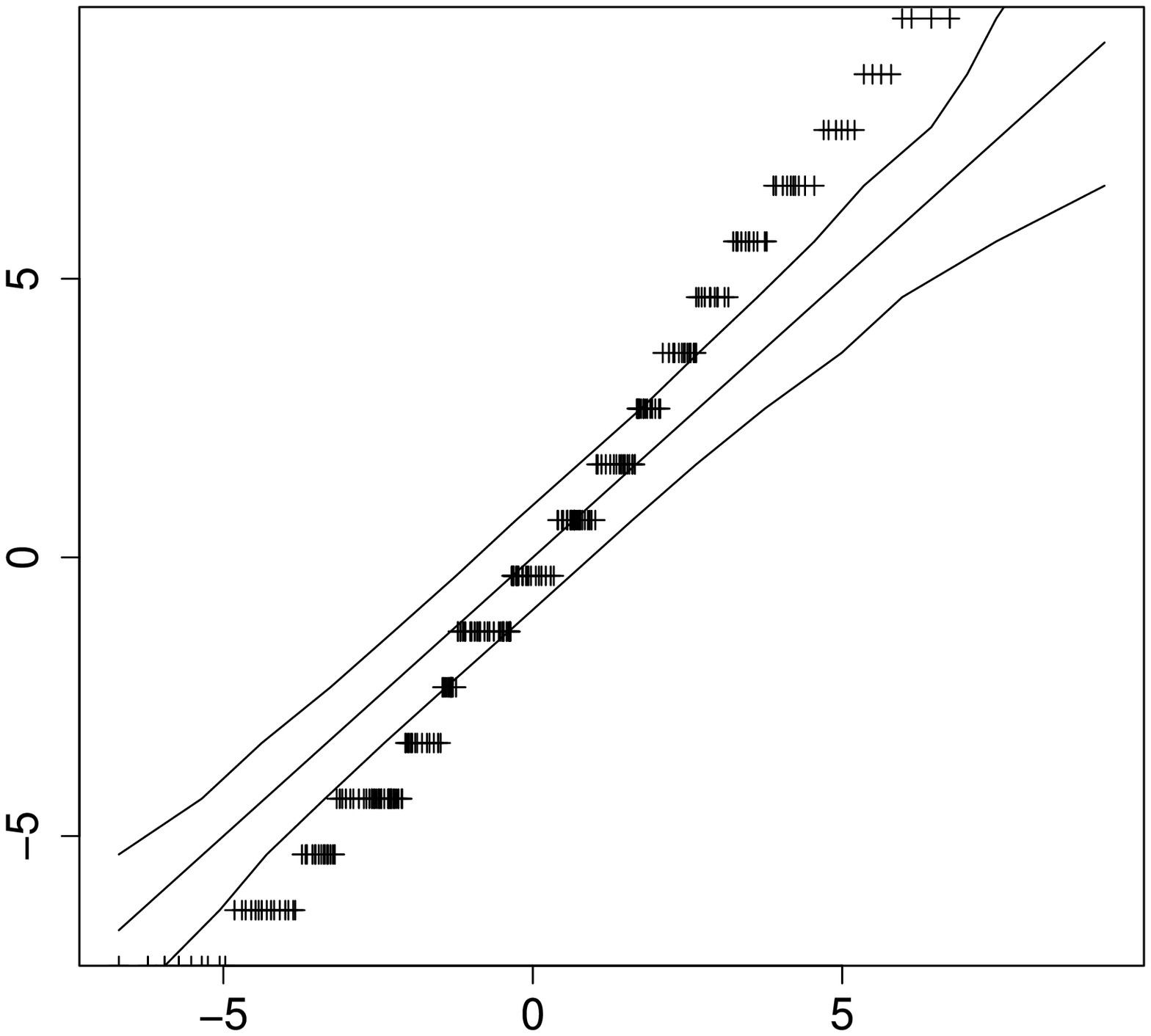} \\
  \footnotesize {Residuals when fitting a Poisson} &
  \footnotesize {Simulated Poisson residuals} & \footnotesize {QQplot from bootstrap}\\
  \includegraphics[angle=0,scale=.25]{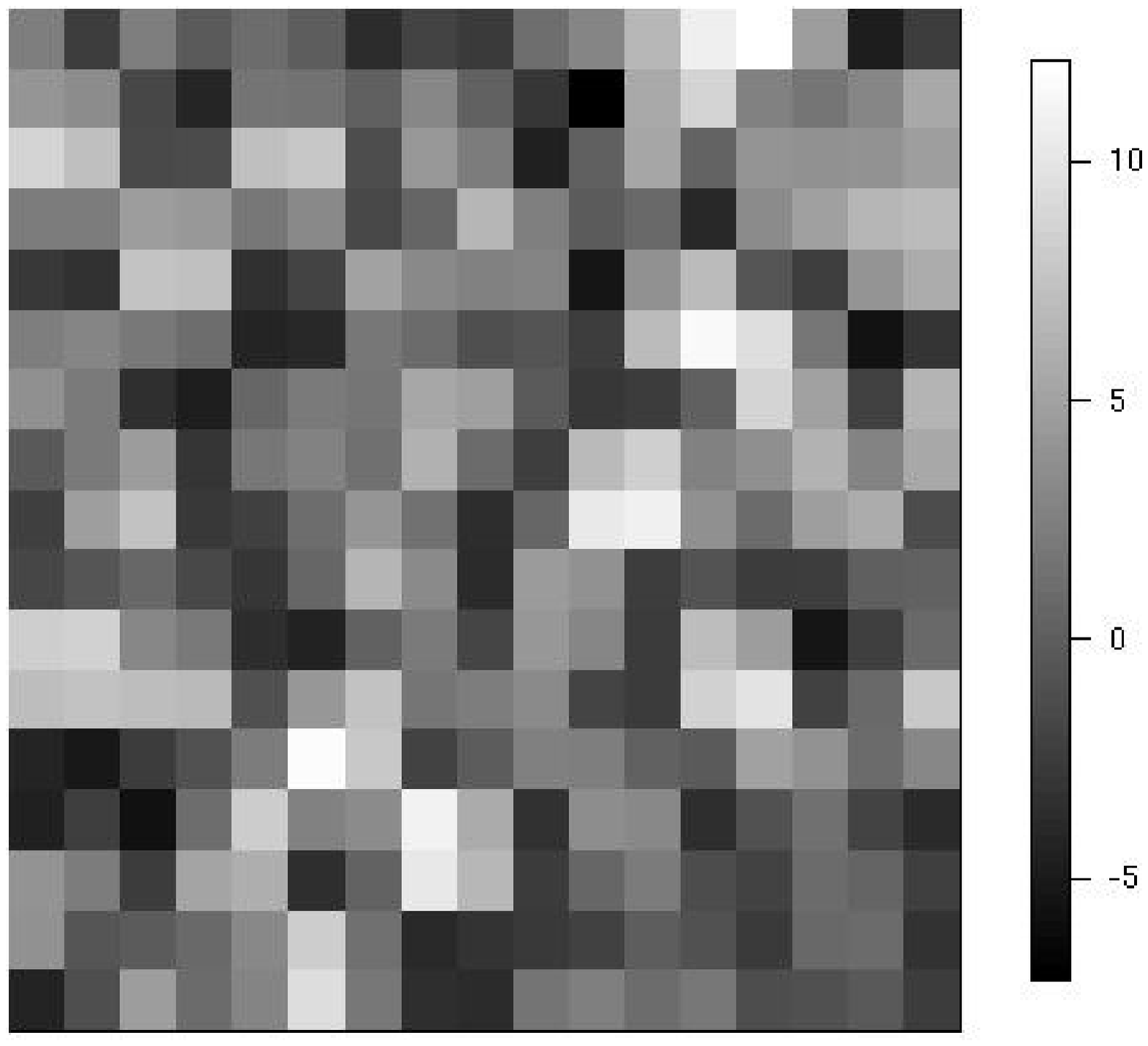}&
   \includegraphics[angle=0,scale=.25]{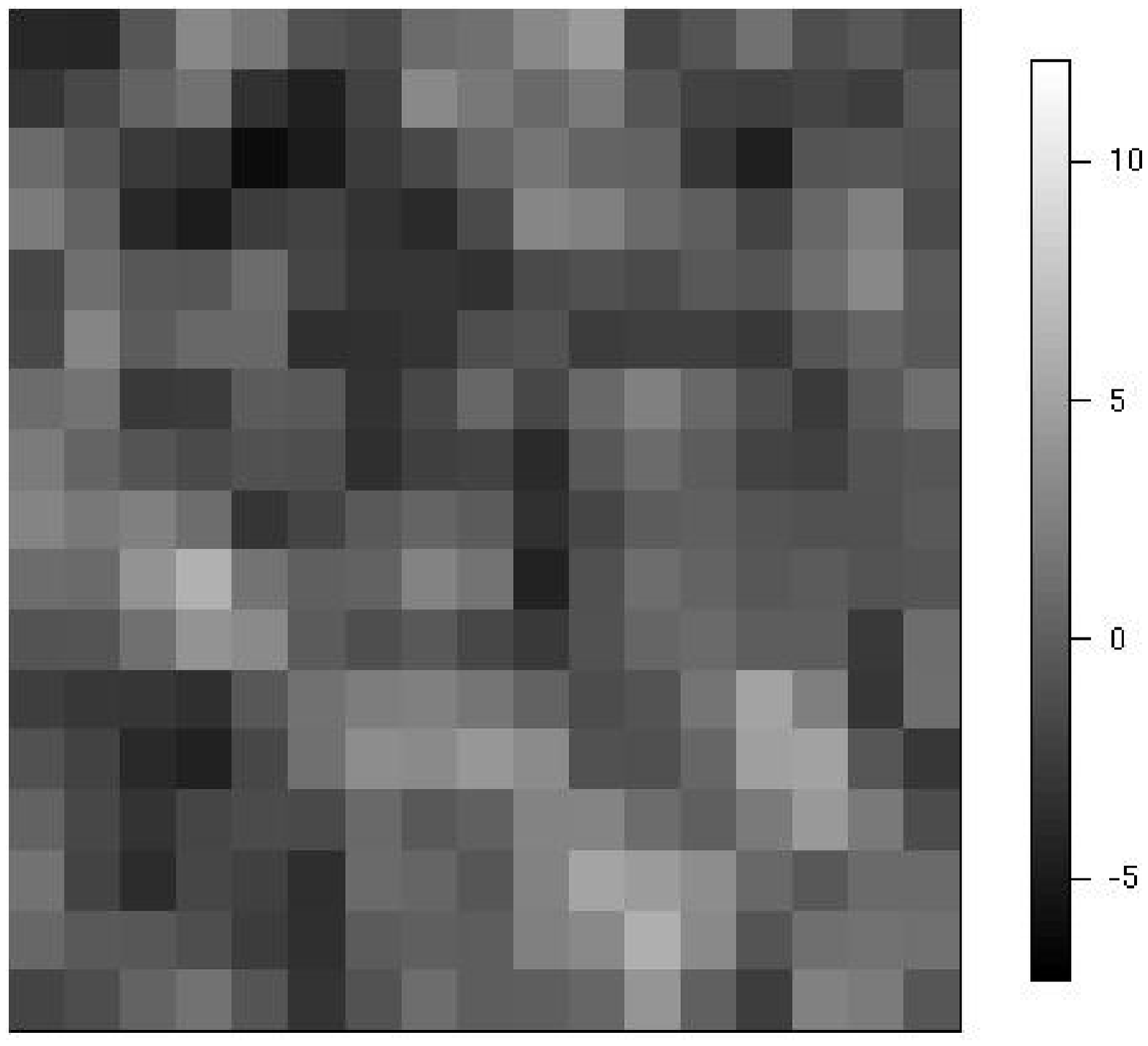}&
   \includegraphics[angle=0,scale=.25]{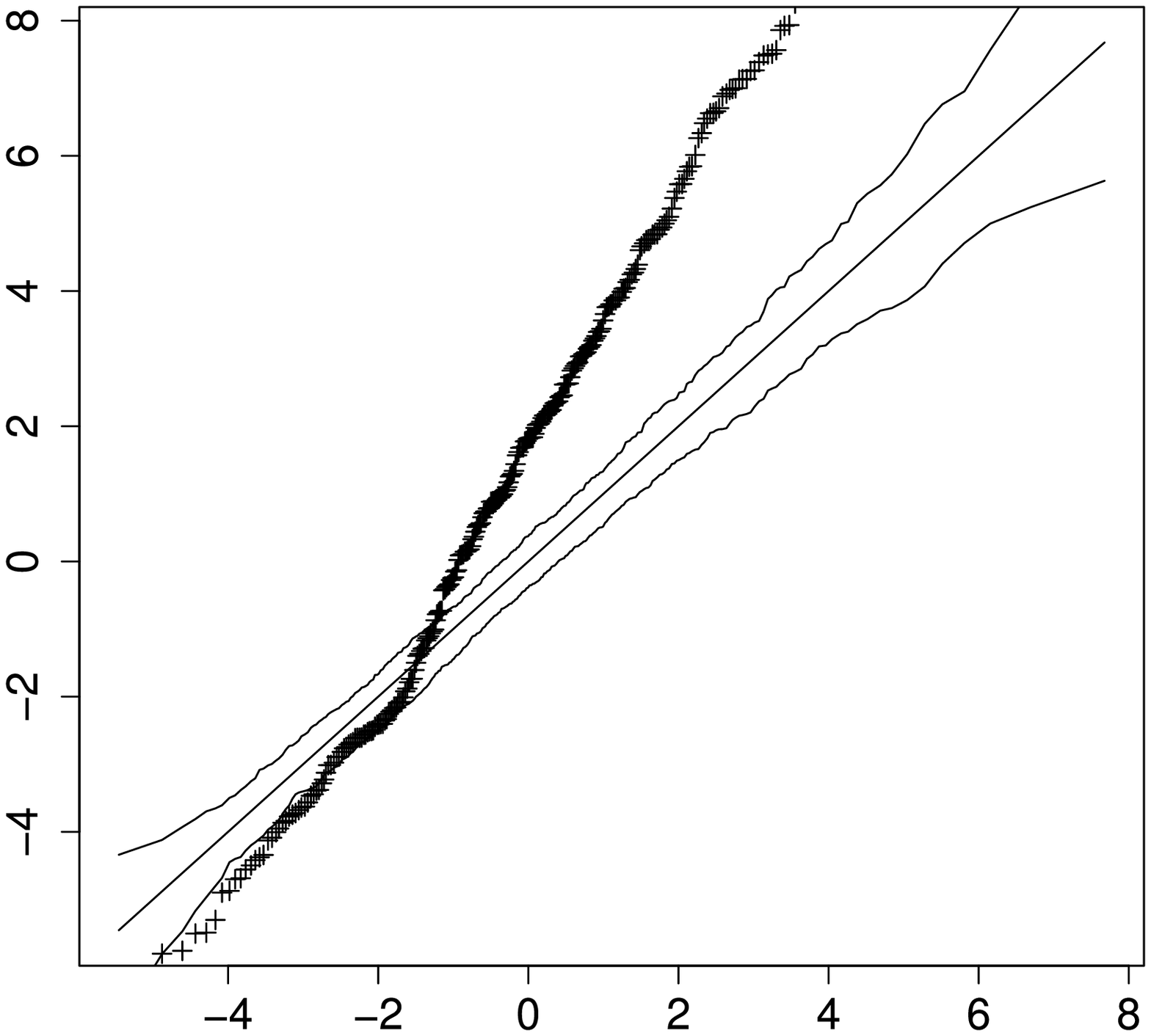}\\
  \footnotesize {Residuals when fitting Model 2} & \footnotesize {Simulated residuals from Model 2} & 
  \footnotesize {QQplot from bootstrap} 
  \end{tabular}
  }
        \caption{{\small Analysis of residuals for misspecified models.}}\label{fig-residuals}

  \end{figure}

\begin{figure}[htbp]
    \setlength{\tabcolsep}{0.1cm} \centerline{
  \begin{tabular}[]{ccc}
 \includegraphics[angle=0,scale=.25]{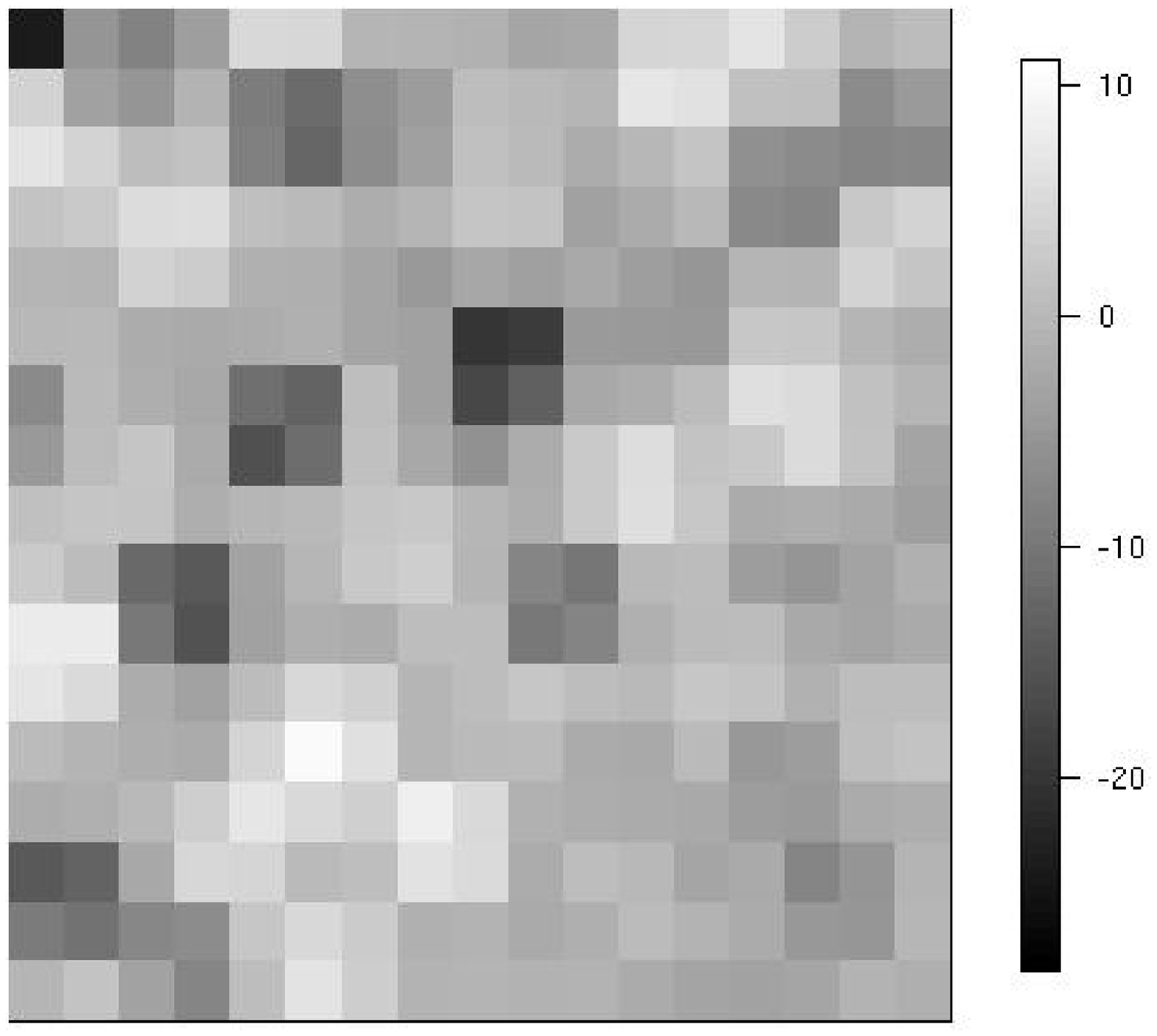}&
   \includegraphics[angle=0,scale=.25]{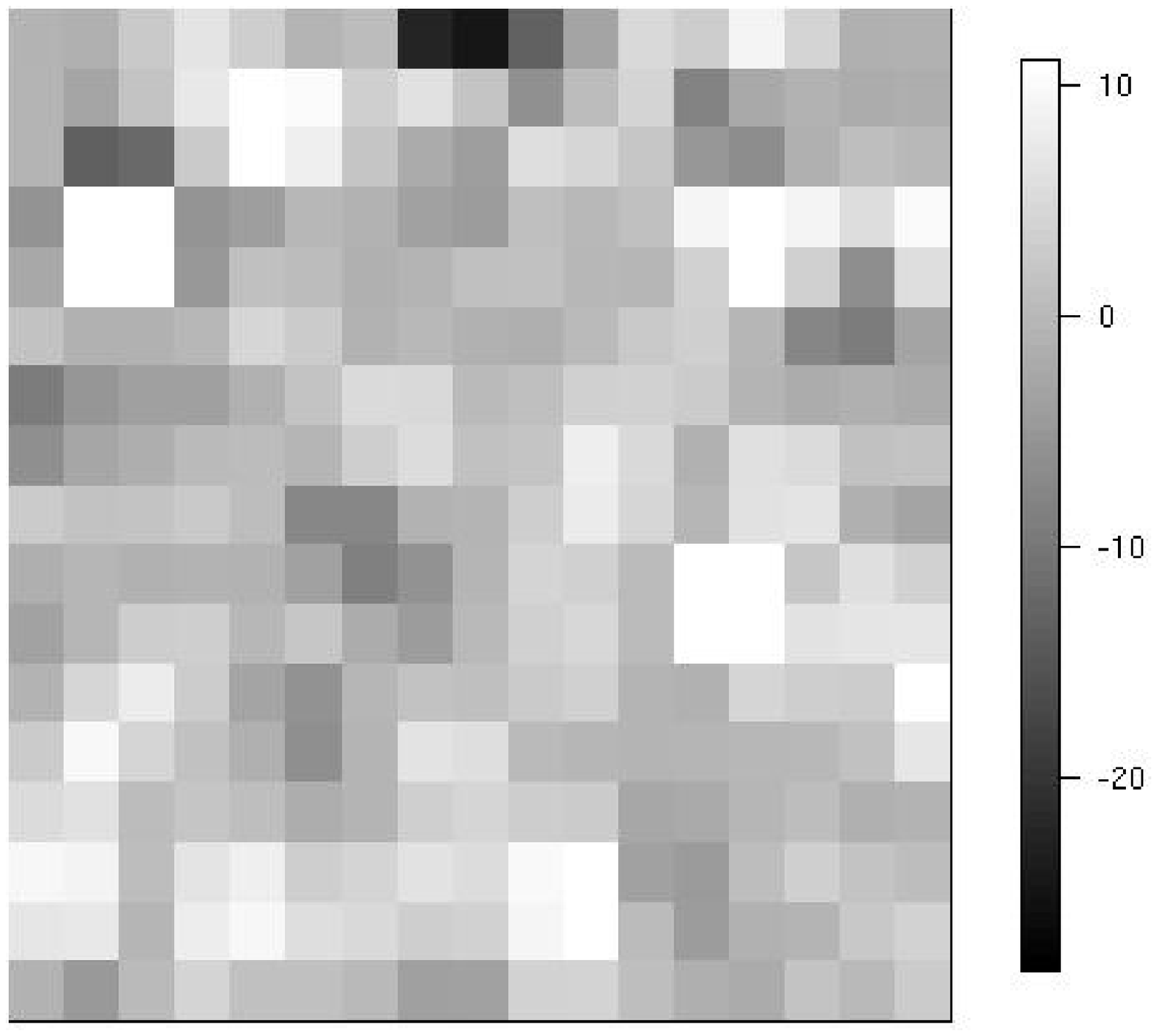}&
   \includegraphics[angle=0,scale=.25]{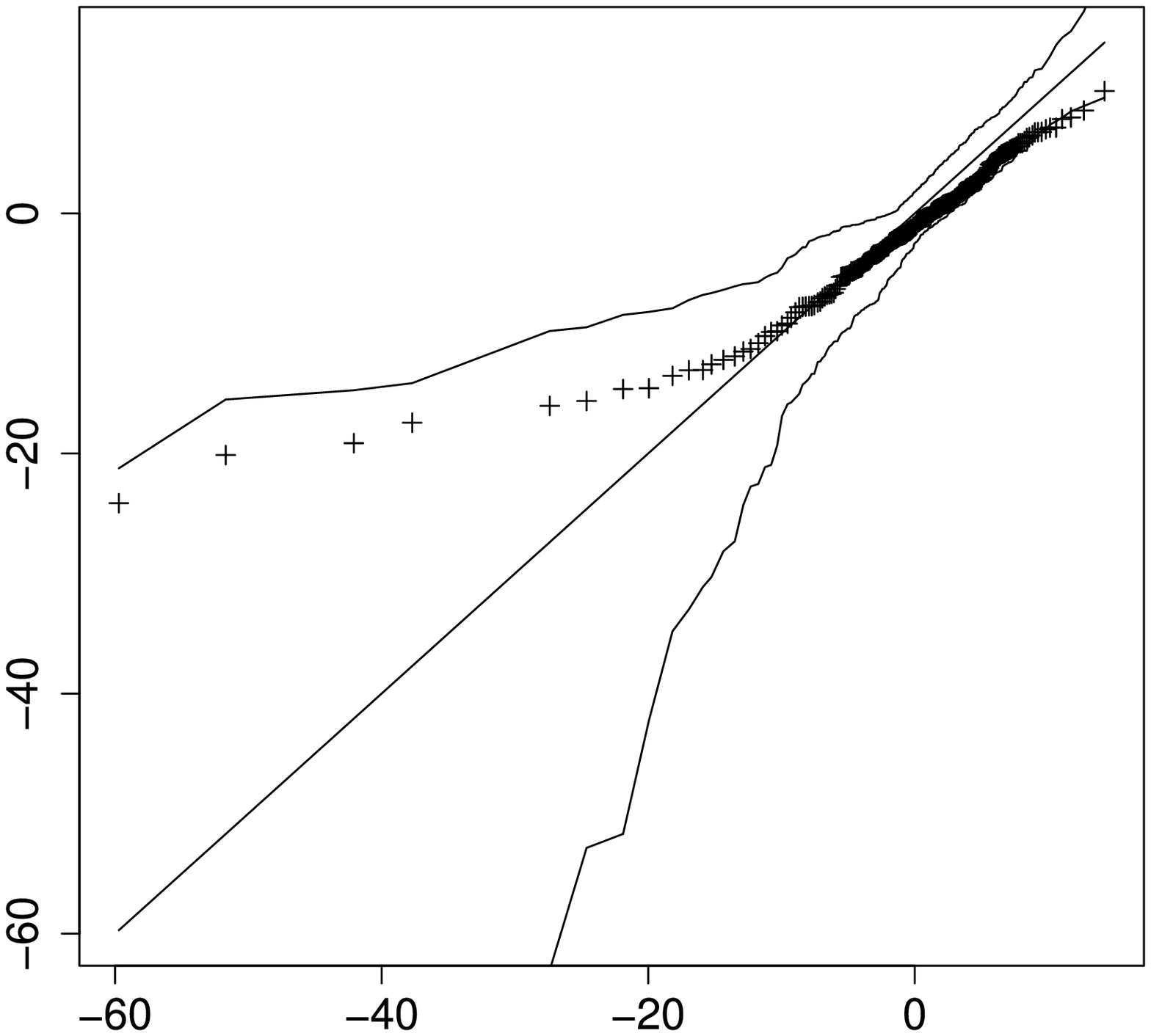}\\
  \footnotesize {Residuals when fitting Model 3} & \footnotesize {Simulated residuals from Model 3} & 
  \footnotesize {QQplot from bootstrap} 
  \end{tabular}
  }
        \caption{{\small Analysis of residuals for the correct model.}}\label{fig-bons-residuals}

  \end{figure}

\appendix
\section{}

\subsection{Existence of Gibbs Delaunay-Voronoi tessellations.}
The existence results presented here are published in a more general setting in \cite{D} and \cite{DDG}. They are slightly modified and simplified so that they suit better the setting of Gibbs Delaunay-Voronoi tessellations. We suppose that the energy functions have the forms (\ref{del}) or (\ref{vor}). The three following assumptions {\bf H1}-{\bf H3} are sufficient to define the conditional densities $f_\L$ in (\ref{density}).
\begin{itemize}
\item[{\bf H1}] $\Mdi \neq \emptyset$.
\end{itemize}
For every $\g$ in $\Mdi$ and every $\L$ in $\B(\Rd)$, $\tilde\g$ in $\Md$ is called a $(r,\L)$-modification of $\g$ (with $r>0$) if there exist distinct $y_1,y_2,\ldots,y_n$ in $\L$ satisfying $|y_i-x_i|<r$ for every $1\le i \le n$ (with $\g_\L=\{x_1,x_2,\ldots,x_n\}$) such that $\tilde\g=\{y_1,y_2,\ldots,y_n\} \cup \g_{\L^c}$.
\begin{itemize}
\item[{\bf H2}] $\Mdi$ is a locally open set in $\Md$ which means that for every $\g$ in $\Mdi$, every $\L$ in $\B(\Rd)$ there exists $r_\L(\g)>0$ such that any $(r_\L(\g),\L)$-modification of $\g$ is in $\Mdi$. 

\item[{\bf H3}] The interactions $V_1$ and $V_2$ are stable which means that there exists a constant $K>0$ such that 
$$ V_1\ge -K \qquad \text{and} \qquad V_2\ge -K. $$
\end{itemize}

Now let us give a collection of assumptions used in the proof of the existence of Gibbs Delaunay-Voronoi tessellations.
For every $R>0$, we denote by $\g_R$ an infinite configuration in $R$-equilateral position. That means that any triangle in $\del(\g_R)$ is equilateral with length of sides equal to $R$ (see the initial configuration presented in Figure \ref{model1} for an example).

\begin{itemize}
\item[{\bf H4}] There exist $K_1>0$ and $K_1'>0$ such that, for all  $\g \in \Mdi$ and all  $\L\in\B(\Rd)$,
$$\qquad  \text{Card}(\g_\L) \ge K_1 \vol(\L) - K_1'.$$

\item[{\bf H5}] There exist $R>0$, $0<r<R/2$ and $A\ge 0$ such that for every $\L\in\B(\Rd)$, every $(r,\L)$-modification $\tilde \g$ of $\g_R$,  and every $T,T' \in \del(\tilde\g)$ with $T\sim_{\del} T'$,
\begin{equation}\label{condition1}
 V_1(T)\le A \quad \text{ and }\quad V_2(T,T')\le A ;
\end{equation}
Respectively, for every $C,C' \in \vor(\tilde\g)$ with $C\sim_{\vor} C'$, 
\begin{equation}\label{condition1bis}
V_1(C)\le A \quad \text{ and }\quad V_2(C,C')\le A.
\end{equation}

\end{itemize}

Now we are able to give a first existence theorem
\begin{theoreme}\label{existence1}
There exists a stationary Gibbs Delaunay-Voronoi tessellation for any intensity $\nu=z\lambda$ ($z>0$) and any energy functions $(E_\L)_{\L\in\B(\Rd)}$ satisfying assumptions {\bf H1}, {\bf H2}, {\bf H3}, {\bf H4} and {\bf H5}.
\end{theoreme}
Assumptions {\bf H4} and {\bf H5} can be substituted by the following one.
\begin{itemize}
\item[{\bf H6}] There exists $A>0$ such that for every $r>0$ we can find $R>2r$ such that for every $\L\in\B(\Rd)$ and every $(r,\L)$-modification $\tilde \g$ of $\g_R$ the property (\ref{condition1}) or (\ref{condition1bis}) holds.
\end{itemize}

We have the second following existence theorem.
\begin{theoreme}\label{existence2}
There exists a stationary Gibbs Delaunay-Voronoi tessellation for any intensity $\nu=z\lambda$ ($z>0$) and any energy functions $(E_\L)_{\L\in\B(\Rd)}$ satisfying the Assumptions {\bf H1}, {\bf H2}, {\bf H3} and  {\bf H6}.
\end{theoreme}
The proofs of Theorems \ref{existence1} and \ref{existence2} can be found in \cite{DDG}. They rely on entropy tools which are only available in the setting of stationary processes (i.e. $\nu=z\l$). Concerning our three example models, the following corollary holds (the existence of Model 2 is also proved in \cite{D}).  

\begin{corollaire} In the stationary case, i.e. when the intensity measure $\nu$ is equal to $z\l$, Gibbs Delaunay-Voronoi tessellations for models 1, 2 and 3 exist.
\end{corollaire}

\begin{proof}
First of all, assumptions {\bf H1}, {\bf H2}, {\bf H3} are obviously satisfied for the three models. 
Concerning Model 1, the existence is given by Theorem \ref{existence2}. Assumption {\bf H6} is proved by taking $A=0$ and $R$ large enough with respect to $r$ such that any $(r,\L)$-modification $\tilde \g$ of $\g_R$ have Delaunay triangles with angles larger than $\alpha$.  
Concerning Models 2 and 3, the existence is given by Theorem \ref{existence1}. Assumption {\bf H4} comes from the hardcore interaction which forces the cells to be not too large. The uniform bound in {\bf H5} is obvious if $R$ and $r$ are chosen such that any $(r,\L)$-modification $\tilde \g$ of $\g_R$ is in $\Mdi$.
\end{proof}

\subsection{Convergence of the algorithm}\label{convergence-algo}

The Birth-Death-Move algorithm used in this paper is presented in \cite{MW} page 115 where the convergence is proved in Proposition 7.7 if the associated Markov Chain is irreducible and aperiodic (see also \cite{MT}). In our setting, there is no problem with the aperiodicity since the probability that nothing happens during one step of the algorithm is positive. In general to prove the irreducibility, it is sufficient to point out that every configuration $\g$ is linked by a finite number of algorithm steps to the empty configuration. In our case, it is not possible because there is a strong hardcore interaction and so the connection with the empty configuration is in general false. So we need the connectivity assumption {\bf H7} based on the following definition.
\begin{definition}\label{connected}
$\g$ and $\g'$ in $\Mdi$ are connected if there exist $n\ge 0$ and a sequence of configurations $\g_0, \g_1, \ldots ,\g_{n-1},\g_n$ in $\Mdi$ such that $\g_0=\g$,  $\g_n=\g'$ and, for each $0\le i \le n-1$, $\g_i$ and $\g_{i+1}$ differ only by one step of the algorithm (a birth, a death or a move).
\end{definition}
\begin{itemize}
\item[{\bf H7}] For every $\g$ and $\g'$ in $[0,1]^2$ such that $\bar \g$ and $\bar \g'$, defined in (\ref{periodise}), are in $\Mdi$, then $\g$ and $\g'$ are connected.
\end{itemize}
The deterministic connectivity assumption {\bf H7} and the flexibility assumption {\bf H2} on the space $\Mdi$ ensure that for all configurations $\g$, $\g'$ in $\Mdi$ and every $r>0$, the algorithm may generate from $\g$, with a positive probability and a finite number of steps, a $(r,[0,1]^2)$-modification of $\g'$ (see the definition after {\bf H1}). Irreducibility of the Markov chain follows and we have the following proposition.

\begin{proposition}\label{convergence}
Under the assumptions {\bf H2} and {\bf H7}, the Birth-Death-Move algorithm presented in Section 3.2 converges to $\bar P$ (see Definition \ref{Pbar}) in total variation norm for $\bar P$-a.s. every initial condition.
\end{proposition}  
It seems difficult in our context to obtain rates of convergence, because  the energy functions are not locally stable (the local stability means that $|E_\L(\g\cup\{x\})- E_\L(\g)|$ is uniformly bounded with respect to $\L$, $\g$ and $x$). Moreover the space $\Mdi$ may be very complicated since there is no upper bound  in general  for the number of steps $n$ in assumption {\bf H7}. 
  
Let us remark that assumption {\bf H7} is not easy to check. If {\bf H7} is not satisfied then the algorithm converges to the restriction of $\bar P$ on the connected component of the initial configuration $\g_0$ for the  connection relation defined below (see definition \ref{connected}). In this case, the limiting distribution may depend on the initial configuration. For Model 1, we can show that {\bf H7} is satisfied if $\alpha$ is small enough. We don't give the proof here but the scheme is essentially the same than in the following Proposition \ref{connection} which deals with Model 2. For Model 3, it is more complicated, we have not proved it but it seems satisfied if $\alpha$ and $B$ are large enough.

\begin{proposition} \label{connection}
In  Model 2, if $\; 2\varepsilon < \alpha < \frac{1}{2}$ then the algorithm presented in Sections 3.2 converges to $\bar P$.
\end{proposition}
\begin{proof}
According to Proposition \ref{convergence}, it suffices to show {\bf H2} and {\bf H7}. Since {\bf H2} is obviously satisfied for Model 2, it remains to show {\bf H7}.

Let $\g$ and $\g'$ be in $[0,1]^2$ such that $\bar \g$ and $\bar \g'$ are in $\Mdi$. To simplify the notations, we say that $\g$ is in $\Mdi$ if $\bar \g$ is in $\Mdi$. We start the sequence by putting $\g_0=\g$ and we construct the sequence $\g_i$ by an algorithmic procedure. 

In a first step (called saturation) we add points until there does not exist any ball with radius $\vp$ without points. More precisely, we test if there exists $x$ in $[0,1]^2$ such that $\bar \g_0 \cap B(x,\vp)=\emptyset$. If it is not the case, the saturation is finished. If it is the case we add the point $x$ by a birth-step action and we put $\g_1=\g_0+x$. Then, we test again if there exists $x$ in $[0,1]^2$ such that $\bar \g_1 \cap B(x,\vp)=\emptyset$. If it is not the case the saturation is finished otherwise we put $\g_2=\g_1+x$. We go on like this until the saturation procedure stops which is always the case since $[0,1]^2$ may contain only a finite number of points with a distance between them bigger than $\vp$. Let us remark that this construction produces configurations in $\Mdi$. We denote by $\g_{m_1}$ the saturated configuration of $\g$.

In a second step, we add the points of $\g'$ to $\g_{m_1}$ by the following way. Let $x$ be a point of $\g'$. By definition of the saturation, the configuration $\g_{m_1}+x$ is not in $\Mdi$ since there exists at least one point $y$ in $\g_{m_1}$ such that $|x-\bar y|\le \vp$ ($\bar y$ is the periodic version of $y$ such that $\bar y\in B(x,\vp)$). If this point $y$ is unique we use a move-step action to move $y$ to $x$. So we put $\g_{m_1+1}=(\g-y)+x$. If these points are non unique, they are removed (except one) by death-step actions and the last one is moved as above. We denote by $\g_{m_1+1},\ldots,\g_{m_2}$ this sequence and we remark that these configurations are in $\Mdi$ since $\g_{m_1}$ is saturated and $ 2\varepsilon < \alpha < \frac{1}{2}$. Now we saturate again the configuration $\g_{m_2}$ as above and we add another point of $\g'$ to $\g_{m_2}$. We go on until we have added all the points of $\g'$ and we denote by $\g_{m_3}$ the final configuration. 

It remains to remove the points of $\g_{m_3}$ which are not in $\g'$. It is sufficient to apply several death-step actions since the obtained configurations are in $\Mdi$. 

\end{proof}

\subsection{Consistency of the estimation procedure}\label{theorieestimation}

Let us suppose that the energy function, defined in (\ref{del}) and (\ref{vor}), is parameterized by $\beta$ and $\theta$, and is denoted by   $E_\L^{\beta, \theta}$. We first need to distinguish properly the hardcore parameter $\beta$ from the other parameter $\theta$. This is the purpose of the following assumption.\\ 

{\bf S1}: For all $\g\in\Md$, for all $\beta$ and for all $\theta$ and $\theta'$,
$$\forall \L\in\Bd,\qquad E_\L^{\beta, \theta}(\g_\L,\g_{\L^c})<\infty \iff E_\L^{\beta, \theta'}(\g_\L,\g_{\L^c})<\infty.$$

Under   {\bf S1}, the support of the energy is parameterized by  $\beta$ only, and not by $\theta$, which confirms that $\beta$ is the hardcore parameter. This assumption is satisfied by Models 2 and 3 with $\beta=(\epsilon, \alpha)$ and $\beta=(\epsilon, \alpha, B)$ respectively.

The strong consistency (under any stationary Gibbs measure) of $\hat \beta$ defined in (\ref{esthardcore}) and $(\hat z,\ \hat \theta)$ defined in (\ref{estsmooth}) are established in Theorem 2 in \cite{DL}, under some regularity assumptions. These assumptions have been checked for Model 2 in Proposition 5 in \cite{DL}.  Concerning  Model 3, the assumptions could be checked in the same way excepted for assumption {\bf S3} involved in  \cite{DL}. We have not succeeded to prove it but it seems true at least for $B$ large enough.

\section*{Acknowledgements}  
Simulations have been done on the computer servers supported by the research program MOSTAPAD-CPER 2007-2013, Project 7 "STIC et Calcul", of the F\'ed\'eration de Math\'ematiques des Pays de Loire.

\end{document}